\pdfoutput=1
\documentclass[11pt,a4paper, reqno]{amsart}
\usepackage{tabls}
\usepackage{url}
\usepackage[linktocpage = true]{hyperref}
\hypersetup{
 colorlinks = true,
 citecolor = blue,
}
\usepackage{color}
\definecolor{red}{rgb}{1,0,0}
\definecolor{green}{rgb}{0,1,0}
\definecolor{blue}{rgb}{0,0,1}
\definecolor{refkey}{gray}{.625}
\definecolor{labelkey}{gray}{.625}
\usepackage[lite]{amsrefs}
\usepackage[a4paper]{geometry}
\usepackage{amsfonts}
\usepackage{amssymb}
\usepackage{mathrsfs}
\usepackage{amsmath}
\usepackage{amsthm}
\usepackage{stmaryrd}
\usepackage{amscd}
\usepackage{enumerate}
\usepackage{paralist}
\usepackage{graphicx}
\usepackage{mathtools}
\usepackage{kantlipsum}
\usepackage{tikz-cd}
\allowdisplaybreaks

\newcommand{\abs}[1]{\lvert#1\rvert}

\newcommand{\id}{\operatorname{id}}
\newcommand{\R}{\mathbb{R}}
\newcommand{\C}{\mathbb{C}}
\newcommand{\Z}{\mathbb{Z}}
\newcommand{\N}{\mathbb{N}}

\makeatletter
 \def\@textbottom{\vskip \z@ \@plus 1pt}
 \let\@texttop\relax
  \def\title@font{\normalsize\bfseries}
 \let\ltx@maketitle\@maketitle
 \def\@maketitle{\bgroup
 \let\ltx@title\@title
 \def\@\title{\resizebox{\textwidth}{!}{
  \mbox{\title@font\ltx@title}
 }}
 \ltx@maketitle
 \egroup}
\makeatother

\theoremstyle{plain}
\newtheorem{thm}[equation]{Theorem}
\newtheorem{lem}[equation]{Lemma}
\newtheorem{cor}[equation]{Corollary}

\theoremstyle{definition}
\newtheorem{defn}[equation]{Definition}
\newtheorem{example}[equation]{Example}
\newtheorem{prop}[equation]{Proposition}
\newtheorem{Ex}[equation]{Example}
\newtheorem{Rem}[equation]{Remark}

\numberwithin{equation}{section}

\begin{document}
\def\CE{\mathrm{CE}}
\def\D{\mathcal{D}}
\def\E{\mathscr{E}}
\def\RR{\mathscr{R}}
\def\F{\mathscr{F}}
\def\H{\textbf{H}}
\def\k{\mathbb{K}}
\def\G{\mathcal{G}}
\def\M{\mathcal{M}}
\def\O{\mathcal{O}}
\def\img {\operatorname{img}}
\def\coker{\operatorname{coker}}
\def\End{\operatorname{End}}
\def\pr{\operatorname{pr}}
\def\Lie{\operatorname{Lie}}
\def\id{\operatorname{id}}
\def\Der{\operatorname{Der}}
\def\Hom{\operatorname{Hom}}
\def\bas{\operatorname{bas}}
\def\Map{\operatorname{Map}}
\def\Mod{\operatorname{Mod}}
\def\sgn{\operatorname{sgn}}
\def\sh{\operatorname{sh}}
\newcommand{\dst}{d_{E}}
\newcommand{\dr}{d_{\operatorname{reg}}}
\newcommand{\dnv}{d_{\operatorname{nv}}}
\newcommand{\Cr}{C_{\operatorname{reg}}}
\newcommand{\Cst}{C_{\operatorname{st}}}
\newcommand{\Cnv}{C_{\operatorname{nv}}}
\newcommand{\dstone}{d_{\CE}}
\newcommand{\dsttwo}{d_{T}}
\newcommand{\dCE}{d_{\operatorname{CE}}}
\newcommand{\Transgression}{d_T}
\newcommand{\Ker}{\operatorname{Ker}}
\newcommand{\fieldK}{\mathbb{K}}
\newcommand{\cinfM}{C^{\infty}(M)}
\newcommand{\coanchor}{\varrho}
\newcommand{\HH}{\mathcal{H}}
\newcommand{\rhothree}{\widetilde{\rho}}
\newcommand{\Cnablathree}{C_\nabla}

\title{The Standard cohomology of regular Courant algebroids}

\author{Xiongwei Cai}
\address{School of Mathematics and Statistics, Minnan Normal University, Zhangzhou, China}
\email{\href{mailto:cxw2620@mnnu.edu.cn}{cxw2620@mnnu.edu.cn}}

\author{Zhuo Chen}
\address{Department of Mathematics, Tsinghua University, Beijing, China}
\email{\href{mailto:chenzhuo@tsinghua.edu.cn}{chenzhuo@tsinghua.edu.cn}}

\author{Maosong Xiang}
\address{School of Mathematics and Statistics, Center for Mathematical Sciences, Huazhong University of Science and Technology, Wuhan, China}
\email{\href{mailto: msxiang@hust.edu.cn}{msxiang@hust.edu.cn}}
\thanks{Research partially supported by Natural Science Foundation of Fujian Province 2022J01894 (Cai), NSFC grant 12071241 (Chen) and NSFC grant 11901221 (Xiang).}

\begin{abstract}
For any regular Courant algebroid $E$ over a smooth manifold $M$ with characteristic distribution $F$ and ample Lie algebroid $A_E$, we prove that there exists a canonical homological vector field on the graded manifold $A_E[1] \oplus (TM/F)^\ast[2]$ such that the resulting dg manifold $\M_E$, which we call the minimal model of the Courant algebroid $E$, encodes all cohomological information of $E$. Indeed, the standard cohomology of $E$ can be identified with the cohomology of the function space on $\M_E$, which can be computed by a Hodge-to-de Rham type spectral sequence. We apply this result to generalized exact Courant algebroids and those arising from regular Lie algebroids.
\end{abstract}

\maketitle

\noindent  {\it Keywords:} \hspace*{0.6cm}
  Courant algebroid, ~symplectic graded manifold, ~standard cohomology, \\
  \hspace*{2.5cm}~~minimal model, semifull algebra contraction.\\
\noindent  {\it MSC(2020):}\hspace*{0.5cm}  primary 58A50; secondary 53D17, 16E45, 17B63,17B70, 18G40.

\tableofcontents

\section{Introduction}
This paper introduces a minimal model of regular Courant algebroids and exploits it in order to compute the standard cohomology of regular Courant algebroids.

The concept of Courant algebroids started with Courant and Weinstein's study on Dirac manifolds via the bracket  defined in \cites{C-W-1986, C1990}, now known as Courant bracket. Around the same time and developed independently,  the  bracket introduced by Dorfman~\cite{D1993} in the context of  Hamiltonian structures, and similarly the one used by Cabras and Vinogradov~\cite{CabrasVinogradov1992}, are equivalent versions of the Courant bracket.  Courant and Weinstein's earlier work eventually led to the definition of Courant algebroids by  Liu,  Weinstein and Xu in~\cite{L-W-X1997}.
Since then, Courant algebroids  have been a fertile source of inspirations in the study of higher structures~\cites{U2013,BR2018}, Poisson geometry~\cite{MR2507112},  generalized complex geometry~\cites{N.Hitchin2003, M.Gualtieri2011},  while also having some important applications to the study of $3$-dimensional topological field theory~\cites{Ikeda2001, Ikeda2003,Park2001,HP2004, Roy2007}, string theory~\cites{D2015, JV2016}, T-duality~\cites{S2015,SV2017}, and double field theory~\cites{DS2015,HZ2009}.  We refer to~\cite{MR3033556} for a thorough survey of Courant algebroids from their initial appearance to their applications and generalizations.

Roughly speaking, a Courant algebroid $(E,g,\rho,\circ)$ is a pseudo-Euclidean vector bundle $(E,g)$ over a smooth manifold $M$ together with a morphism of vector bundles $\rho\colon  E \to TM$, called \textit{anchor}, and a Leibniz algebra structure $\circ$, called \textit{Dorfman bracket}, on the section space $\Gamma(E)$, satisfying some compatibility conditions.
There is an alternative approach from differential graded (dg for short) geometry by \u{S}evera \cite{S2005} and Roytenberg \cite{Roytenberg2001}, which we now recall. Given a pseudo-Euclidean vector bundle $(E, g)$, the pair $(E[1], g)$ becomes a degree $(-2)$ Poisson manifold.
According to Rothstein~\cite{Rothstein}, by choosing  a metric connection $\nabla$ on $(E,g)$, the \textit{minimal symplectic realization} $\E$ of $(E[1], g)$ is identified with an exact degree $2$ symplectic manifold  $(T^\ast[2]M \oplus E[1], \omega_{g,\nabla})$ (see also~\cite{GMP}). The associated degree $(-2)$ Poisson algebra $(C^\infty(\E), \{-,-\})$ is known as the \textit{Rothstein algebra}. Here $\{-, -\}$ denotes the degree $(-2)$ Poisson bracket, which is called the big bracket in~\cite{YKS1992}.
A Courant algebroid structure on $(E, g)$ is encoded by a degree $3$ function $\Theta$ (known as the \textit{generating Hamiltonian}) on $\E$, subject to the classical master equation $\{\Theta, \Theta\} = 0$.  The anchor $\rho$ and the Dorfman bracket $\circ$ of $E$ are recovered from $\Theta$ via a Vinogradov bracket (or derived bracket) (see~\cites{AX, YKS, Roytenberg2001}).
This approach leads to the definition of the standard complex and cohomology of a Courant algebroid $E$ in~\cite{Roytenberg2001}:
The Hamiltonian vector field $X_\Theta := \{\Theta, -\}$ is a homological vector field on the graded manifold $\E$, and thus defines a differential $\dst :=X_{\Theta}$ on the algebra $C^\infty(\E)$ of functions on $\E$. The cochain complex $(C^\infty(\E),\dst)$ is called the \textit{standard complex} of $E$, and its cohomology $H_{\mathrm{st}}^\bullet(E)$ is called the \textit{standard cohomology} of the Courant algebroid $E$.

In a pure algebraic setting, Keller and Waldmann \cite{KW} introduced for any Courant algebroid a canonical dg Poisson algebra of degree $(-2)$, which is indeed isomorphic to the Rothstein dg algebra in the smooth setting (see also~\cite{CM2019}).
Keller-Waldmann's construction is closely related to the work of Roytenberg~\cite{Roy2009}, where another cochain complex is constructed for each Courant-Dorfman algebra, an algebraic version of Courant algebroid with fewer assumptions compared to the original definition in geometry.
Roytenberg showed that, in the smooth setting, this complex is also isomorphic to the standard complex; see~\cites{KW, Roy2009} for a further comparison between their approaches.

The mathematical significance of the standard cohomology of a Courant algebroid $E$ is reflected by the structural interpretations of some lower orders~\cite{Roytenberg2001} (see also~\cite{KW}).
Meanwhile, it is helpful for the study of the 3 dimensional Courant sigma model~\cite{Roy2007}:
The Batalin-Vilkovisky formalism~\cite{BV} provides a powerful tool for quantization of gauge theories. For example, the AKSZ's construction~\cite{AKSZ} of solutions of the classical master equation is a standard process to produce classical topological field theories from mathematical objects, and in return, many topological invariants are revealed from mathematical perspectives (cf.~\cite{GLL}).
Since cohomology of the target dg manifold naturally defines a set of classical observables in the associated AKSZ theory (cf.~\cites{BMZ,CQZ}), it follows that the standard cohomology of the Courant algebroid $E$ gives rise to a set of classical observables of the associated Courant sigma model.

Although the definition is simple and natural, the standard cohomology $H_{\mathrm{st}}^\bullet(E)$ of a Courant algebroid is very difficult to compute.
As an early attempt, Sti\'{e}non and Xu defined in~\cite{SX2008} the naive complex $(C_{\operatorname{naive}}^\bullet(E), \breve{d})$ of a Courant algebroid $E$, and they proved that  the corresponding degree $1$ naive cohomology $H_{\operatorname{naive}}^1(E)$ is isomorphic to the standard cohomology $H_{\operatorname{st}}^1(E)$. Later, Ginot and Gr\"{u}tzmann proved in~\cite{GM2009} that the naive cohomology of any transitive Courant algebroid is isomorphic to its standard cohomology, which was first conjectured by Sti\'{e}non and Xu in~\emph{op.cit.}.
However, for general Courant algebroids, the two cohomology theories are quite different.
Another achievement of \cite{GM2009} is the construction of a Leray type spectral sequence, called the naive ideal spectral sequence, which is used to calculate the standard cohomology $H_{\operatorname{st}}^\bullet (E)$ under the assumption that $E$ has a \textit{split base}.

We focus on regular Courant algebroids in this paper. A Courant algebroid $E$ is said to be \textit{regular} if its anchor $\rho$ has constant rank. In this case, $F:= \rho(E) \subseteq TM$ defines an integrable distribution, which we call the \textit{characteristic distribution} of $E$.
Note that a Courant algebroid $(E,g,\rho,\circ)$ with split base is  necessarily regular. The converse is not true in general.
For example, let $\mathbb{T}^2$ be the smooth $2$-torus and $F \subset T\mathbb{T}^2$ the rank $1$ integrable distribution generated by a constant vector $(1,\nu)\in \mathbb{R}^2$. Consider the regular Courant algebroid structure on $E:= F \oplus F^\ast$ induced by $F$. It is apparent that $E$ has split base if and only if $\nu$ is rational.

Given a regular Courant algebroid  $(E,g,\rho,\circ)$ over $M$ with characteristic distribution $F$,
we inherit some foundational facts and notations from the joint work~\cite{CSX} of the second author with Sti\'enon and Xu. A particularly important object is the \emph{ample Lie algebroid} $A_E$ (also called the associated Lie algebroid in~\cite{BR2018}), which shares the same characteristic distribution $F \subseteq TM$ with that of $E$.
To see the relation between the regular Courant algebroid $E$ and its ample Lie algebroid $A_E$, we consider the graded manifold $\RR:= F^\ast[2] \oplus E[1]$. It turns out that the standard differential $d_E$ on $E$ defines a differential on the algebra $C^\infty(\RR)$ of smooth functions on $\RR$ as well. We call $(C^\infty(\RR), d_E)$ the \emph{regular dg algebra} of $E$, which is indeed a dg subalgebra of the Rothstein dg algebra $(C^\infty(\E), d_E)$.
We prove that there exists a semifull algebra contraction from the regular dg algebra of $E$ onto the Chevalley-Eilenberg dg algebra of the ample Lie algebroid $A_E$ (see Theorem~\ref{contractionCA}).
Hence, the corresponding cohomology $H^\bullet(C^\infty(\RR), d_E)$, which we call the \textit{regular cohomology} of $E$, is indeed isomorphic to the Chevalley-Eilenberg cohomology $H^\bullet_{\CE}(A_E)$ of the ample Lie algebroid $A_E$,  which is also isomorphic to the naive cohomology of $E$ according to~\cite{CSX}.
Note that if $E$ is a transitive Courant algebroid, then its regular dg algebra coincides with its Rothstein dg algebra. As an immediate application, we recover a conjecture of Sti\'{e}non and Xu in~\cite{SX2008}, which was firstly proved by Ginot and Gr\"{u}tzmann in~\cite{GM2009}, on the isomorphism between the standard cohomology and the naive cohomology for transitive Courant algebroids.

To compute the standard cohomology of general regular Courant algebroids, we observe that the quotient bundle $B:=TM/F$ is canonically a Lie algebroid $A_E$-module. The Chevalley-Eilenberg differential $d_{\CE}$ of this $A_E$-module $B$ determines a homological vector field on the graded manifold $A_E[1] \oplus B^\ast[2]$. However, the dg manifold $(A_E[1] \oplus B^\ast[2], d_{\CE})$ contains only partial information of the Rothstein dg algebra $(C^\infty(\E), d_E)$. We prove in Proposition~\ref{Lem:Transgression} that for each regular Courant algebroid $E$ there exists a canonical cohomology class $[d_T] \in H^3_{\CE}(A_E; B^\ast[2])$ of total degree $1$, which characterizes whether the Courant algebroid structure around each leaf $L$ of $F$ is determined by its restriction $E\mid_{L}$ (See also Theorem~\ref{Thm: geometric meaning of dT}).
Each representative $d_T$ of this class, viewed as a map from $C^\bullet(A_E; S^{\diamond+1}(B[-2]))$ to $C^{\bullet+3}(A_E; S^\diamond(B[-2]))$, defines a perturbation of the homological vector field $d_{\CE}$, and thus gives rise to a new dg manifold $(A_E[1] \oplus B^\ast[2], d_{\CE} + d_T)$.
Since there exists a semifull algebra contraction of the Rothstein dg algebra $(C^\infty(\E), d_E)$ of $E$ onto the algebra $(C^\infty(A_E[1] \oplus B^\ast[2]), d_{\CE} + d_T)$ of functions (see Theorem~\ref{Thm: semifull contraction for Rothstein algebra}), this new dg manifold indeed encodes all cohomological data of $E$.
Note that the corresponding $L_\infty$-algebroid structure on the graded vector bundle $A_E \oplus B^\ast[1]$ over $M$ is minimal. We call this new dg manifold  \textbf{the minimal model} of the Courant algebroid $E$ and denote it by $\mathcal{M}_E$. Our construction of such a minimal model is indeed unique up to isomorphism (see Proposition~\ref{Cor: uniqueness of minimal model}). Moreover, it admits a $2$-shifted derived Poisson structure in the sense of~\cite{BCSX} (see Proposition ~\ref{prop: 2-shifted derived Poisson}).

As a consequence, the standard cohomology of $E$ is isomorphic to the cohomology of its minimal model $\M_E$. In order to compute the cohomology of $\M_E$, we construct a non-positive filtered bounded decreasing filtration on $C^\infty(\M_E)$. The associated Hodge-to-de Rham type spectral sequence, which we call the \textit{Chevalley-Eilenberg-to-standard spectral sequence}, converges to $H^\bullet(\M_E)$.

Then we consider applications of the minimal model of regular Courant algebroids to three particular types of regular Courant algebroids.
The first type consists of Courant algebroids with split base.
Secondly, we consider \textit{generalized exact Courant algebroids}, i.e., those fitting into the following exact sequence of vector bundles over $M$
\[
0 \to F^\ast \xrightarrow{\rho^\ast} E  \xrightarrow{\rho}  F \to 0,
\]
where $F \subseteq TM$ is the characteristic distribution of $E$, and $F^\ast$ is its dual bundle.  In this case, the ample Lie algebroid $A_E$ coincides with $F$. It is proved in~\cite{GM2009} that generalized exact Courant algebroids with characteristic distribution $F$ are classified up to isomorphisms by a degree $3$ cohomology class $[C]$, called \u{S}evera class, in $H^3_{\CE}(F)$. It turns out that the cohomology class $[d_T] \in H^3_{\CE}(F; B^\ast[2])$ is given by a canonical covariant derivative $\nabla [C]$ of the \u{S}evera class $[C]$ along normal vector fields $\Gamma(B)$. Hence, we obtain an explicit description of the standard cohomology of generalized exact Courant algebroids in Theorem~\ref{Thm:application to GEC}.
Finally, we study regular Courant algebroids arising from regular Lie algebroids.
Let $(A, \rho_A, [-,-]_A)$ be a regular Lie algebroid, i.e., a Lie algebroid whose characteristic distribution $F=\rho_A(A) \subseteq TM$ is of locally constant rank. We call the bundle of Lie algebras $K := \ker \rho_A$ the characteristic bundle of $A$.  It is well-known that the Whitney sum $E = A \oplus A^\ast$ admits a canonical regular Courant algebroid structure. In this case, the cohomology class $[d_T] \in H^3_{\CE}(A_E; B^\ast[2])$ is isomorphic to the cohomology class $[\omega] \in H^2_{\CE}(A; \Hom(B,K))$ introduced by Gracia-Saz and Mehta in~\cite{GsM} to characterize the obstruction to a local trivializability of $A$ around any leaf $L$ of $F$. As a result, we obtain an explicit description of the standard cohomology of $E=A \oplus A^\ast$, which is indeed isomorphic to the cohomology of the shifted cotangent bundle $T^\ast[2](A[1])$ of the dg manifold $(A[1],d_A)$ arising from the regular Lie algebroid $A$; see Theorem \ref{Thm:Regular Lie algebroid} and Remark \ref{Rem:Aone}.

We would like to point out works of others that are related to the present paper. In fact, many insights and clues for the present paper appeared in our previous work~\cite{CLX} on cohomology of hemistrict Lie 2-algebras, i.e. 2-term dg Leibniz algebras whose bracket is skew-symmetric up to homotopy.
Recently, Cueca and Mehta~\cite{CM2019} found that the differential of the aforementioned Keller-Waldmann construction satisfies a Cartan-type formula, and it thus provides a way to relate Courant algebroid cohomology to classical de Rham or Chevalley-Eilenberg cohomology theories. This also led them to a thorough study of characteristic classes of Courant algebroids.
Finally, we would like to mention Jotz Lean's work~\cite{Jotz0} on Dorfman connections and their applications to linear splittings of the standard Courant algebroids over vector bundles, which could be related to the present paper.
\paragraph{\textbf{Outline of the paper}}
This paper is organized as follows:
In the preliminary Section~\ref{Sec:BasicStuff}, we establish some global conventions for our notations and recall the definition of standard cohomology of Courant algebroids.
In Section \ref{Sec:1regularcomplex}, we define the regular dg algebra of a regular Courant algebroid, and construct a semifull algebra contraction from the regular dg algebra onto the Chevalley-Eilenberg dg algebra of the ample Lie algebroid.
In Section \ref{Sec:2results}, we first prove the existence of a minimal model for each regular Courant algebroid by extending the semifull algebra contraction for the regular dg algebra to a semifull algebra contraction for the whole Rothstein dg algebra.
As an immediate application, we prove that the minimal model $\M_E$ of $E$ indeed admits a $2$-shifted derived Poisson manifold structure.
Then we compute the standard cohomology of a regular Courant algebroid $E$ by analysing the Chevalley-Eilenberg-to-standard spectral sequence for the minimal model $\M_E$ of $E$.
In Section \ref{Sec:Applications}, we consider applications to three types of regular Courant algebroids.

\paragraph{\bf Acknowledgement}
We would like to thank Jiahao Cheng, Gregory Ginot, Melchior Gr\"{u}tzmann, Luca Vitagliano, and Ping Xu for helpful discussions and comments.
We are also grateful to the anonymous referee for constructive suggestions to improve the presentation of the manuscript.

\section{Preliminaries}\label{Sec:BasicStuff}
In this section, we collect some standard facts of Courant algebroids and their  formulation via degree $2$ symplectic graded manifolds with generating Hamiltonian functions \cites{Roytenberg2001,Roy2007}. Our presentation mainly follows~\cites{GMP}.

\subsection{Symplectic graded manifolds of degree 2}
An $\N$-graded manifold $\mathcal{M}$ is a connected\footnote{In this paper, all smooth manifolds are assumed to be connected.} smooth manifold $M$, called the base of $\mathcal{M}$, equipped with a sheaf $\O$ of $\N$-graded $C^\infty(M)$-algebras such that, for every contractible open set $U \subset M$, the algebra of functions $\O(U)$ is isomorphic to $C^\infty(U) \otimes \widehat{S}(W)$ for some fixed vector space $W = \oplus_{i \geqslant  1} W^i$ which is finite dimensional and positively graded. Here $\widehat{S}(W)   $ is the complete graded symmetric tensor product of $W$ and the grading on $\O(U)$ is induced by the one on $W$. In particular, the degree $0$ component of $\O(U)$ is $C^\infty(U)$. The algebra of global sections of $\O$  is denoted by $C^\infty(\M)$.

Let us  consider ($\N$)-graded manifolds stemming from ordinary ungraded vector bundles. Given a smooth vector bundle $V \to M$   and a non-negative integer $k$, we denote by $V[k]$ the graded manifold   whose base   is $M$ and algebra of global functions is
\[
 C^\infty(V[k]) = \Gamma(\widehat{S}   (V^\ast[-k]))
 \cong \begin{cases}
                                         \Gamma(\oplus_{n\geqslant  0}(\wedge^n V^\ast) ) , & \mbox{if $k$ is odd}, \\
                                         \Gamma(\Pi_{n\geqslant  0}(S^nV^\ast) ), & \mbox{if $k$ is even}.
                                         \end{cases}
\]
Meanwhile, $V[k]$ is also treated as a graded vector bundle over $M$ concentrated in   degree $(-k)$. For $e \in  V$ or $\Gamma(V)$,  the corresponding element in $V[k]$ or $\Gamma(V[k])$   will be denoted by $e[k]$. In the sequel, we will simply  identify $\wedge^n V^\ast$ with $S^n (V^\ast[-1])$ and hence $C^\infty(V[1])=\Gamma(\wedge^\bullet V^\ast)$.  In particular, $\wedge^1 V^\ast$ and $V^\ast[-1]$ are one and the same.  An element in $\wedge^n V^\ast$ will be denoted indicatively by $\xi_1[-1]\odot \cdots \odot \xi_n[-1]$ or $\xi_1\wedge\cdots \wedge\xi_n$ for all $\xi_i \in V^\ast$, if there is no risk of confusion of degrees.

A symplectic graded manifold of degree $n$ consists of a graded manifold $\M$ and a closed non-degenerate $2$-form $\omega$ of degree $n$ on $\M$. Then the algebra $C^\infty(\M)$ of smooth functions on $\M$ admits a degree $(-n)$ Poisson bracket which we denote by
\[
 \{-,-\} \colon C^\infty(\M) \times C^\infty(\M) \to C^\infty(\M)[-n],
\]
and it is subject to  the   graded skew-symmetric condition
\begin{align*}
  \{F,G\} &= -(-1)^{(\abs{F}-n)(\abs{G}-n)}\{G,F\},
  \end{align*}
  and the graded Jacobi identity
  \[
   \{F,\{G,H\}\}  = \{\{F,G\},H\} + (-1)^{(\abs{F}-n)(\abs{G}-n)}\{G,\{F,H\}\},
\]
for all homogeneous $F, G, H \in C^\infty(\M)$.

In what follows, we recall a construction of symplectic graded manifolds of degree $2$ due to Rothstein \cite{Rothstein}, Roytenberg \cite{Roytenberg2001}, and Gr\"{u}tzmann-Michel-Xu \cite{GMP}.
In~\cite{Rothstein}, Rothstein described a class of symplectic supermanifolds in terms of a pseudo-Euclidean vector bundle $(E,g)$ over a symplectic manifold together with a metric connection $\nabla$ on $(E,g)$ (i.e., $\nabla g = 0$).
In~\cite{GMP},  Rothstein's construction is adapted to obtain a symplectic structure of degree $2$ on the graded manifold given by the Whitney sum $T^\ast[2] M \oplus E[1]$ over $M$.
\begin{prop}[\cites{Rothstein, GMP}]\label{Prop:Rothstein algebra}
  Let $\pi \colon E \to M$ be a smooth vector bundle  equipped with a pseudo-Euclidean metric $g$ and a metric connection $\nabla$. Then the graded manifold $\E=T^\ast[2] M \oplus E[1]$ admits a degree $2$ symplectic $2$-form
\[
 \omega_{g,\nabla}:= d(\pi_1^\ast \alpha + \pi_2^\ast \beta),
\]
where
\begin{enumerate}
	\item  $\pi_1$ and $\pi_2$ are, respectively, the canonical projections from $\E$ onto $T^\ast[2] M$ and $E[1]$;
	\item $\alpha \in \Omega^1(T^\ast[2]M)$ is the degree $2$ Liouville $1$-form on $T^\ast[2]M$; and
	\item $\beta \in \Omega^1(E[1])$   annihilates the horizontal subspace of $\Gamma(TE[1])$ lifted from $\Gamma(TM)$ via the chosen metric connection $\nabla$, and satisfies $\beta_e(v ) = \frac{1}{2}g_{\pi(e)}(v ,e)$ for all vertical tangent vectors $v  \in T_e E[1]$.
\end{enumerate}
\end{prop}
The algebra $C^\infty(\E)$ is generated by elements in $C^\infty(M)$, $\Gamma(E^\ast [-1])$, and $\Gamma(T[-2]M)$. In the sequel, we make repeated use of the identification of $\Gamma(E^\ast[-1])$ and $\Gamma(E[-1])$ induced by $g$.
The degree $(-2)$ Poisson algebra structure $\{-,-\}$ on $C^\infty(\E)$ that is induced from the symplectic structure $\omega_{g,\nabla}$ in Proposition~\ref{Prop:Rothstein algebra} is determined by the following relations:
\begin{align}\label{GlobalPoisson}
  \{X[-2],f\} &= X(f), & \{X[-2],e_1[-1]\} &= (\nabla_X e_1)[-1], \\
  \{X[-2],Y[-2]\} &= ([X,Y])[-2] + (R^\nabla(X,Y))[-2], & \{e_1[-1],e_2[-1]\} &= g(e_1,e_2), \label{Eq:Poisson bracket 2}\\
 \{f,h\} &= 0, & \{f,e_1[-1]\} &= 0, \label{Poisson bracket on E[1]}
\end{align}
for all $X ,Y  \in \Gamma(T M)$, $e_1 , e_2 \in \Gamma(E )$, and   $f,h \in C^\infty(M)$. Here $R^\nabla$ is the curvature of $\nabla$. Since $\nabla$ is a metric connection, its curvature $R^\nabla$ indeed defines a degree $2$ element $R^\nabla(X,Y) \in \Gamma(\wedge^2 E)$ by
\[
 R^\nabla(X,Y)(e_1,e_2) := g\big(R^\nabla(X,Y)e_1, e_2\big).
\]

According to Equations~\eqref{Eq:Poisson bracket 2} and ~\eqref{Poisson bracket on E[1]}, the graded manifold $E[1]$ is a degree $2$ Poisson submanifold of $\E$ (see~\cite{Jotz1} for a discussion on Poisson graded manifolds of degree $2$). Moreover, the degree $2$ symplectic graded manifolds $(\E, \omega_{g,\nabla})$ arising from different metric connections are all isomorphic, and each of them provides a minimal symplectic realization of the Poisson manifold $E[1]$ (see~\cite{GMP} for details).

\subsection{The standard and naive complexes of Courant algebroids}
For definition of Courant algebroids, we follow  the convention of~\cites{CSX,GM2009,GMP}.
\begin{defn}\label{Def of Courant algebroid}
	A \textbf{pre-Courant algebroid} is a vector bundle $E \to M$ equipped with a fiberwise nondegenerate pseudo-metric $g$, a bundle map $\rho\colon E \to TM$ called \emph{anchor}, and an $\R$-bilinear operation $\circ$ on $\Gamma(E)$ called \emph{Dorfman bracket}. These structure maps are subject to the following axioms:
\begin{align}
	e_1 \circ (fe_2) &= (\rho(e_1)f) e_2 + f e_1 \circ e_2, \label{Eq:anchor}\\
	\rho(e_1)g(e_2, e_3) &=  g(e_1 \circ e_2, e_3) + g(e_2, e_1 \circ e_3), \notag \\
e_1 \circ e_1 &=  \partial\left(g(e_1,e_1)\right)   , \label{Eq:selfbracket}
\end{align}
for all $e_1,e_2,e_3 \in \Gamma(E)$ and $f \in C^\infty(M)$. Here $\partial\colon C^\infty(M)\to \Gamma(E)$ in~\eqref{Eq:selfbracket} is defined by
\begin{equation}\label{Eq: Def of partial}
\partial f:= \frac{1}{2}\rho^\ast df,\qquad \forall f\in C^\infty(M),
\end{equation}
where $d$ stands for the de Rham differential $C^\infty(M)\to \Omega^1(M)$ and $\rho^\ast \colon T^\ast M \to E^\ast \cong E$ is the dual of the anchor $\rho$. As before, the identification $E^\ast \cong E$ is induced by the pseudo-metric $g$.

Moreover, if the Dorfman bracket $\circ$ satisfies the Leibniz-Jacobi identity
\begin{equation}\label{Eq:Jacobi}
  	e_1 \circ (e_2 \circ e_3) = (e_1 \circ e_2) \circ e_3 + e_2 \circ (e_1 \circ e_3),
\end{equation}
then we call $(E,g,\rho,\circ)$ a \textbf{Courant algebroid}.
\end{defn}
Some people prefer to use the \emph{Courant bracket} which is the skew-symmetric part of the Dorfman bracket:
\[
 \llbracket e_1, e_2 \rrbracket := \frac{1}{2}(e_1 \circ e_2 - e_2 \circ e_1).
\]
 Thus, we have
\[
 e_1 \circ e_2 = \llbracket e_1, e_2 \rrbracket + \partial g(e_1,e_2).
\]

It is observed by Roytenberg in~\cite{Roytenberg2001} that a pre-Courant algebroid structure underlying a pseudo-Euclidean vector bundle $(E,g)$ is indeed encoded by a degree $3$ smooth function $\Theta$  on the minimal symplectic realization $ (\E, \omega_{g,\nabla})$ of the Poisson manifold $E[1]$.
More precisely, any smooth function $\Theta$ of degree $3$ on $\E$ induces a pre-Courant algebroid structure on the pseudo-Euclidean vector bundle $(E,g)$ through the following relations
\begin{align}\label{Eq:derived bracket 1}
\rho(e_1) f &= \{\{\Theta, e_1[-1]\},f\},\\
(e_1 \circ e_2)[-1] &= \{\{\Theta,e_1[-1]\},e_2[-1]\}. \label{Eq:derived bracket 2}
\end{align}
Moreover, if $\Theta$ is subject to the classical master equation $\{\Theta, \Theta\} = 0$, then the bracket $\circ$ satisfies the Leibniz-Jacobi identity~\eqref{Eq:Jacobi}, and thus $(E,g,\rho,\circ)$ is a Courant algebroid.

Conversely, to find the degree $3$ smooth function $\Theta$ from a pre-Courant algebroid $(E,g,\rho,\circ)$, we need the following steps. Firstly, the anchor map $\rho\colon E\to TM$ corresponds to an element $\rhothree$ in $\Gamma(E^\ast[-1] \otimes T[-2]M)$ after a degree shifting, and thus a degree $3$ function in $C^\infty(\E) $. Secondly, from $g$, $\circ$, and the metric connection $\nabla$ on $(E,g)$, we obtain another degree $3$ function $C_\nabla \in \Gamma(\wedge^3 E^\ast)$ in $C^\infty(\E)$, called the torsion\footnote{This general torsion  $C_\nabla$ is related to a particular torsion $	C_{\nabla^E}$  introduced in \cite{AX}. See Remark \ref{Rem:3charclass}.} of $\nabla$,
defined by
\begin{equation}\label{Eq:Def of torsion}
  C_\nabla(e_1,e_2,e_3) = \frac{1}{2} \operatorname{cycl}_{123}g\left(\frac{1}{3}(e_1 \circ e_2 - e_2 \circ e_1) - (\nabla_{\rho(e_1)}e_2 - \nabla_{\rho(e_2)}e_1), e_3 \right),
\end{equation}
 for all $e_1,e_2,e_3 \in \Gamma(E)$, where $\operatorname{cycl}_{123}$ denotes the sum over cyclic permutations.
 Finally, by setting the degree $3$ element
\[
 \Theta:= \rhothree +\Cnablathree\in C^\infty(\E),
\]
one can examine that $\Theta$ satisfies Equations~\eqref{Eq:derived bracket 1} and~\eqref{Eq:derived bracket 2}.
Moreover, if $\circ$ satisfies the Leibniz-Jacobi identity~\eqref{Eq:Jacobi}, then $\Theta$ solves the classical master equation $\{\Theta, \Theta\} = 0$.

In summary,  Courant algebroid structures on a pseudo-Euclidean vector bundle $(E,g)$ are in one-to-one correspondence with degree $3$ functions $\Theta\in C^\infty(\E)$, which we call \emph{generating Hamiltonian functions or potentials}, satisfying the classical master equation.
The triple $(\E, \{-,-\}, \Theta)$ is called a Hamiltonian $2$-algebroid in~\cite{Roytenberg2001} and a gauge system in~\cite{LS}.
The Hamiltonian vector field $X_{\Theta}:= \{\Theta, -\}$ is a homological vector field, i.e., a degree $1$ vector field on the graded manifold $\E$ such that $[X_{\Theta}, X_{\Theta}] = 0$. Thus, the pair $(\E, X_{\Theta})$ is a dg manifold or a Q-manifold~\cite{AKSZ}.
The associated commutative dg algebra $(C^\infty(\E), \{\Theta, -\})$ is called the \textbf{Rothstein dg algebra}; the underlying cochain complex
\[
 \big(\Cst^\bullet (E) := C^\infty(\E), \dst := \{\Theta, -\}\big),
\]
which was first introduced in \cite{Roytenberg2001}, is called the \textbf{standard complex} of the Courant algebroid $E$. The corresponding cohomology $H_{\mathrm{st}}^\bullet(E)$ is called the \textbf{standard cohomology} of $E$.

The standard complex $(\Cst^\bullet(E), \dst)$ admits a subcomplex $(\Gamma(\wedge^\bullet \ker \rho), \breve{d})$ which is called the \textbf{naive complex} of $E$ by Sti\'{e}non and Xu~\cite{SX2008}.
Note that the pseudo-metric $g$ on $E$ extends to a pseudo-metric, also denoted by $g$, on $\wedge^\bullet E$.
Hence, any element
\[
\eta \in \Gamma(\wedge^k \ker \rho) \subset \Gamma(\wedge^k E) \cong \Gamma(\wedge^k E^\ast) \subset C^\infty(\E )
\]
can be naturally viewed as a degree $k$ function on the graded manifold $\E $.
The differential $\breve{d}\colon \Gamma(\wedge^\bullet \ker \rho)\to \Gamma(\wedge^{ \bullet+1} \ker \rho)$ is defined by
\begin{align*}
 & g(\breve{d}\eta, e_0\wedge \cdots \wedge e_{k}) \\
&=\sum_{i=0}^{k}(-1)^{i}\rho(e_i)g(\eta, e_0\wedge\cdots  \wedge \widehat{e_i} \wedge \cdots \wedge e_{k})\\
&\quad +\sum_{i<j}(-1)^{i+j}g(\eta, \llbracket e_i, e_j \rrbracket \wedge e_0 \wedge\cdots \wedge \widehat{e_i} \wedge\cdots \wedge \widehat{e_j}\wedge\cdots \wedge e_{k}),
\end{align*}
for all $\eta \in \Gamma(\wedge^k \ker \rho)$ and   $e_0,\cdots, e_k \in \Gamma(E)$.
The cohomology of this cochain complex is called the \textbf{naive cohomology} of $E$ and is denoted by $H^\bullet_{\operatorname{naive}}(E)$. It is proved by Sti\'{e}non and Xu in \emph{op.cit.} that this complex $(\wedge^\bullet \ker \rho, \breve{d})$ is indeed a subcomplex of the standard one $(\Cst^\bullet(E), \dst)$.

\section{A contraction theorem}\label{Sec:1regularcomplex}
Let $(E,g,\rho,\circ)$ be a regular Courant algebroid, i.e., the subbundle $F := \rho(E) \subseteq TM$ is of locally constant rank, and thus gives rise to an integrable distribution in $M$ which we call the \textbf{characteristic distribution} of $E$.
In this section, we introduce a dg subalgebra of the Rothstein dg algebra of $E$, which we call the regular dg algebra.
The main result of this section is Theorem~\ref{contractionCA} on the existence of a homotopy contraction from the underlying complex of the regular dg algebra to the naive complex of $E$.
\subsection{The regular dg algebra and the ample Lie algebroid}\label{Sec:regularcomplex}
\subsubsection{The regular dg algebra}
Since the characteristic distribution $F$ is regular, one has a short exact sequence of vector bundles over $M$
\begin{equation}\label{Eq:SES of VB}
  0 \to F \xrightarrow{\;i\;} TM \xrightarrow{\;\pr_B\;} B \to 0,
  \end{equation}
where we denote by $B$ the normal bundle $ TM/F$.
Choose a metric connection $\nabla$ on the pseudo-Euclidean vector bundle $(E,g)$. Then by Proposition~\ref{Prop:Rothstein algebra}, we obtain an exact symplectic graded manifold $(\E = T^\ast[2]M \oplus E[1] , \omega_{g,\nabla})$ of degree $2$.
Consider the graded manifold defined by
 \[
 \RR:= F^\ast[2] \oplus E[1].
 \]
It is clear that the inclusion $i \colon F \hookrightarrow TM$ induces a morphism of graded manifolds
\[
   \Phi  = (i^\ast \oplus \id_{E[1]}) \colon~ \E = T^\ast[2]M \oplus E[1] \to  \RR =F^\ast[2] \oplus E[1].
\]

Although the graded manifold $\RR $ is not symplectic in general, it admits a  degree $(-2)$ Poisson structure. In fact, according to Equations~\eqref{GlobalPoisson}, \eqref{Eq:Poisson   bracket 2}, and~\eqref{Poisson bracket on E[1]}, the degree $(-2)$ Poisson bracket on $C^\infty(\E)$ induces a Poisson bracket on $C^\infty(\RR )$ via the morphism of graded commutative algebras $\Phi^\ast \colon C^\infty(\RR ) \to C^\infty(\E)$ such that $(C^\infty(\RR ), \{-,-\})$ is a Poisson subalgebra of $(C^\infty(\E), \{-,-\})$.
Note further that the generating Hamiltonian function $\Theta = \rhothree +\Cnablathree$ of the regular Courant algebroid $E$ indeed belongs to the subalgebra $C^\infty(\RR )$ of $C^\infty(\E)$. Thus, the Hamiltonian vector field $X_{\Theta}=\{\Theta, -\}=d_E$ restricts to a homological vector field on the graded manifold $\RR$. Hence, $\Phi^\ast \colon (C^\infty(\RR), d_E) \to (C^\infty(\E), d_E)$ is a morphism of commutative dg algebras, and $\Phi\colon \E\to \RR$ is a morphism of dg manifolds.

\begin{defn}
	We call the commutative dg algebra $(C^\infty(\RR), d_E)$ the  \textbf{regular dg algebra}  of the regular Courant algebroid $E$, and the corresponding cohomology group $H^\bullet_{\operatorname{reg}} (E)$ the \textbf{regular cohomology} of $E$.
\end{defn}

\subsubsection{The ample Lie algebroid}
Let $(E,g,\rho,\circ)$ be a Courant algebroid. The subbundle $(\ker \rho)^\perp$ of $E$ orthogonal to $\ker\rho$ with respect to $g$ coincides with $\rho^\ast(T^\ast M) \subset E^\ast \cong E$, i.e., the subbundle of $E$ generated by the image of $\partial$ defined  by Equation~\eqref{Eq: Def of partial}.
According to Uchino~\cite{Uchino},  the Leibniz rule~\eqref{Eq:anchor} and the Leibniz-Jacobi identity~\eqref{Eq:Jacobi} imply the following identity
\[
\rho(e_1 \circ e_2) = [\rho(e_1), \rho(e_2)]_{TM},\quad \forall e_1, e_2 \in \Gamma(E).
\]
Thus, by Equation~\eqref{Eq:selfbracket}, we have
\[
\rho(\partial g(e_1,e_1)) = \rho(e_1 \circ e_1) = [\rho(e_1),\rho(e_1)]_{TM} = 0.
\]
Therefore, the kernel of the anchor is coisotropic, i.e.,
\[
 (\ker\rho)^\perp \subseteq \ker\rho.
\]
It is also easy to see that both $\Gamma(\ker\rho)$ and $\Gamma((\ker\rho)^\perp)$ are two-sided ideals in  $\Gamma(E)$ with respect to the Dorfman bracket $\circ$.

As we have assumed  that $E$ is  {regular}, both $\ker\rho$ and $(\ker\rho)^\perp$ are
smooth (constant rank) subbundles of $E$. Moreover, the quotient $E/(\ker\rho)^\perp$ is a Lie algebroid, called the \textbf{ample Lie algebroid} associated to $E$ \cite{CSX}, and will be denoted by $A_E$.
Note that $A_E^\ast=(E/(\ker\rho)^\perp)^\ast\cong \ker \rho$. Thus, the space $C^\bullet(A_E):=\Gamma(\wedge^\bullet A_E^\ast)$ is identified with $\Gamma (\wedge^\bullet \ker \rho)$. Moreover, via this identification,  the Chevalley-Eilenberg differential $\dCE$ of the ample Lie algebroid $A_E$ coincides with the naive differential $\breve{d}$.
Hence, the naive complex $(C_{\operatorname{naive}}^\bullet (E), \breve{d})$ of a regular Courant  algebroid $E$ is isomorphic to the Chevalley-Eilenberg complex $(C^\bullet(A_E),\dCE)$ of its ample Lie algebroid $A_E$.  For more details, see~\cite{CSX}.

\subsection{The regular-to-naive contraction}
In this section, we establish a contraction from the regular complex $(C^\bullet_{\operatorname{reg}}(E) = C^\infty(\RR), d_E)$ of a regular Courant algebroid $E$ onto the Chevalley-Eilenberg complex $(C^\bullet(A_E),\dCE)$ of the ample Lie algebroid $A_E$, or equivalently onto the naive complex $(C_{\operatorname{naive}}^\bullet (E), \breve{d})$ of $E$. For this, we  need to choose a \textit{dissection}  of $E$, which was introduced by the second author with Sti\'{e}non  and Xu in~\cite{CSX}.

\subsubsection{Dissections of regular Courant algebroids}
Denote by $\G$ the quotient bundle $\ker\rho/(\ker\rho)^\perp$. We use the symbol $\pi$ to denote the projection map $\ker\rho \rightarrow \G$. The Dorfman bracket $\circ$ of $E$ induces a $C^\infty(M)$-bilinear operation on $\Gamma(\G)$:
\[
[\pi(\mathbf{r}),\pi(\mathbf{s})]^\G := \pi(\mathbf{r} \circ \mathbf{s}),\quad \forall \mathbf{r}, \mathbf{s} \in \ker\rho.
\]
It turns out that $[-,-]^\G$  is a Lie bracket and hence $(\G, [-,-]^\G)$ is a bundle of Lie algebras. Moreover, the map
\[
\pi(\mathbf{r}) \otimes \pi(\mathbf{s}) \mapsto g(\mathbf{r}, \mathbf{s})
\]
is a well-defined nondegenerate symmetric and ad-invariant pseudo-metric on $\G$, which we will denote by the symbol ${g^{\G}}$. Hence, the triple $(\G, [-,-]^\G, {g^{\G}})$ is a bundle of quadratic Lie algebras.
The classical Cartan $3$-form ${C^\G}$ defined by
\[
 {C^\G}(r,s,t) := {g^{\G}}([r,s]^\G, t)
\]
for all $r,s,t \in \Gamma(\G)$, is treated as an element in $\Gamma(\wedge^3 \G)$.

Consider the vector bundle $F^\ast \oplus \G \oplus F$ equipped with the pseudo-metric $\langle -,- \rangle$ defined by
\[
\langle \xi+r+x,\eta+s+y\rangle  = \frac{1}{2}\langle \xi \mid y \rangle+ \frac{1}{2} \langle \eta \mid x\rangle + {g^{\G}}(r,s),
\]
for all $\xi, \eta \in \Gamma(F^\ast), r, s \in \Gamma(\G)$ and $x,y\in \Gamma(F)$. Here $\langle -\mid - \rangle$ is the standard pairing between $F^\ast$ and $F$.
\begin{defn}\label{Def:dissection}
A \textbf{dissection} of the regular Courant algebroid $(E,g,\rho, \circ)$  is an isomorphism of pseudo-Euclidean vector bundles
\[
\Psi \colon (F^\ast \oplus \G \oplus F, \langle -,- \rangle ) \rightarrow (E,g),
\]
i.e.,
\[
g(\Psi(\xi+r+x), \Psi(\eta+s+y)) = \langle \xi+r+x, \eta+s+y \rangle
\]
for all $\xi, \eta \in \Gamma(F^\ast), r, s\in\Gamma(\G)$ and $x,y\in \Gamma(F)$.
\end{defn}Lemma 1.2 of \cite{CSX} shows that dissections always exist.
Using a dissection $\Psi$ of $E$, we identify $E$ with $F^\ast \oplus \G \oplus F$, and obtain the following key maps:
\begin{align}\label{Eq:F-conn on G}
  \nabla^\G \colon &\Gamma(F) \otimes \Gamma(\G) \to \Gamma(\G), & \nabla^\G_x  r_1 &= \pr_\G(x \circ r_1); \\
\label{Eq:FFtoG}
  R \colon &\Gamma(F) \otimes \Gamma(F) \to \Gamma(\G),& R(x,y) &= \pr_\G(x \circ y);  \\
\label{Eq:FFtoFast}
  H \colon &\Gamma(F) \otimes \Gamma(F) \otimes \Gamma(F) \to C^\infty(M), & H(x,y,z) &= \langle \pr_{F^\ast}(x \circ y) \mid z \rangle;
\end{align}
for all $x,y,z \in \Gamma(F), r_1,r_2 \in \Gamma(\G)$, where $\pr_\G$ and $\pr_{F^\ast}$ are projections onto $\G$ and $F^\ast$, respectively. The maps $\nabla^\G$ and $R$ induce another two maps as follows:
\begin{align*}
  P \colon &\Gamma(\G) \otimes \Gamma(\G) \to \Gamma(F^\ast),  & \langle P(r_1,r_2) \mid y \rangle &= 2 {g^{\G}}(r_2, \nabla_y^\G r_1),\\
 Q \colon &\Gamma(F) \otimes \Gamma(\G) \to \Gamma(F^\ast), & \langle Q(x,r_1) \mid y \rangle &= {g^{\G}}(r_1, R(x,y)).
\end{align*}
Moreover, we have
\begin{prop}[\cite{CSX}*{Lemma 2.1}]\label{Prop:CSX 2.1}
Each dissection $\Psi$ transfers the  Courant algebroid structure of $E$ to a \textit{standard Courant  algebroid} structure on the pseudo-Euclidean vector bundle $(F^\ast \oplus \G \oplus F, \langle -, -\rangle)$, whose anchor map is the projection $\pr_F$ onto $F$, and whose Dorfman bracket $\circ$ is defined as follows:
\begin{align*}
x_1 \circ x_2 &= [x_1,x_2] + R(x_1,x_2) + H(x_1,x_2,-), & r_1 \circ r_2 &= [r_1,r_2]^\G + P(r_1,r_2), \\
\xi_1 \circ \xi_2 &= \xi_1 \circ r_2 = r_1 \circ \xi_2 = 0, & x_1 \circ \xi_2 &= L_{x_1} \xi_2, \\
\xi_1 \circ x_2 &= -L_{x_2}\xi_1 + d_F \langle \xi_1 \mid x_2 \rangle, & x_1 \circ r_2 &= -r_2 \circ x_1 = \nabla_{x_1}^\G r_2 - 2Q(x_1,r_2),
\end{align*}
for all $x_1,x_2 \in \Gamma(F), r_1,r_2 \in \Gamma(\G), \xi_1,\xi_2 \in \Gamma(F^\ast)$. Here $d_F\colon C^\infty(M)\to \Gamma(F^*)$ denotes the leafwise de Rham differential. We denote this standard Courant algebroid by $S(\nabla^\G, R, H)$.

Furthermore, the underlying vector bundle of the ample Lie algebroid $A_E$ is naturally identified with $F \oplus \G$, the anchor map is the projection $\pr_F$, and the bracket $[-,-]_{A_E}$ on $\Gamma(A_E) \cong \Gamma(F \oplus \G)$ is given by
\begin{align*}
  [x_1,x_2]_{A_E} &= [x_1, x_2] + R(x_1, x_2), & [x_1, r_1]_{A_E} &= \nabla^\G_{x_1} r_1, &\mbox{and }~ [r_1,r_2]_{A_E} &= [r_1,r_2]^\G.
\end{align*}
We denote this ample Lie algebroid by $L(\nabla^\G, R)$.
\end{prop}

\begin{prop}[\cites{CSX, MR3661534}]\label{Prop: different dissection}
  Suppose that $\hat{\Psi}$ is another dissection of $E$ and that $(\hat{\nabla}^\G, \hat{R}, \hat{H})$ are the associated linear maps. Then the change of dissection $\delta = \hat{\Psi}^{-1} \circ \Psi \colon S(\nabla^\G, R,H) \to S(\hat{\nabla}^\G, \hat{R}, \hat{H})$ is a Courant algebroid isomorphism of the form
  \[
   \delta(\xi + r + x) = \big(\xi - \varphi^\dagger(r) + \beta(x) - \frac{1}{2}\varphi^\dagger(\varphi(x))\big) + \big(\tau(r) + \tau(\varphi(x))\big) + x,
  \]
  for any $\xi \in \Gamma(F^\ast), r \in \Gamma(\G), x \in \Gamma(F)$. Here $\tau$ is an  automorphism  of the bundle $\G$ of quadratic Lie algebras,    $\varphi \colon F \to \G$  and $\beta \colon F \to F^\ast$ are some bundle maps, and $\varphi^\dagger \colon \G \to F^\ast$ is the dual of $\varphi$ defined by $g^\G(\varphi(x), r) = \langle x, \varphi^\dagger(r) \rangle$.

  Furthermore, the restriction of $\delta$ onto ample Lie algebroids defines an isomorphism of Lie algebroids:
  \[
  \delta_L \colon L(\nabla^\G, R) \to L(\hat{\nabla}^\G, \hat{R}), \quad \delta_L(r+x) = (\tau(r) + \tau(\varphi(x))) + x.
  \]
\end{prop}

\subsubsection{Dissection-compatible metric connections}
Given a dissection $\Psi$,  the regular Courant algebroid $E$ is identified with the standard one $S(\nabla^\G, R, H)$ as described above.
To explicitly write down the generating Hamiltonian function for the regular Courant algebroid $E$ compatible with the given dissection $\Psi$, one needs a special metric connection on the pseudo-Euclidean vector bundle $(F^\ast \oplus \G \oplus F, \langle -, - \rangle)$.  We will use the following \textit{triple} $(j,\nabla^F,\nabla^B)$ which is composed of
\begin{compactenum}
  \item a splitting $j \colon B \to TM$ of the short exact sequence \eqref{Eq:SES of VB} of vector bundles over $M$;
\item an $F$-connection $\nabla^F$ on $F$ which is torsion-free, i.e., $\nabla^F_{x_1}{x_2}-\nabla^F_{x_2}{x_1}=[x_1,x_2]$ for all  $x_1,x_2 \in \Gamma(F)$;
\item a metric $B$-connection on $\G$, i.e., a bilinear map $\nabla^B \colon \Gamma(B) \times \Gamma(\G) \to \Gamma(\G)$ satisfying
\begin{align*}
  \nabla^B_{fb} r &= f\nabla^B_b r, & \nabla^B_b (fr) &= (j(b)f) r + f \nabla^B_b r,  &\mbox{and~}~ j(b) {g^{\G}}(r,s) &= {g^{\G}}(\nabla_b^B r, s) + {g^{\G}}(r, \nabla_b^B s),
\end{align*}
for all $f \in C^\infty(M), b \in \Gamma(B)$ and $r, s \in \Gamma(\G)$.
\end{compactenum}
The existence of $(j,\nabla^F,\nabla^B)$  is evident. By introducing such a triple, one is able to construct an appropriate connection on $  F^\ast \oplus F \oplus \G$:
\begin{itemize}
	\item Given a torsion-free connection $\nabla^F$ on $F$,  the dual connection on $F^\ast$ is also denoted by $\nabla^F$. Then $\nabla^F$ defines an $F$-connection on $F^\ast \oplus F$ that is compatible with the standard pairing $\langle - \mid -\rangle$.
	\item The operator $\nabla^\G$ defined in Equation~\eqref{Eq:F-conn on G} satisfies
	\[
	\nabla^\G_{fx}r = f\nabla_x^\G r, \qquad \nabla^\G_x (fr) = f\nabla^G_x r + (x(f))r,
	\]
	and
	\begin{align*}
		x{g^{\G}}(r,s) &= \rho(\Psi(x)) g(\Psi(r),\Psi(s)) = g(\Psi(x) \circ \Psi(r), \Psi(s)) + g(\Psi(r), \Psi(x) \circ \Psi(s)) \\
		&= {g^{\G}}(\nabla_x^\G r, s) + {g^{\G}}(r, \nabla_x^\G s),
	\end{align*}
	for all $x \in \Gamma(F)$, $r, s \in \Gamma(\G)$, and $f \in C^\infty(M)$.
	\item Extend the bilinear map $\nabla^B \colon \Gamma(B) \times \Gamma(\G) \to \Gamma(\G)$ to a new bilinear map
	\begin{equation}\label{Eq: extended B-connection}
	\nabla^B \colon \Gamma(B) \otimes \Gamma(F^\ast \oplus \G \oplus F) \to \Gamma(F^\ast \oplus \G \oplus F)
	\end{equation}
  by setting additionally that
	\begin{equation}\label{Eq:Def of nablaB}
		\nabla^B \colon \Gamma(B) \otimes \Gamma(F) \to \Gamma(F), \quad \nabla^B_b x : = \pr_F [j(b), x],
	\end{equation}
	and
	\[
	\nabla^B \colon \Gamma(B) \otimes \Gamma(F^\ast) \to \Gamma(F^\ast), \quad
	\langle  \nabla^B_b \xi \mid x \rangle := j(b) \langle \xi \mid x \rangle - \langle \xi \mid \nabla^B_b x \rangle,
	\] for all $b \in \Gamma(B), x \in \Gamma(F), \xi \in \Gamma(F^\ast)$.
	
	It is clear that $\nabla^B$ satisfies the usual requirements of connections:
	\[
	\nabla^B_{fb}x = f\nabla^B_b x, \quad \nabla^B_b (fx) = j(b)(f) x + f \nabla^B_b x,
	\]
	for all $f \in C^\infty(M)$.	
\end{itemize}

	In summary, we have
\begin{lem}\label{Lem:F-connection on E}
Any triple $(j,\nabla^F, \nabla^B)$ induces a metric connection $\nabla$ on the pseudo-Euclidean vector bundle $(F^\ast \oplus F \oplus \G, \langle-,-\rangle)$ defined by
$\nabla:= \nabla^F + \nabla^\G + \nabla^B$, i.e.
\[
\nabla_{x_0 + j(b_0)}(\xi + r + x) = \nabla_{x_0}^F \xi + \nabla^B_{b_0} \xi + \nabla^\G_{x_0}r + \nabla^B_{b_0}r + \nabla^F_{x_0}x + \nabla^B_{b_0} x,
\]
for all $x_0, x \in \Gamma(F), b_0 \in \Gamma(B), \xi \in \Gamma(F^\ast)$ and $r \in \Gamma(\G)$.
\end{lem}
The metric connection $\nabla$ on $(F^\ast \oplus \G \oplus F, \langle -, - \rangle)$ we constructed from the triple $(j,\nabla^F, \nabla^B)$ transfers to a metric connection on $(E,g)$ via the chosen dissection $\Psi$. It will be also denoted by $\nabla$ and referred to as the $\Psi$-\textit{compatible metric connection}  on $E$.

\subsubsection{The regular-to-naive contraction}
In the sequel, let us fix a dissection $\Psi$ of the regular Courant algebroid $E$ such that $E$ is  identified with the standard Courant algebroid $S(\nabla^\G,R,H)$. 
We also fix a triple $(j, \nabla^F, \nabla^B)$ as earlier, and denote by $\nabla$ the $\Psi$-compatible metric connection.
Then the underlying graded manifold of the minimal symplectic realization $\E$ of the Poisson manifold $E[1] \cong (F^\ast \oplus\G \oplus F)[1]$ can be identified with $T^\ast[2]M \oplus F^\ast[1] \oplus \G[1] \oplus F[1]$, with its algebra of functions being
\[
 C^\infty(\E) \cong \Gamma(\widehat{S}(T[-2]M)  \otimes {\wedge^\bullet F} \otimes {\wedge^\bullet \G} \otimes {\wedge^\bullet F^\ast}).
\]
Here we have adopted the identification between $\G^\ast[-1]$ and $\G[-1]$ via the pseudo-metric ${g^{\G}}$.
The Rothstein dg algebra $(C^\infty(\E), d_E)$ contains two dg subalgebras:
one is the regular algebra
\[
\big(C^\bullet_{\operatorname{reg}}(E) \cong \Gamma(\widehat{S}(F[-2])  \otimes {\wedge^\bullet F} \otimes {\wedge^\bullet \G} \otimes {\wedge^\bullet F^\ast}), d_E\big),
\]
and the other is the Chevalley-Eilenberg algebra of the ample Lie algebroid $A_E \cong L(\nabla^\G, R)$
\[
   \big(C^\bullet(A_E) \cong C^\bullet(L(\nabla^\G, R)) = \Gamma({\wedge^\bullet \G} \otimes {\wedge^\bullet F^\ast}), d_{\CE}\big).
\]
Our purpose is to establish a \emph{semifull algebra contraction} from the former onto the latter.  We recall the definition of this notion below; see \cites{BCSX, Real} for more details.

\begin{defn} \label{Def of semifull algebra contraction}
  Let $(A,d_A)$ and $(B,d_B)$ be commutative dg algebras. A \textbf{contraction} $(\phi, \psi, h)$ of $(A,d_A)$ onto $(B,d_B)$ is the datum of a pair of cochain maps $\phi \colon (A, d_A) \to (B, d_B), \psi \colon (B,d_B) \to (A,d_A)$ and a contracting homotopy $h \colon A \to A[-1]$, denoted by the diagram
\[
 \begin{tikzcd}
	(A, d_A) \arrow[loop left, distance=2em, start anchor={[yshift=-1ex]west}, end anchor={[yshift=1ex]west}]{}{h} \arrow[r,yshift = 0.7ex, "\phi"] &   (B, d_B), \arrow[l,yshift = -0.7ex, "\psi"]
\end{tikzcd}
\]
satisfying the relations
  \begin{align*}
    \phi \psi &= \id_B, & [h,d_A]:= hd_A + d_Ah &= \psi\phi - \id_A,
  \end{align*}
and the side conditions
\begin{align*}
    \phi h &=0, & h\psi &= 0, & h^2 &= 0.
\end{align*}
The above contraction is said to be a \textit{semifull} algebra contraction, if the following identities are satisfied for all $a, b \in A$ and $x, y \in B$:
\begin{align*}
  h\big((-1)^{\abs{a}+1}h(a)b + ah(b)\big) &= h(a)h(b), \\
                                  h(a\psi(x)) &= h(a)\psi(x), \\
\phi\big((-1)^{\abs{a}+1}h(a)b + ah(b)\big) &= 0, \\
                              \phi(a\psi(x)) &= \phi(a) x, \\
                                  \psi(xy) &= \psi(x)\psi(y).
\end{align*}
\end{defn}
According to~\cite{BCSX}*{Remark 2.13}, the first four equations above are indeed equivalent to the following conditions introduced by Real~\cite{Real}:
\begin{align}\label{Eq: semifull condition of Real}
  h(h(a)h(b)) &= 0, & h(h(a)\psi(x)) &= 0, & \phi(h(a)h(b)) &= 0, & \phi(h(a)\psi(x)) &= 0.
\end{align}
We now state and prove the main result of this section.
\begin{thm}[The Contraction Theorem]\label{contractionCA}
There exists a semifull algebra contraction
\[	
\begin{tikzcd}
	(C^\bullet_{\operatorname{reg}}(E), d_E) \arrow[loop left, distance=2em, start anchor={[yshift=-1ex]west}, end anchor={[yshift=1ex]west}]{}{h} \arrow[r,yshift = 0.7ex, "\phi"] &   (C^\bullet(A_E), \dCE) \arrow[l,yshift = -0.7ex, "\psi"].
\end{tikzcd}
\]
 \end{thm}

To prove this theorem, we need an explicit expression of the standard differential $d_E =\{\Theta, -\} = \{\rhothree+C_{\nabla}, - \}$.
First of all, by substituting expressions of the Dorfman bracket in Proposition~\ref{Prop:CSX 2.1} and the metric connection $\nabla$ in Lemma~\ref{Lem:F-connection on E} into the definition of $C_\nabla$ in~\eqref{Eq:Def of torsion}, we obtain the following explicit expression of the torsion $C_\nabla$.
\begin{lem}\label{Lem:Cnabla}
  The torsion $C_\nabla$ is a Chevalley-Eilenberg $3$-cochain of the ample Lie algebroid $A_E$, i.e. $C_\nabla \in C^3(A_E)$. Explicitly, one has
  \[
  \iota_{\xi} C_\nabla=0,
  \]
 for all $\xi \in \Gamma(F^\ast)$, and
 \begin{align}\nonumber
  &\quad C_\nabla(  r +x,  s+y,  t +z) \\ \label{Eqt:Cnablaexplict}
  &= {C^\G}(r,s,t) + {g^{\G}}(R(x,y),t) + {g^{\G}}(R(y,z),r) + {g^{\G}}(R(z,x),s) + \frac{1}{2}H(x,y,z),
 \end{align}
 for all $r,s,t \in \Gamma(\G)$ and $x,y,z \in \Gamma(F)$. Here ${g^{\G}}$ is the induced ad-invariant pseudo-metric on $\G$, ${C^\G}$ the Cartan $3$-form on $\G$,  $R$ the linear operator defined by Equation~\eqref{Eq:FFtoG}, and $H$ the map defined by Equation~\eqref{Eq:FFtoFast}.
\end{lem}

\begin{Rem}\label{Rem:3charclass}
	In~\cite{CSX}, the second author with Sti\'{e}non and Xu introduced for each regular Courant algebroid $E$ an intrinsic degree $3$ characteristic class $[C] \in H^3_{\CE}(A_E)$.
	Indeed, when we identity $E$ with the standard one $F^\ast \oplus F \oplus \G$ via a dissection $\Psi$, and choose a particular $E$-connection\footnote{Recall from~\cites{AX, CSX, GMP} that an $E$-connection on $E$ is a bilinear map $\nabla^E \colon \Gamma(E) \otimes_{\mathbb{R}} \Gamma(E) \to \Gamma(E)$ satisfying $\nabla^E_{fe_1}e_2 = f\nabla^E_{e_1}e_2$ and $\nabla^E_{e_1}fe_2 = f\nabla^E_{e_1}e_2 + (\rho(e_1)f) e_2$ for all $f \in C^\infty(M)$ and $e_1,e_2 \in \Gamma(E)$.}  $\nabla^E$ on $E$ defined by
	\begin{equation}\label{Eqt:nablaE}
		\nabla^E_{\xi+r+x}(\eta+s+y) = \nabla^F_x \eta - \frac{1}{3}H(x,y,-) + \nabla_x^\G s + \frac{2}{3}[r,s]^\G + \nabla^F_x y,
	\end{equation}
	for any metric torsion-free $F$-connection $\nabla^F$ on $(F^\ast \oplus F, \langle -, - \rangle)$, the class $[C]$ is represented by the associated Courant algebroid torsion $C_{\nabla^E}$  introduced in \cite{AX} and defined by
	\begin{equation}\label{Eqt:CnablaE}
	C_{\nabla^E}(e_1,e_2,e_3) = \frac{1}{2} \operatorname{cycl}_{123}g\left(\frac{1}{3}(e_1 \circ e_2 - e_2 \circ e_1) - (\nabla_{e_1}^Ee_2 - \nabla^E_{e_2}e_1), e_3 \right).
	\end{equation}
Note that any metric $TM$-connection $\nabla$ on $E$ induces an $E$-connection $\nabla^E$ on $E$ by setting $\nabla^E_{e_1}e_2 = \nabla_{\rho(e_1)}e_2$. The associated Courant algebroid torsion $C_{\nabla^E}$ coincides with the torsion $C_\nabla$ defined in Equation~\eqref{Eq:Def of torsion}.
	
	By substituting the particular $\nabla^E$ in Equation \eqref{Eqt:nablaE}, we get
	\[
\iota_{\xi}C_{\nabla^E}=0,\qquad \forall \xi\in \Gamma(F^\ast)
	\]
	and
	\begin{align}\nonumber
		&\quad C_{\nabla^E}(  r+x,  s+y,  t+z) \\\label{Eqt:CnablaEexplict}
		&=  - {C^\G}(r,s,t) + {g^{\G}}(R(x,y), t) + {g^{\G}}(R(y,z), r) + {g^{\G}}(R(z,x), s)+H(x,y,z).
	\end{align}
	
	So, the torsion  $C_{\nabla}$ in Equation \eqref{Eqt:Cnablaexplict} of the current paper is \textit{not} the one $C_{\nabla^E}$ in \cite{CSX} as the former is merely a $3$-cochain while the latter is $3$-cocycle.
\end{Rem}
Next, we note that the degree $3$ element $\rhothree \in \Gamma(F^\ast[-1] \otimes F[-2]) \subset C^\infty(\E)$ stemming from the anchor map $\rho$ is locally of the form $\sum_i \xi^i[-1] \otimes x_i[-2]$.
Here $\{x_i\}$ is a local frame of $\Gamma(F)$ and $\{\xi^i\}$ the dual frame of $\Gamma(F^\ast)$.
Denote by $\rhothree[2]$ $\in \Gamma(F^\ast[-1] \otimes F )$ the degree $1$ element which is locally expressed by $\sum_i \xi^i[-1] \otimes x_i$.
\begin{lem}\label{dEonF12}
  For all $x  \in \Gamma(F) \subseteq \Gamma(TM), y  \in \Gamma(F) \subset \Gamma(E )$,  we have
  \begin{align*}
    d_E(x[-2]) &= \nabla^F_{\rhothree[2]} (x[-2]) + R^\nabla(\rhothree[2], x[-2]) - \nabla_x C_\nabla, \\
  \mbox{~and~}~  d_E(y[-1]) &= y[-2] + \nabla^F_{\rhothree[2]} (y[-1]) + {\iota_y C_\nabla},
  \end{align*}
  where $\nabla^F_{\rhothree[2]} (x[-2])$ is locally expressed by
  $\sum_i \xi^i[-1] \otimes (\nabla^F_{x_i}x)[-2]$,  $R^\nabla(\rhothree[2], x[-2])$ by $\sum_i \xi^i[-1] \otimes R^\nabla(x_i,x)[-2]$,   $\nabla^F_{\rhothree[2]} (y[-1])$ by $\sum_i \xi^i[-1] \otimes (\nabla^F_{x_i} y)[-1]$, and ${\iota_y C_\nabla} \in \Gamma(\wedge^2 F^\ast\oplus (F^\ast[-1]\otimes \G[-1]) )\subset C^2(A_E)$ stands for the contraction   of the torsion $C_\nabla$ along $y \in \Gamma(F)$.
\end{lem}
\begin{proof}
We directly verify the desired formulas:
\begin{align*}
 &\quad d_E(x[-2]) = \{\rhothree, x[-2]\} + \{C_\nabla, x[-2]\} \\
 &= -\{x[-2], \xi^{i}[-1] \otimes x_i[-2]\} - \{x[-2], C_\nabla\} \qquad\qquad  \text{by Eqs.~\eqref{GlobalPoisson}, \eqref{Eq:Poisson   bracket 2} and~\eqref{Poisson bracket on E[1]}}\\
 &= -\nabla^F_x \xi^i[-1] \otimes x_i[-2] - \xi^i[-1] \otimes [x,x_i][-2] - \xi^{i}[-1] \otimes R^\nabla(x,x_i)[-2] - \nabla_x C_\nabla \\
 &= \xi^i[-1] \otimes (\nabla^F_x x_i - [x,x_i])[-2] - \xi^{i}[-1] \otimes R^\nabla(x,x_i)[-2] - \nabla_x C_\nabla \quad \text{since $\nabla^F$ is torsion-free} \\
 &= \xi^i[-1] \otimes (\nabla^F_{x_i} x)[-2] + \xi^{i}[-1] \otimes R^\nabla(x_i, x)[-2] - \nabla_x C_\nabla \\
 &= \nabla^F_{\rhothree[2]} (x[-2]) + R^\nabla(\rhothree[2], x[-2]) - \nabla_x C_\nabla;
\end{align*}
and
\begin{align*}
  d_E(y[-1]) &= \{\rhothree, y[-1]\} + \{C_\nabla, y[-1]\}  \qquad \quad \qquad \text{by Lemma~\ref{Lem:Cnabla}} \\
  &= \{y[-1], \xi^i[-1] \otimes x_i[-2]\} + {\iota_y C_\nabla}   \qquad \text{by Eqs.~\eqref{GlobalPoisson},\eqref{Eq:Poisson   bracket 2} and \eqref{Poisson bracket on E[1]}}\\
  &= y[-2] + \xi^i[-1]\otimes (\nabla^F_{x_i} y)[-1] + {\iota_y C_\nabla} \\
  &= y[-2] + \nabla^F_{\rhothree[2]} (y[-1]) + {\iota_y C_\nabla}.
\end{align*}
\end{proof}

We also note that the regular complex $C^\bullet_{\operatorname{reg}}(E)$ is a $C^\bullet(A_E)$-module. Via the dissection $\Psi$ of $E$ and the splitting $j $ of the short exact sequence \eqref{Eq:SES of VB}, the Chevalley-Eilenberg dg algebra $C^\bullet(A_E)$ of the ample Lie algebroid $A_E$ is generated by $\Gamma(\G[-1])$ and $\Gamma( F^\ast[-1])$, while the $C^\bullet(A_E)$-module  $C^\bullet_{\operatorname{reg}}(E)$ is generated by $\Gamma(F[-2])$ and $\Gamma(F[-1])$.
In order to write down the desired contracting homotopy, we introduce two auxiliary \textit{endomorphisms} of this $C^\bullet(A_E)$-module as follows.
Firstly, the torsion element $C_\nabla \in C^3(A_E)$ determines an  $C^\bullet(A_E)$-linear   map $C_\nabla\colon C^\bullet_{\operatorname{reg}}(E) \to C^\bullet_{\operatorname{reg}}(E)$ by setting
\begin{compactenum}
	\item $C_\nabla(\Gamma(  \G[-1]\oplus F[-1]\oplus F^\ast[-1]))=0$;
	\item $C_\nabla(x[-2]):={\iota_x C_\nabla}$ for all $x \in \Gamma(F)$;
	\item
\begin{align*}
  &\quad C_\nabla(x_1[-2]\odot \cdots \odot x_p[-2] \otimes y_1[-1] \odot \cdots \odot y_q [-1]) \\
   &:= \frac{1}{p+q-1}\sum_{i=1}^{p}{(\iota_{x_i} C_\nabla)} \otimes x_1[-2]\odot \cdots \odot \widehat{x_i[-2]} \odot \cdots \odot x_p[-2]\otimes y_1[-1] \odot \cdots \odot y_q [-1],
\end{align*}
 for all $p+q \geqslant 2$ and all $x_i, y_j  \in \Gamma(F)$.
\end{compactenum}
The second $C^\bullet(A_E)$-linear map
\[
  \rho^{-1} \colon C^\bullet_{\operatorname{reg}}(E) \to C^{\bullet-1}_{\operatorname{reg}}(E)
\]
is determined by
\begin{compactenum}
	\item $\rho^{-1}(\Gamma(\G[-1]\oplus F[-1]\oplus F^\ast[-1]))=0$;
	\item
\begin{align*}
 &\quad \rho^{-1}(x_1[-2] \odot \cdots \odot x_p[-2] \otimes y_1[-1] \odot \cdots \odot y_q [-1]) \\
 &:= \frac{-1}{p+q}\sum_{i=1}^{p} x_1[-2] \odot \cdots \odot \widehat{x_i[-2]} \odot \cdots \odot x_p[-2] \otimes x_i[-1] \odot y_1[-1] \odot \cdots \odot y_q [-1]
\end{align*}
for all $p\geqslant 1, q\geqslant 0$ and all $x_1 ,\cdots,x_p, y_1 ,\cdots,y_q \in \Gamma(F)$.
\end{compactenum}
\begin{lem}\label{Lem:properties of rhoinverse}
  The two maps $C_\nabla$ and $\rho^{-1}$ introduced above satisfy the following properties:
  \begin{align*}
    (\rho^{-1})^2 &= 0 \colon ~C^\bullet_{\operatorname{reg}}(E) \to C^{\bullet-2}_{\operatorname{reg}}(E) , \\
   \rho^{-1}(x \odot y) &= \frac{p+q}{p+q+r+s} \rho^{-1}(x) \odot y + (-1)^{q}\frac{r+s}{p+q+r+s} x \odot \rho^{-1}(y),
  \end{align*}
  for all $x \in \Gamma(S^p(F[-2]) \otimes \wedge^q F )$ and $y \in \Gamma(S^r(F[-2]) \otimes \wedge^s F )$, and
  \begin{align}\label{Eq:rhoinverse and Cnabla}
   \rho^{-1}d_E C_\nabla^k \rho^{-1} &= C_\nabla^k \rho^{-1} d_E \rho^{-1} = - C^k_\nabla\rho^{-1} - C_\nabla^{k+1} \rho^{-1} \colon  C^\bullet_{\operatorname{reg}}(E) \to C^\bullet_{\operatorname{reg}}(E)[-1],
  \end{align}
 for all $k \geqslant  0$, where $C_\nabla^k$ denotes the $k$-th power of the operator $C_\nabla$.
\end{lem}
\begin{proof}
 The first two properties follow directly from the definition of $\rho^{-1}$.
To prove Equation~\eqref{Eq:rhoinverse and Cnabla}, we argue by induction on $k$. For the $k = 0$ case, we need to show that
\begin{equation}\label{Eq:k=0}
   \rho^{-1} d_E \rho^{-1} = -\rho^{-1} - C_\nabla \rho^{-1}.
\end{equation}
In fact, since
\begin{align*}
  &\quad (d_E\rho^{-1})(x_1[-2] \odot \cdots \odot x_p[-2] \otimes y_1[-1] \odot \cdots \odot y_q [-1]) \\
  &= \frac{-1}{p+q}d_E\left(\sum_{i=1}^{p}x_1[-2]\odot \cdots \odot \widehat{x_i[-2]} \odot \cdots \otimes x_i[-1] \odot y_1[-1] \odot \cdots \odot y_q [-1] \right) \\
  &= \frac{-1}{p+q}\left(\sum_{i,j=1}^{p}(d_E(x_j[-2])\odot\cdots \odot \widehat{x_i[-2]} \odot \widehat{x_j[-2]} \odot \cdots \otimes x_i[-1] \odot y_1[-1] \odot \cdots \odot y_q [-1]\right) \\
  &\quad - \frac{1}{p+q}\sum_{i=1}^{p} x_1[-2] \odot \cdots \odot \widehat{x_i[-2]} \odot \cdots \odot x_p[-2] \otimes d_E\left(x_i[-1] \odot y_1[-1] \odot \cdots \odot y_q [-1]\right),
\end{align*}
and according to  Lemma~\ref{dEonF12}, we obtain
\begin{align*}
  &\quad (\rho^{-1}d_E\rho^{-1})(x_1[-2] \odot \cdots \odot x_p[-2] \otimes y_1[-1] \odot \cdots \odot y_q [-1]) \\
  &= -(\rho^{-1}+C_\nabla\rho^{-1})(x_1[-2] \odot \cdots \odot x_p[-2] \otimes y_1[-1] \odot \cdots \odot y_q [-1]).
\end{align*}
Now assume that Equation~\eqref{Eq:rhoinverse and Cnabla} holds for some $k \geqslant 0$, i.e.,
\begin{equation}\label{Eq: inductive assumption k}
  \rho^{-1}d_E C_\nabla^k \rho^{-1} = C_\nabla^k \rho^{-1} d_E \rho^{-1}.
\end{equation}
We proceed to prove that Equation~\eqref{Eq:rhoinverse and Cnabla} holds for $k+1$. 
Note that
\begin{align*}
 &\quad  (\rho^{-1}C_\nabla)(x_1[-2] \odot \cdots \odot x_p[-2] \otimes y_1[-1] \odot \cdots \odot y_q [-1]) \\
 &= \frac{p+q}{p+q-1}(C_\nabla\rho^{-1})(x_1[-2] \odot \cdots \odot x_p[-2] \otimes y_1[-1] \odot \cdots \odot y_q [-1]).
\end{align*}
Thus, we have
\begin{align*}
  &\quad (\rho^{-1}d_E C_\nabla^{k+1} \rho^{-1})(x_1[-2] \odot \cdots \odot x_p[-2] \otimes y_1[-1] \odot \cdots \odot y_q [-1]) \\
  &= (\rho^{-1}d_E C_\nabla^k) (C_\nabla \rho^{-1})(x_1[-2] \odot \cdots \odot x_p[-2] \otimes y_1[-1] \odot \cdots \odot y_q [-1]) \\
  &= \frac{p+q-1}{p+q}(\rho^{-1}d_E C_\nabla^k\rho^{-1}C_\nabla)(x_1[-2] \odot \cdots \odot x_p[-2] \otimes y_1[-1] \odot \cdots \odot y_q [-1])\; \text{by Eq.~\eqref{Eq: inductive assumption k}}\\
  &= \frac{p+q-1}{p+q}(C_\nabla^k\rho^{-1}d_E\rho^{-1} C_\nabla)(x_1[-2] \odot \cdots \odot x_p[-2] \otimes y_1[-1] \odot \cdots \odot y_q [-1]) \; \text{by Eq.~\eqref{Eq:k=0}} \\
  &= -\frac{p+q-1}{p+q}(C_\nabla^k+C_\nabla^{k+1})(\rho^{-1}C_\nabla)(x_1[-2] \odot \cdots \odot x_p[-2] \otimes y_1[-1] \odot \cdots \odot y_q [-1]) \\
  &= -(C_\nabla^k+C_\nabla^{k+1})(C_\nabla\rho^{-1})(x_1[-2] \odot \cdots \odot x_p[-2] \otimes y_1[-1] \odot \cdots \odot y_q [-1]) \\
  &= (-C_\nabla^{k+1}\rho^{-1} - C_\nabla^{k+2}\rho^{-1})(x_1[-2] \odot \cdots \odot x_p[-2] \otimes y_1[-1] \odot \cdots \odot y_q [-1]),
\end{align*}
as desired.
\end{proof}

We need the following well-known method to produce contractions from a degree $(-1)$ map satisfying certain conditions. A proof of this fact is given in Appendix~\ref{APP: Proof of contraction lemma} for completeness.
\begin{lem}\label{Lem:how to construct a contraction}
  Let $(A, d_A)$ be a cochain complex. Given a degree $(-1)$ map $h \colon A \to A$ satisfying
  \begin{equation}\label{Eq:condition on h}
     h^2  = 0 \qquad\qquad  \mbox{and} \qquad \qquad   hd_Ah  = -h,
  \end{equation}
  there exists a contraction of $(A, d_A)$ onto the subcomplex $B := \ker(hd_A) \cap \ker(d_A h)\subset A$,
  \[
 \begin{tikzcd}
	(A, d_A) \arrow[loop left, distance=2em, start anchor={[yshift=-1ex]west}, end anchor={[yshift=1ex]west}]{}{h} \arrow[r,yshift = 0.7ex, "\phi"] &   (B, d_B:= d_A\mid_B) \arrow[l,yshift = -0.7ex, "\psi"],
\end{tikzcd}
\]
where $\psi \colon B \hookrightarrow A$ is the inclusion, and $\phi \colon A \to B$ is the projection
\begin{equation}\label{Eq:Def of sigma}
 \phi = \id_A + (hd_A + d_A h).
\end{equation}
\end{lem}

Now we are ready to complete the proof of Theorem~\ref{contractionCA}.
\begin{proof}[Proof of Theorem~\ref{contractionCA}]
Define a $C^\bullet(A_E)$-linear map
\begin{equation}\label{Eq:Def of h}
  h := \frac{1}{1+C_\nabla} \circ \rho^{-1} = \sum_{k\geqslant  0}(-1)^k C_\nabla^k \circ \rho^{-1} \colon  C^\bullet_{\operatorname{reg}}(E) \to C^\bullet_{\operatorname{reg}}(E)[-1].
\end{equation}
Observe that for each homogeneous element $x_1[-2] \odot \cdots \odot x_p[-2] \otimes y_1[-1] \odot \cdots \odot y_q [-1] \in C^\bullet_{\operatorname{reg}}(E)$, one has
\[
 (C_\nabla^k \circ \rho^{-1})(x_1[-2] \odot \cdots \odot x_p[-2] \otimes y_1[-1] \odot \cdots \odot y_q [-1]) = 0,
\]
for all $k \geqslant p-1$. Thus, $h$ is well-defined.
We then check that $h$ satisfies all conditions in \eqref{Eq:condition on h} of Lemma~\ref{Lem:how to construct a contraction}.
To see that $h^2 = 0$, we note that for all $k \geqslant  1$ and all $x_1 ,\cdots, x_p  \in \Gamma(F )$, $y_1 , \cdots, y_q  \in \Gamma(F )$,
\begin{align*}
 &\quad  (\rho^{-1}C^k_\nabla)(x_1[-2] \odot \cdots \odot x_p[-2] \otimes y_1[-1] \odot \cdots \odot y_q [-1]) \\
 &= \frac{p+q}{p+q-k}(C^k_\nabla\rho^{-1})(x_1[-2] \odot \cdots \odot x_p[-2] \otimes y_1[-1] \odot \cdots \odot y_q [-1]).
\end{align*}
Thus,  one obtains
\begin{align*}
  &\quad h^2(x_1[-2] \odot \cdots \odot x_p[-2] \otimes y_1[-1] \odot \cdots \odot y_q [-1]) \\
  &= \left(\sum_{k,l \geqslant  0}(-1)^{k+l} C_\nabla^k\rho^{-1}C_\nabla^l\rho^{-1}\right)(x_1[-2] \odot \cdots \odot x_p[-2] \otimes y_1[-1] \odot \cdots \odot y_q [-1]) \\
  &= \left(\sum_{k,l \geqslant  0}(-1)^{k+l}\frac{p+q}{p+q-l}C_\nabla^{k+l} (\rho^{-1})^2\right) (x_1[-2] \odot \cdots \odot x_p[-2] \otimes y_1[-1] \odot \cdots \odot y_q [-1]) \\
  &= 0.
\end{align*}
Here in the last equality we have used the fact that $(\rho^{-1})^2 = 0$ according to Lemma~\ref{Lem:properties of rhoinverse}. Using Equation~\eqref{Eq:rhoinverse and Cnabla}, we obtain
\begin{align*}
  hd_E h &= \sum_{k,l \geqslant  0}(-1)^{k+l}C^k_\nabla \rho^{-1} d_E C^l_\nabla \rho^{-1} \\
  &=  \sum_{k,l \geqslant  0}(-1)^{k+l} C_\nabla^{k+l} \rho^{-1}d_E\rho^{-1} = \sum_{m \geqslant  0} (-1)^{m} mC_\nabla^{m} (-\rho^{-1} - C_\nabla \rho^{-1}) \\
  &= - \sum_{m\geqslant  0}(-1)^m C_\nabla^m \rho^{-1} = -h.
\end{align*}
It is clear that $\ker d_E h \cap \ker hd_E = C^\bullet(A_E)$.   Hence, by Lemma~\ref{Lem:how to construct a contraction}, we obtain a contraction
\[
\begin{tikzcd}
	(C^\bullet_{\operatorname{reg}}(E), d_E) \arrow[loop left, distance=2em, start anchor={[yshift=-1ex]west}, end anchor={[yshift=1ex]west}]{}{h} \arrow[r,yshift = 0.7ex, "\phi"] &   (C^\bullet(A_E), \dCE) \arrow[l,yshift = -0.7ex, "\psi"],
\end{tikzcd}
\]
where $\psi \colon (C^\bullet(A_E), \dCE) \to (C^\bullet_{\operatorname{reg}}(E), d_E)$ is the inclusion, and the projection $\phi$ is given by
\[
 \phi := \id + hd_E + d_E h \colon (C^\bullet_{\operatorname{reg}}(E), d_E) \to (C^\bullet(A_E), d_{\CE}).
\]
Finally, since the inclusion $\psi \colon C^\bullet(A_E) \hookrightarrow C^\bullet_{\operatorname{reg}}(E)$ is   an algebra morphism,  it suffices to show that the obtained contraction satisfies the conditions in Equation~\eqref{Eq: semifull condition of Real}.
In fact, all these conditions can be verified directly from the definition of $h$ in Equation~\eqref{Eq:Def of h} and the properties of $\rho^{-1}$ and $C_\nabla$ in Lemma~\ref{Lem:properties of rhoinverse}. We omit the details of the verifications.
\end{proof}
\begin{Rem}
   The regular-to-naive contraction in Theorem~\ref{contractionCA} does depend a priori on a choice of dissection $\Psi$ of $E$ and a triple $(j, \nabla^F, \nabla^B)$. It is natural to expect that different choices would lead to isomorphic contractions. We can not confirm this fact so far, and would like to leave it as an open problem.
\end{Rem}
As the inclusion $\psi$ in Theorem~\ref{contractionCA} is canonical, we obtain
\begin{cor}
 The regular cohomology $ H^\bullet_{\mathrm{reg}}(E)$ of $E$ and the Chevalley-Eilenberg cohomology $H^\bullet_{\mathrm{CE}}(A_E)$ of the ample Lie algebroid $A_E$ are canonically isomorphic as graded commutative algebras.
 \end{cor}
As an application, consider a transitive Courant algebroid $E$, that is, a regular Courant algebroid whose characteristic distribution $F$ coincides with $TM$. In this case, the regular subalgebra of $E$ is indeed the whole Rothstein algebra $C^\infty(\E)$, and the regular cohomology coincides with the standard cohomology. As an immediate consequence of the above corollary, we recover the following result conjectured by Sti\'{e}non and Xu~\cite{SX2008}, and first proved by Ginot and Gr\"{u}tzmann in~\cite{GM2009}.
\begin{cor}[\cites{GM2009, SX2008}]
  The standard cohomology of any transitive Courant algebroid $E$ is isomorphic to the Chevalley-Eilenberg cohomology of its ample Lie algebroid $A_E$, or equivalently, the naive cohomology of $E$.
\end{cor}

As another application, we have the following proposition.
\begin{prop}\label{Prop:phitheta}
  The projection $\phi$ of Theorem~\ref{contractionCA} sends the   generating Hamiltonian function $\Theta = \rhothree +\Cnablathree$ to the negative value of the torsion $	C_{\nabla^E}$ as in Equation \eqref{Eqt:CnablaE}. Therefore, on the cohomology level,  $\phi$ sends the
  class $[\Theta] \in H^3_{\operatorname{reg}}(E)$   to $(-[C]) \in H_{\CE}^3(A_E)$, where $[C]$ is   the degree $3$ characteristic class of the regular Courant algebroid $E$ introduced in~\cite{CSX} (see also Remark \ref{Rem:3charclass}).
\end{prop}
\begin{proof}
  By Lemma~\ref{Lem:Cnabla}, the torsion $C_\nabla$ lands in $C^3(A_E)$, and thus is invariant by the projection $\phi$, i.e., $\phi(C_\nabla) = C_\nabla$. We now compute $\phi(\rhothree)$.
  Note that $\rhothree \in \Gamma(F^\ast[-1] \otimes F[-2])$  has the expression $  \sum_{i} \xi^i[-1] \otimes x_i[-2]$ if we choose a local frame $\{x_i\}$ of $\Gamma(F)$ with dual frame $\{\xi^i\}$ of $\Gamma(F^\ast)$.
  Then we have
  \begin{align*}
    d_E(h(\rhothree)) &= \sum_i d_E(h(\xi^{i}[-1] \otimes x_i[-2])) = \sum_i d_E(\xi^i[-1] \otimes x_i[-1])\\
    &= \sum_i \big(d_{\CE}(\xi^i[-1]) \otimes x_i[-1] - \xi^{i}[-1] \otimes d_E(x_i[-1])\big)  \qquad \text{by Lemma~\ref{dEonF12}} \\
    &= \sum_i \bigl( d_{\CE}(\xi^i[-1]) \otimes x_i[-1] - \xi^{i}[-1] \otimes (x_i[-2] + \nabla^F_{\rhothree[2]} (x_i[-1]) + {\iota_{x_i} C_\nabla})\bigr),
  \end{align*}
  and
  \begin{align*}
    &\quad h(d_E(\rhothree)) = \sum_i h(d_E(\xi^i[-1] \otimes x_i[-2]))  \\
    &= \sum_i \bigl(-d_{\CE}(\xi^i[-1]) \otimes x_i[-1] + \xi^i[-1] \otimes h(d_E(x_i[-2]))\bigr) \quad \text{by Lemma~\ref{dEonF12}}\\
    &= \sum_i \bigl(-d_{\CE}(\xi^i[-1]) \otimes x_i[-1] + \xi^i[-1] \otimes h(\nabla^F_{\rhothree[2]} (x_i[-2]) + R^\nabla(\rhothree[2], x_i[-2]) - \nabla_{x_i}C_\nabla)\bigr) \\
    &= \sum_i \bigl(-d_{\CE}(\xi^i[-1]) \otimes x_i[-1] + \xi^i[-1] \otimes \nabla^F_{\rhothree[2]}(x_i[-1])\bigr).
  \end{align*}
  Thus, we obtain
  \begin{align*}
    \phi(\rhothree) &= \rhothree + d_E(h(\rhothree)) + h(d_E(\rhothree)) \\
                     &= \rhothree - \sum_i ( \xi^i[-1] \otimes x_i[-2] + \xi^i[-1] \otimes {\iota_{x_i} C_\nabla}) \\
                     &= -\sum_i \xi^i[-1] \otimes {\iota_{x_i} C_\nabla} = -C_\nabla(\rhothree[2]).
  \end{align*}
  The term $C_\nabla(\rhothree[2])\in \Gamma((\wedge^2 F^\ast \otimes \G[-1]) \oplus \wedge^3 F^\ast)\subset C^3(A_E)$ is determined by
  \begin{align*}
  	C_\nabla(\rhothree[2])&(r +x,  s+y,  t +z)=C_\nabla(x,s+y,t+z)+C_\nabla(r+x, y,t+z)+C_\nabla(r+x,s+y, z)\\
  	&= 2\big( {g^{\G}}(R(x,y),t) + {g^{\G}}(R(y,z),r) + {g^{\G}}(R(z,x),s)\big) + \frac{3}{2}H(x,y,z),
  \end{align*}
 according to Equation~\eqref{Eqt:Cnablaexplict}.
Hence, using Equation~\eqref{Eqt:Cnablaexplict} again,  one sees that the $3$-form
  \[
   \phi(\Theta) = \phi(\rhothree +\Cnablathree) = \phi(\rhothree) +\Cnablathree =  - C_\nabla(\rhothree[2])+C_\nabla
  \]
is given explicitly by
  \begin{align*}
  \phi(\Theta)(r +x,  s+y,  t +z) &={C^\G}(r,s,t) - {g^{\G}}(R(x,y), t) - {g^{\G}}(R(y,z), r) - {g^{\G}}(R(z,x), s)-H(x,y,z) \\
  &=- C_{\nabla^E}(r+x, s+y, t+z),
\end{align*}
according to the definition of the degree $3$ characteristic cocycle $C_{\nabla^E}$ of $E$ in Equation \eqref{Eqt:CnablaEexplict}. Consequently, we have $\phi([\Theta])=-[C_{\nabla^E}] = -[C]$.
\end{proof}

\section{The minimal model and the standard cohomology}\label{Sec:2results}
Throughout this section, we continue with the assumption that $(E,g, \rho, \circ)$ is a regular Courant algebroid over $M$, $F=\rho(E) \subseteq TM$ the characteristic distribution, and $\G=\ker\rho/(\ker\rho)^\perp$ the bundle of quadratic Lie algebras.
\subsection{The minimal model of a regular Courant algebroid}
Let us fix a dissection $\Psi$ of $E$ as in Definition \ref{Def:dissection} so that $(E,g) \cong (F^\ast \oplus \G \oplus F, \langle -, - \rangle)$.
We first extend the regular-to-naive contraction in Theorem~\ref{contractionCA} to a semifull algebra contraction of the Rothstein dg algebra $(C^\infty(\E), d_E)$.

As a starting point, we fix a triple $(j,\nabla^F,\nabla^B)$ and the associated metric connection $\nabla$ compatible with $\Psi$ as in Lemma~\ref{Lem:F-connection on E}.
Denote by $(\E = T^\ast[2]M \oplus E[1], \omega_{g,\nabla}, \Theta)$ the corresponding Hamiltonian $2$-algebroid.
Now the  vector space $\Cst^n(E)$ consisting of degree $n$ elements in $C^\infty(\E)$ has a decomposition
\[
\Cst^n(E) = \bigoplus_{k+2l=n}\Cr^k(E) \otimes_{C^\infty(M)} \Gamma(S^l (B[-2])).
\]
The behaviour of the homological vector field $d_E$ on $\Cr^\bullet(E)$ has been explained in the previous section.
To investigate $d_E$ on $\Cst^\bullet(E)$  under the above decomposition, it suffices to consider the restriction of $d_E$ on $\Gamma(\widehat{S}B([-2]) )$.
Recall that the normal bundle $B$ admits a canonical $F$-module structure, known as the Bott-$F$-connection, defined by
\[
  \nabla_x^{\operatorname{Bott}}b := \pr_B[x, j(b)]
\]
for all $x \in \Gamma(F), b \in \Gamma(B)$.
Meanwhile, the ample Lie algebroid $A_E$ of $E$ is now identified with the standard one $L(\nabla^\G, R)$. The Bott $F$-connection $\nabla^{\operatorname{Bott}}$ extends to a flat $A_E$-connection $\widetilde{\nabla}$ on $B$ simply by setting
\[
 \widetilde{\nabla}_{x+r} b = \nabla^{\operatorname{Bott}}_x b,
\]
for all $x \in \Gamma(F), r \in \Gamma(\G)$ and $b \in \Gamma(B)$.
The corresponding Chevalley-Eilenberg differential $d_{\CE}$ is determined by
\[
  d_{\CE}(b) = \sum_i \xi^i[-1] \otimes \nabla_{x_i}^{\operatorname{Bott}}b  \in \Gamma(F^\ast[-1] \otimes B) \subset C^1(A_E; B),
\]
for all $b \in \Gamma(B)$, where $\{x_i\}$ is any local frame of $\Gamma(F)$ and $\{\xi^i\}$ its dual frame.
We will use the same notation for the Chevalley-Eilenberg differential of the $A_E$-module $B[-2]$, i.e., $d_{\CE}(b[-2]) = (d_{\CE}b) [-2]$ for all $b \in \Gamma(B )$.

We then introduce  a $C^\infty(M)$-linear map
\begin{align*}
	d_T &\colon \Gamma(B[-2]) \to C^3(A_E)
\end{align*}
   defined by
 \begin{equation}\label{Eq: Def of dT}
 d_T(b[-2]) := R^\nabla(\rhothree[2], j(b)[-2]) - \nabla_{j(b)} C_\nabla,
 \end{equation}
 where $R^\nabla(\rhothree[2], j(b)[-2]) \in \Gamma(F^\ast[-1] \otimes \wedge^2 \G^\ast) \subset C^3(A_E)$ is given   by
 \begin{equation}\label{Eq:Def of Rrhob}
 	R^\nabla(\rhothree[2], j(b)[-2])(x[1], r[1], s[1]) = {g^{\G}}\big(R^\nabla(x, j(b))r, s\big),
 \end{equation}for all $x \in \Gamma(F), r, s \in \Gamma(\G)$,
 and $\nabla_{j(b)} C_\nabla$ is the covariant derivative of $C_\nabla\in C^3(A_E)$ along $j(b) \in \Gamma(TM)$.

\begin{lem}\label{Lem:dst on B}
	The restriction of $\dst$ on $\Gamma(B[-2])$ is the summation of  $\dstone$ and $\dsttwo$.
\end{lem}
\begin{proof}
For any $b  \in \Gamma(B )$, we have
\[
  \dst(b[-2]) = \{\rhothree, b[-2]\} + \{C_\nabla, b[-2]\} = - \{b[-2], \rhothree\} - \nabla^B_{j(b)} C_\nabla.
\]
We then work locally.
Since $\rhothree = \sum_i \xi^i[-1] \otimes x_i[-2] \in \Gamma(F^\ast[-1] \otimes F[-2])$, the first term $-\{b[-2], \rhothree\}$ has two components in $\Gamma(F^\ast[-1] \otimes B[-2])$ and $\Gamma(\wedge^3 E^\ast) \cong \Gamma(\wedge^3 E)$, respectively. Denote the component in $\Gamma(F^\ast[-1] \otimes B[-2])$ by $d^1(b[-2])$. Using the expressions of the degree $(-2)$ Poisson bracket in~\eqref{GlobalPoisson} and~\eqref{Eq:Poisson   bracket 2}, we have
\begin{align*}
  d^1(b[-2]) &= -\sum_i (\nabla_{j(b)} \xi^i)[-1] \otimes x_i[-2] - \sum_i \xi^i[-1] \otimes [j(b), x_i][-2] \\
&= \sum_i \xi^i[-1] \otimes (\nabla_b^B x_i - [j(b),x_i])[-2] = \sum_i \xi^i[-1] \otimes (\pr_F[j(b), x_i] - [j(b),x_i])[-2] \\
&= \sum_i \xi^i[-1] \otimes (\pr_B[x_i, j(b)])[-2] = d_{\CE}(b[-2]).
\end{align*}
On the other hand, the component $R^\nabla(\rhothree[2], j(b)[-2])$ of $(-\{b[-2],\rhothree\})$ in $\Gamma(\wedge^3 E^\ast)$ satisfies
\[
  R^\nabla(\rhothree[2], j(b)[-2])(e_1[1],e_2[1], e_3[1]) = \sum \operatorname{cycl}_{123} g\big(R^\nabla(\rho(e_1), j(b))e_2, e_3\big),
\]
for all $e_1,e_2,e_3 \in \Gamma(E)$.
To prove that $R^\nabla(\rhothree[2], j(b)[-2])$ is one and the same as in~\eqref{Eq:Def of Rrhob}, it suffices to show the relation
\begin{equation}\label{Eq:Rxyzeta=0}
 R^\nabla(\rhothree[2], j(b)[-2])(x[1],y[1],\zeta[1]) = 0,
\end{equation}
for all $x,y \in \Gamma(F), \zeta \in \Gamma(F^\ast)$.
This indeed follows from a direct computation:
\begin{align*}
 &\quad  R^\nabla(\rhothree[2], j(b)[-2])(x[1],y[1],\zeta[1]) \\
&= g(R^\nabla(x,j(b))y, \zeta) + g(R^\nabla(\zeta,j(b))x,y) + g(R^\nabla(y,j(b))\zeta, x) \\
&= \frac{1}{2} \left(\langle \zeta \mid \nabla^F_x\nabla^B_{b} y - \nabla^B_{b}\nabla^F_x y - \nabla_{[x, j(b)]}y \rangle + \langle \nabla^F_y\nabla^B_{b} \zeta- \nabla^B_{b} \nabla^F_y \zeta - \nabla_{[y, j(b)]}\zeta \mid x \rangle\right) \\
&=\frac{1}{2} \langle \zeta \mid \nabla^F_x\nabla^B_{b} y - \nabla^F_y\nabla^B_{b} x  - \nabla^B_{b}\nabla^F_x y + \nabla^B_{b}\nabla^F_y x  - \nabla_{[x,j(b)]}y + \nabla_{[y,j(b)]}x \rangle.
\end{align*}
Then, using Equation~\eqref{Eq:Def of nablaB} and the fact that $\nabla^F$ is torsion-free, one has
\begin{align*}
&\quad \nabla^F_x\nabla^B_{b} y - \nabla^F_y\nabla^B_{b} x  - \nabla^B_{b}\nabla^F_x y + \nabla^B_{b}\nabla^F_y x  - \nabla_{[x,j(b)]}y + \nabla_{[y,j(b)]}x \\
&= \nabla^F_x \pr_F[j(b),y] - \nabla^F_y \pr_F[j(b),x] - \pr_F[j(b), [x,y]]-\pr_F[j(\pr_B[x,j(b)]), y] \\
&\quad - \nabla^F_{\pr_F[x,j(b)]}y + \pr_F[j(\pr_B[y,j(b)]), x] + \nabla^F_{\pr_F[y,j(b)]}x \\
&= -[x, \pr_F[y,j(b)]] - [\pr_F[x,j(b)], y]  - \pr_F[j(b), [x,y]] \\
&\quad - \pr_F[j(\pr_B[x,j(b)]), y] + \pr_F[j(\pr_B[y,j(b)]), x] \\
&= -\pr_F([x,[y,j(b)]] +[[x,j(b)],y] + [j(b),[x,y]]) = 0.
\end{align*}
This proves Equation~\eqref{Eq:Rxyzeta=0}.
So we have
\[
d_T(b[-2]) = \dst(b[-2]) - \dstone(b[-2]) = R^\nabla(\rhothree[2], j(b)[-2]) - \nabla_{j(b)} C_\nabla \in C^3(A_E).
\]
\end{proof}
We also denote by $d_{\CE}$ the Chevalley-Eilenberg differential of the $A_E$-module  $\widehat{S}(B[-2])$, i.e.,
\[
d_{\CE} \colon  C^\bullet(A_E; \widehat{S}(B[-2]) \to C^{\bullet+1}(A_E; \widehat{S}(B[-2]).
\]
Meanwhile, the map $d_T$ extends by Leibniz rule to the following one
\begin{align*}
 d_T \colon &\Gamma(S^\diamond (B[-2])) \to  C^3(A_E; S^{\diamond-1} (B[-2])),
\end{align*}
which extends further to a $C^{\bullet}(A_E)$-linear operator
\begin{align*}
d_T \colon &C^{\bullet}(A_E; S^{\diamond} (B[-2])) \to  C^{\bullet+3}(A_E; S^{\diamond-1} (B[-2])).
\end{align*}
In summary, when restricted to the subalgebra $C^\bullet(A_E; \widehat{S}(B[-2]))$ of $C^\infty(\E)$, the differential $d_E$ is exactly $d_{\CE} + d_T$.
In other words, we obtain a dg submanifold
\[
 (\mathcal{M}_E, Q_E) := (B^\ast[2] \oplus A_E[1], d_{\CE} + d_T)
\]
of $(\E = T^\ast[2]M \oplus E[1], d_E)$.

This dg submanifold indeed contains all cohomological data of the regular Courant algebroid $E$ since it is a deformation retract of the dg manifold $(\E, d_E)$. More precisely, we prove
\begin{thm}\label{Thm: semifull contraction for Rothstein algebra}
	The semifull algebra contraction $(\phi,\psi,h)$ in Theorem~\ref{contractionCA} extends to a semifull algebra contraction for the Rothstein dg algebra
	\[	
	\begin{tikzcd}
		(C^\infty(\E) = C_{\mathrm{reg}}^\bullet(E) \otimes_{C^\infty(M)} \Gamma(\widehat{S}(B[-2])),  d_E) \arrow[loop left, distance=2em, start anchor={[yshift=-1ex]west}, end anchor={[yshift=1ex]west}]{}{{\overline{h}}} \arrow[r,yshift = 0.7ex, "{\overline{\phi}}"] &   \big(C^\infty(\mathcal{M}_E) = C^\bullet(A_E; \widehat{S}(B[-2])), Q_E\big), \arrow[l,yshift = -0.7ex, "{\overline{\psi}}"]
	\end{tikzcd}
	\]	
	where the three maps $(\overline{\phi},\overline{\psi},\overline{h})$ are defined by
	\begin{align*}
		{\overline{\phi}} &:= \phi \otimes \id_{\Gamma(\widehat{S}(B[-2]) )} \colon C^\infty(\E) \to  C^\infty(\mathcal{M}_E) , \\
		{\overline{\psi}} &:= \psi \otimes  \id_{\Gamma(\widehat{S}(B[-2]) )} \colon C^\infty(\mathcal{M}_E)  \to  C^\infty(\E) \\
		{\overline{h}} &:= h \otimes  \id_{\Gamma(\widehat{S}(B[-2]) )} \colon C^\infty(\E) \to  C^\infty(\E).
	\end{align*}
\end{thm}
\begin{proof}
	Indeed, by Lemma~\ref{Lem:dst on B} and Theorem~\ref{contractionCA}, both the inclusion ${\overline{\psi}}$ and the projection ${\overline{\phi}}$ are cochain maps.
	To see that $({\overline{\phi}}, {\overline{\psi}}, {\overline{h}})$ is a contraction, it remains to prove the identity
	\[
	{\overline{\psi}} \circ {\overline{\phi}} = \id_{C^\infty(\E)} + [d_E \otimes \id_{\Gamma(\widehat{S}(B[-2]) )} + \id_{C^\bullet_{\operatorname{reg}}(E)} \otimes (\dstone + d_T), {\overline{h}}].
	\]
	In fact, since $\psi\circ\phi=\id_{C_{\mathrm{reg}}^\bullet(E)} + [d_E, h]$ (according to Theorem~\ref{contractionCA}), it amounts to verify that for all $\omega \in C_{\mathrm{reg}}^k(E), b  \in \Gamma(B )$, one has
	\begin{equation}\label{eq111}
		[\id_{C^\bullet_{\operatorname{reg}}(E)}  \otimes \dstone, {\overline{h}}](\omega \otimes b[-2]) := (-1)^{k-1}h(\omega)\cdot \dstone(b[-2]) + (-1)^k{\overline{h}}(\omega \cdot \dstone (b[-2]))=0,
	\end{equation}
	and
	\begin{equation}\label{eq222}
		[\id_{C^\bullet_{\operatorname{reg}}(E)}  \otimes d_T, {\overline{h}}](\omega \otimes b[-2]) := (-1)^{k-1}h(\omega)\cdot d_T(b[-2]) + (-1)^k{\overline{h}}(\omega \cdot d_T(b[-2]))=0.
	\end{equation}
	We observe that
	\begin{align*}
		{\overline{h}}(\omega\cdot \dstone (b[-2])) &=  {\overline{h}}\left(\sum_i (\omega \cdot \xi^i)\otimes (\nabla_{x_i}^{\operatorname{Bott}}b)[-2] \right)\\
		&= \sum_i h(\omega \cdot \xi^i) \otimes (\nabla_{x_i}^{\operatorname{Bott}}b)[-2] \quad \text{since $h$ is $C^\bullet(A_E)$-linear}\\
		&= \sum_i\bigl(h(\omega)\cdot \xi^i\bigr)\otimes (\nabla_{x_i}^{\operatorname{Bott}}b)[-2] \\
		&= h(\omega)\cdot \dstone (b[-2]),
	\end{align*}
	and thus Equation~\eqref{eq111} holds. Equation~\eqref{eq222} can be verified similarly.
	
	Finally, since the tripe $(\phi,\psi, h)$ in Theorem~\ref{contractionCA} is a semifull algebra contraction, it follows directly that $({\overline{\phi}}, {\overline{\psi}}, {\overline{h}})$ is a semifull algebra contraction as well.
\end{proof}
As a consequence, the dg submanifold $(\mathcal{M}_E, Q_E)$ provides a model for the standard cohomology of the regular Courant algebroid $E$.
Although the map $d_T$ in $Q_E$ depends on the choice of dissection $\Psi$ and the triple $(j,\nabla^F,\nabla^B)$), we prove that it is indeed unique in the following sense.
\begin{prop}\label{Lem:Transgression}
The $C^\infty(M)$-linear map $d_T$, viewed as an element in $C^3(A_E; B^\ast[2])$, is a $3$-cocycle of the Lie algebroid $A_E$ valued in $B^\ast[2]$.
Furthermore, the cohomology class $[d_T] \in H_{\CE}^3(A_E; B^\ast[2])$ is canonical, i.e., is independent of the choice of the triple $(j,\nabla^F,\nabla^B)$ and the dissection $\Psi$.
\end{prop}
The proof of this proposition is given in Appendix~\ref{App: Proof of  trangression}.
As a consequence, we have
\begin{prop}\label{Cor: uniqueness of minimal model}
The dg submanifold $(\mathcal{M}_E, Q_E)$ of $(\E, d_E)$ is unique up to isomorphism. More precisely, given another dissection $\widehat{\Psi}$ and triple $(\widehat{j},\widehat{\nabla}^F, \widehat{\nabla}^B)$, we denote by $\widehat{\nabla}$ the $\widehat{\Psi}$-compatible metric connection on $E$ and by $\widehat{d_T} \in C^3(A_E; B^\ast[2])$ the associated element defined as in~\eqref{Eq: Def of dT}.
Then there exists an isomorphism of dg manifolds
\[
   \Delta \colon (\mathcal{M}_E, Q_E = d_{\CE} + d_T) \to (\mathcal{M}_E, \widehat{Q}_E:= d_{\CE} + \widehat{d_T}).
\]
\end{prop}
\begin{proof}
  Assume that $L = L(\nabla^\G, R)$ and $\widehat{L} = L(\widehat{\nabla}^\G, \widehat{R})$ are the Lie algebroids obtained from the ample Lie algebroid $A_E$ of $E$ via the dissections $\Psi$ and $\widehat{\Psi}$ respectively as in Proposition~\ref{Prop:CSX 2.1}. Then the algebra $C^\infty(\mathcal{M}_E)$ of functions of $\mathcal{M}_E$ is identified as $C^\bullet(L; \widehat{S}(B[-2]))$ and $C^\bullet(\widehat{L}; \widehat{S}(B[-2]))$, respectively.

  Assume further that $\delta_L \colon L \to \widehat{L}$ is the isomorphism of Lie algebroids induced by the change of dissection $\delta = \widehat{\Psi}^{-1} \circ \Psi$ as in Proposition~\ref{Prop: different dissection}. This isomorphism extends obviously to an isomorphism of dg manifolds from $(B^\ast[2] \oplus L[1] ,d_{\CE})$ to $(B^\ast[2] \oplus \widehat{L}[1], d_{\CE})$.
  Denote the associated morphism of algebras of functions by
  \[
   \delta_L^\ast \colon C^\bullet(\widehat{L};\widehat{S}(B[-2])) \to C^\bullet(L;\widehat{S}(B[-2]))
  \]
  --- it is obviously an isomorphism of $C^\infty(M)$-algebras.
  By Proposition~\ref{Lem:Transgression}, there exists an element $\gamma \in C^2(L; B^\ast[2]) \cong C^2(A_E; B^\ast[2])$ such that $\delta_L^\ast(\widehat{d_T}) - d_T = d_{\CE}(\gamma)$.
  Note that the element $\gamma$ induces a degree $0$ inner derivation of the algebra $C^\infty(\M_E) \cong C^\bullet(L; \widehat{S}(B[-2]))$
  \[
   \iota_\gamma \colon C^\bullet(L; S^{\diamond+1}(B[-2])) \to C^{\bullet+2}(L; S^\diamond(B[-2])).
  \]
  This inner derivation determines an automorphism of $C^\infty(M)$-algebras
  \[
  e^\gamma := \sum_{k=0}^{\infty} \frac{1}{k!}(\iota_\gamma)^k \colon C^\bullet(L; \widehat{S}(B[-2])) \to  C^\bullet(L; \widehat{S}(B[-2])).
  \]
  Consider the following isomorphism of $C^\infty(M)$-algebras
  \begin{equation}\label{Eq: Def of Thetaast}
   \Delta^\ast:= e^{\gamma} \circ \delta_L^\ast \colon C^\bullet(\widehat{L}; \widehat{S}(B[-2])) \to  C^\bullet(L; \widehat{S}(B[-2])).
  \end{equation}
  Since
  \begin{align*}
    \Delta^\ast \circ \widehat{Q}_E &= (e^\gamma \circ \delta^\ast) \circ (d_{\CE} + \widehat{d_T}) = e^\gamma \circ (\delta^\ast \circ d_{\CE} + \delta^\ast \circ \widehat{d_T}) \\
    &= e^\gamma \circ (d_{\CE} \circ \delta^\ast) + e^\gamma \circ (d_T \circ \delta^\ast + d_{\CE}(\gamma) \circ \delta^\ast) \\
    &= (d_{\CE} \circ e^\gamma - d_{\CE}(e^\gamma)) \circ \delta^\ast + d_T \circ (e^{\gamma} \circ \delta^\ast) + d_{\CE}(\gamma) \circ (e^\gamma \circ \delta^\ast) \\
    &= (d_{\CE} + d_T) \circ (e^\gamma \circ \delta^\ast) = Q_E \circ \Delta^\ast,
  \end{align*}
  it follows that
  \[
  \Delta = (\id_M, \Delta^\ast) \colon (\M_E \cong B^\ast[2] \oplus L[1], Q_E) \to (\M_E \cong B^\ast[2] \oplus \hat{L}[1], \widehat{Q}_E)
  \]
  is an isomorphism of dg manifolds as desired.
\end{proof}

Moreover, $(\mathcal{M}_E, Q_E)$ corresponds to an $L_\infty$-algebroid structure on the graded vector bundle $B^\ast[1] \oplus A_E$ over $M$ (cf.~\cite{Roy0}), which is minimal in the sense that its unary bracket vanishes (cf.~\cite{JRSW}).
Hence, we call $(\mathcal{M}_E, Q_E)$ the \textbf{minimal model} of the regular Courant algebroid $E$.

Note that the minimal symplectic realization $\E$ of $E[1]$ together with the Hamiltonian vector field $X_{\Theta}$ is a symplectic dg manifold of degree $2$. In contrast, our minimal model $\mathcal{M}_E = B^\ast[2] \oplus A_E[1]$ of $E$ admits a \textit{$2$-shifted derived Poisson manifold} structure in the sense of \cite{BCSX} by the second author with Bandiera, Sti\'{e}non and Xu. This is indeed an application of Theorem \ref{Thm: semifull contraction for Rothstein algebra}.
\begin{prop}\label{prop: 2-shifted derived Poisson}
   The algebra $C^\infty(\mathcal{M}_E) = C^\bullet(A_E; \widehat{S}(B[-2]) )$ of functions on $\mathcal{M}_E$ admits a degree $(-2)$  derived Poisson algebra structure $\{\lambda_k\}_{k\geqslant 1}$  where
  \begin{compactenum}
    \item the differential $\lambda_1 = Q_E$;
    \item the binary bracket $\lambda_2$ is determined by
    \begin{align*}
      \lambda_2(\xi[-1],\eta[-1]) &= g^\G(i^\ast\xi,i^\ast\eta), \\
     \lambda_2(b_1[-2], \xi[-1]) &= (\nabla^B_{b_1} \xi)[-1], \\
     \lambda_2(b_1[-2],b_2[-2]) &= \pr_B([j(b_1),j(b_2)])[-2] - C_\nabla(\pr_F[j(b_1),j(b_2)]) + R_\G^\nabla(j(b_1), j(b_2)),
    \end{align*}
    for all $b_1 , b_2  \in \Gamma(B )$ and $\xi ,\eta  \in \Gamma(A_E^\ast )$, where $i^\ast$ is the dual of the inclusion  $i \colon \G \hookrightarrow A_E$, and $R^{\nabla}_\G$ is the curvature of the linear connection $\nabla$ on $\G$;
    \item the ternary bracket $\lambda_3$ is determined by
    \[
       \lambda_3(b_1[-2], b_2[-2], \xi[-1]) = -g\big(\pr_F([j(b_1),j(b_2)]), \xi\big).
    \]
   \item the quaternary and all higher brackets vanish: $\lambda_k = 0$ for all $k \geqslant 4$.
  \end{compactenum}
\end{prop}
\begin{proof}
  Note that $(C^\infty(\E), d_E, \{-,-\})$ is a dg Poisson algebra of degree $(-2)$. Applying the homotopy transfer theorem~\cite{BCSX}*{Theorem 2.16}  to the semifull algebra contraction $(\overline{\phi}, \overline{\psi}, \overline{h})$ in Theorem~\ref{Thm: semifull contraction for Rothstein algebra}, we obtain a degree $(-2)$-shifted derived Poisson algebra structure on $C^\infty(\mathcal{M}_E)$  with the unary bracket $\lambda_1 = Q_E=d_{\CE}^B + d_T$  and the higher brackets $\{\lambda_k\}_{k\geqslant  2}$  defined   as follows:
  \begin{itemize}
    \item[$k=2\colon$] The binary bracket $\lambda_2$ is, by definition,
    \[
       \lambda_2(\alpha, \beta) := \overline{\phi}(\{\overline{\psi}(\alpha), \overline{\psi}(\beta)\}),\quad \forall \alpha,\beta \in C^\bullet(A_E; \widehat{S}(B[-2]) ).
    \]
    More precisely, we consider the situation on generating elements. Using the explicit expressions of
the degree $2$ Poisson bracket in Equations~\eqref{GlobalPoisson} and \eqref{Eq:Poisson   bracket 2}, we have
     \begin{align*}
      \lambda_2(\xi[-1],\eta[-1]) &= \overline{\phi}(\{\overline{\psi}(\xi[-1]), \overline{\psi}(\eta[-1])\}) = \phi(\{\xi[-1],\eta[-1]\}) = g(\xi,\eta)=g^\G(i^\ast\xi,i^\ast\eta); \\
      \lambda_2(b[-2],\xi[-1]) &= \overline{\phi}(\{\overline{\psi}(b[-2]), \overline{\psi}(\xi[-1])\}) = \phi(\{j(b[-2]), \xi[-1]\})  = (\nabla^B_b \xi)[-1],
    \end{align*}
    for all $b \in \Gamma(B), \xi, \eta \in \Gamma(A_E^\ast)$; and
    \begin{align*}
      &\lambda_2(b_1[-2], b_2[-2]) = \overline{\phi}(\{\overline{\psi}(b_1[-2]), \overline{\psi}(b_2[-2])\}) = \overline{\phi}(\{j(b_1[-2]), j(b_2[-2])\}) \\
      &= \overline{\phi}\big([j(b_1),j(b_2)][-2] + R^\nabla(j(b_1), j(b_2))\big) \qquad\quad \text{by Definition of $\overline{\phi}$} \\
      &= \pr_B[j(b_1), j(b_2)][-2] + \phi\big(\pr_F[j(b_1), j(b_2)][-2] + R^\nabla(j(b_1), j(b_2))\big) \quad \text{by Lemma~\ref{dEonF12}} \\
      &= \pr_B[j(b_1), j(b_2)][-2] - C_\nabla(\pr_F[j(b_1),j(b_2)]) + R^{\nabla}_\G(j(b_1),j(b_2)).
    \end{align*}
     \item[$k=3\colon$] The ternary bracket $\lambda_3$ is given by
 \[
    \lambda_3(\alpha_1,\alpha_2,\alpha_3) = \sum_{\sigma \in \operatorname{sh}(2,1)} \pm \overline{\phi}\big(\{\overline{h}(\{\overline{\psi}(\alpha_{\sigma(1)}), \overline{\psi}(\alpha_{\sigma(2)})\}), \overline{\psi}(\alpha_{\sigma(3)})\}\big),
 \]
 for all $\alpha_i \in C^\bullet(A_E; \widehat{S}(B[-2]) ), 1 \leqslant  i \leqslant  3$, where $\operatorname{sh}(2,1)$ denotes the set of all $(2,1)$-shuffles, and $\pm$ is some proper Koszul sign.
 Since
 \begin{align*}
   \overline{h}(\{\overline{\psi}(b[-2]), \overline{\psi}(\xi[-1])\}) &= 0, & \overline{h}(\{\overline{\psi}(\xi[-1]), \overline{\psi}(\eta[-1])\}) &= 0,
 \end{align*}
 for all $\xi,\eta \in \Gamma(A^\ast_E)$ and $b  \in \Gamma(B )$,  it follows that the only nonzero trinary bracket on generating elements is given by
  \begin{align*}
    &\lambda_3(b_1[-2],b_2[-2],\xi[-1]) = \overline{\phi}\big(\{{\overline{h}}(\{j(b_1[-2]),j(b_2[-2])\}), \psi(\xi[-1])\}\big)  \quad \text{by Eq.~\eqref{Eq:Poisson bracket 2}}\\
    &= \overline{\phi}\big(\{{\overline{h}}(\pr_B[j(b_1),j(b_2)][-2] + \pr_F[j(b_1),j(b_2)][-2] + R^\nabla(j(b_1), j(b_2))), \psi(\xi[-1])\}\big) \\
    &= -\overline{\phi}\big(\{\pr_F[j(b_1),j(b_2)][-1], \psi(\xi[-1]) \} \big) \qquad \text{by Eqs.~\eqref{GlobalPoisson} and~\eqref{Eq:Poisson bracket 2}}\\
    &= -\overline{\phi}\big(g(\pr_F[j(b_1),j(b_2)], \xi[-1])\big) =  -g(\pr_F[j(b_1),j(b_2)], \xi ),
  \end{align*}
  for all $b_1, b_2 \in \Gamma(B)$ and $\xi \in \Gamma(A_E^\ast)$.
  \item[$k=4\colon$] The quaternary bracket $\lambda_4$ is given by
  \begin{align*}
    \lambda_4(\alpha_1,\alpha_2,\alpha_3,\alpha_4) &= \sum_{\sigma \in \operatorname{sh}(2,2)} \pm \overline{\phi}\big(\{{\overline{h}}(\{\overline{\psi}(\alpha_{\sigma(1)}), \overline{\psi}(\alpha_{\sigma(2)})\}), \overline{h}(\{\overline{\psi}(\alpha_{\sigma(3)}), \overline{\psi}(\alpha_{\sigma(4)})\})\}\big) \\
    &\quad + \sum_{\tau \in \operatorname{sh}(2,1,1)} \pm \overline{\phi}\big(\{{\overline{h}}(\{\overline{h}(\{\overline{\psi}(\alpha_{\sigma(1)}), \overline{\psi}(\alpha_{\sigma(2)})\}), \overline{\psi}(\alpha_{\sigma(3)})\}), \overline{\psi}(\alpha_{\sigma(4)})\}\big),
  \end{align*}
  for all $\alpha_i \in C^\bullet(A_E; \widehat{S}(B[-2]) ), 1 \leqslant  i \leqslant  4$. It is clear that both terms vanish when acting on generating elements of the algebra $C^\bullet(A_E; \widehat{S}(B[-2]) )$. Thus, we have $\lambda_4 = 0$.
  \item[$k\geqslant 5 \colon$] All higher brackets $\{\lambda_k\}_{k\geqslant  4}$ vanish for similar reasons.
  \end{itemize}
\end{proof}
\begin{Rem}
  Although the structure maps $\{\lambda_k\}_{k=1}^3$ on $C^\infty(\M_E)$ as above depends a priori on a choice of dissection $\Psi$ of $E$ and a triple $(j,\nabla^F, \nabla^B)$, it is natural to expect that different choices would lead to isomorphic derived Poisson algebra structures of degree $(-2)$. More precisely, under assumptions in Proposition~\ref{Cor: uniqueness of minimal model}, let $(\M_E, \{\widehat{\lambda}_k\}_{k \geqslant 1})$ be the degree $(-2)$ shifted derived Poisson algebra obtained from another dissection $\hat{\Psi}$ and another triple $(\hat{j}, \hat{\nabla}^F, \hat{\nabla}^B)$. We conjecture that there exists an isomorphism of derived Poisson algebras of degree $(-2)$ from $(\M_E, \{\lambda_k\}{k \geqslant 1})$ to $(\M_E, \{\widehat{\lambda}_k\}_{k \geqslant 1})$, whose first map is given by the isomorphism $\Delta^\ast$ in~\eqref{Eq: Def of Thetaast}.
\end{Rem}

Finally, we give a geometric interpretation of the cohomology class $[d_T] \in H^3_{\CE}(A_E; B^\ast[2])$.
Given a leaf $L$ of the characteristic distribution $F \subseteq TM$ and a dissection $\Psi$ of $E$, according to Proposition~\ref{Prop:CSX 2.1}, the structure of the restricted Courant algebroid $E\!\!\mid_L$, which is obviously transitive, is completely determined by the bundle $(\G, [-,-]^\G, g^\G)$ of quadratic Lie algebras and the maps $(R, H,\nabla^\G)$ as in Equations~\eqref{Eq:F-conn on G},\eqref{Eq:FFtoG},  and~\eqref{Eq:FFtoFast}.
In fact, the cohomology class $[d_T]$ is the obstruction to local trivialization of $E$ around each leaf $L$ in the following sense:
\begin{thm}\label{Thm: geometric meaning of dT}
  The cohomology class $[d_T] \in H^3_{\CE}(A_E; B^\ast[2])$ vanishes if and only if, near each leaf $L$ of the characteristic distribution $F$, there exists a tubular neighbourhood $\widetilde{L} \cong L \times N$ of $L$ in $M$ such that the Courant algebroid $E\!\!\mid_{\widetilde{L}}$ is isomorphic to the cross product of $E\!\!\mid_L$ and the zero Courant algebroid over $N$. In other words, $E\!\!\mid_{\widetilde{L}}$ is completely determined by the transitive Courant algebroid $E\!\!\mid_L$.
\end{thm}
\begin{proof}
  First of all, note that the cohomology class $[d_T] \in H^3_{\CE}(A_E; B^\ast[2])$ vanishes if and only if given any dissection $\Psi$ of $E$, there exist a splitting $j$ of the short exact sequence~\eqref{Eq:SES of VB} and metric $B$-connection $\nabla^B$ on $\G$ such that the associated extended metric $B$-connection $\nabla^B$ on $E$ as in Equation~\eqref{Eq: extended B-connection} satisfies the following conditions:
  \begin{compactenum}
    \item The Cartan $3$-form $C^\G$ of the bundle $\G$ of quadratic Lie algebras is parallel along the subbundle $j(B) \subseteq TM$, i.e., $\nabla^B(C^\G) = 0$;
    \item $H \in \Gamma(\wedge^3 F^\ast)$ is parallel along the subbundle $j(B)$ in the following sense:
        \[
          j(b)H(x,y,z) = H(\nabla^B_b x,y,z) + H(x, \nabla^B_b y,z) + H(x, y, \nabla^B_b z),
        \]
        for all $b \in \Gamma(B)$ and $x,y,z \in \Gamma(F)$;
    \item $g^\G(R(-,-), -) \in \Gamma(\wedge^2 F^\ast \otimes \G^\ast)$ is parallel along the subbundle $j(B)$, i.e.,
    \[
    j(b)g^\G(R(x,y), r) = g^\G(R(\nabla^B_b x,y), r) + g^\G(R(x, \nabla^B_b y), r) + g^\G(R(x,y), \nabla^B_b r),
     \]
     for all $b \in \Gamma(B)$, $x,y \in \Gamma(F)$, and $r \in \Gamma(\G)$;
    \item the $(1,1)$-curvature of the linear connection on $\G$ vanishes, i.e.,
    \[
       R^{\nabla}(x,b)r := \nabla_x^\G \nabla_b^B r - \nabla^B_b \nabla^\G_x r - \nabla^\G_{\pr_F([x,j(b)])} r - \nabla^B_{\pr_B([x,j(b)])} r = 0,
    \]
    for all $x \in \Gamma(F), b \in \Gamma(B)$ and $r \in \Gamma(\G)$.
  \end{compactenum}
   Let $\widetilde{L} \subset M$ be a tubular neighbourhood of the leaf $L$ such that for any point $q \in \widetilde{L}$, there exists a point $p \in L$ and a smooth curve $\gamma(t) \colon [t_0, t_1] \to \widetilde{L}$ satisfying $\gamma(t_0) = q, \gamma(t_1) = p$ and $\gamma^\prime(t) = j(b)(\gamma(t))$ for some $b \in \Gamma(B)$.

  Given such a metric $B$-connection $\nabla^B$ on $E$ as above, consider the parallel transport $P_{t_0t_1}^\gamma \colon E_{\gamma(t_0)} \to E_{\gamma(t_1)}$ between the fibres of $E$ along $\gamma(t)$. It is clear that $P_{t_0t_1}^\gamma$ is compatible with the anchor map $\rho = \pr_F$, i.e.,
  \[
     P_{t_0t_1}^\gamma(\rho(\xi_0+r_0+x_0)) = \rho(P_{t_0t_1}^\gamma(\xi_0+r_0+x_0)).
  \]
  Meanwhile, since $\nabla^B$ is compatible with the pseudo-metric on $E$, it follows that $P_{t_0t_1}^\gamma$ is an isometry.
  Finally, using the above conditions $(1)-(4)$ for $\nabla^B$ and the explicit expressions for the Dorfman bracket $\circ$ on $E \cong F^\ast \oplus \G \oplus F$ in Proposition~\ref{Prop:CSX 2.1}, one can show that
  \[
    P_{t_0t_1}^\gamma( (\xi_0+r_0+x_0) \circ (\xi_0^\prime + r_0^\prime + x_0^\prime)) = P_{t_0t_1}^\gamma(\xi_0 +r_0+x_0) \circ P_{t_0t_1}^\gamma(\xi_0^\prime + r_0^\prime + x_0^\prime),
  \]
  for all vectors $\xi_0+r_0+x_0, \xi_0^\prime+r_0^\prime+x_0^\prime \in E_{\gamma(t_0)}$.
  Hence, the restriction $E\!\!\mid_{\widetilde{L}}$ of $E$ on this tubular neighbourhood $\widetilde{L}$ is completely determined by its restriction $E\!\!\mid_L$ on the leaf $L$. This completes the proof.
\end{proof}
\begin{Rem}
The geometric meaning of $[d_T] $ is reminiscent of the characteristic class $[\omega]$ of a regular Lie algebroid introduced by Gracia-Saz and Mehta~\cite{GsM}.
In fact, when we consider the regular Courant algebroid $E = A \oplus A^\ast$ arising from a regular Lie algebroid $A$,  the cohomology class $[d_T]$ of $E$ specializes to the cohomology class $[\omega]$ of $A$ (see Lemma~\ref{Lem:transgression of A}).
Thus, the cohomology class $[d_T]$ of regular Courant algebroids could be viewed as the generalization of the class $[\omega]$ of regular Lie algebroids.
\end{Rem}

\subsection{The standard cohomology}
We now compute the cohomology of the minimal model $\mathcal{M}_E$, which according to Theorem~\ref{Thm: semifull contraction for Rothstein algebra}, is isomorphic to the standard cohomology of $E$.
For this, we introduce a non-positively filtered and decreasing filtration of the algebra $C^\infty(\mathcal{M}_E)$:
\begin{equation}\label{Eqt:Filtration}
C^\infty(\mathcal{M}_E) \supset \cdots  \supset F^{p} C^\infty(\mathcal{M}_E) \supset F^{p+1}C^\infty(\mathcal{M}_E)  \supset \cdots  \supset F^{0}C^\infty(\mathcal{M}_E) =C_{\CE}^\bullet(A_E),
\end{equation}
where  for each \textit{non-positive} integer $p$, we set
\[
(F^{p}C^\infty(\mathcal{M}_E))^\bullet  := \bigoplus_{k+2l = \bullet, 0\leqslant l \leqslant  -p}C^k(A_E; S^{l}(B[-2])) = \bigoplus_{k+2l = \bullet, 0\leqslant l \leqslant  -p}\Gamma(\wedge^kA^\ast_E \otimes S^{l}(B[-2])).
\]
Since for any fixed integer $n \geqslant  0$, the filtration \eqref{Eqt:Filtration} stabilizes to a sequence of finite terms:
\[
C^\infty(\mathcal{M}_E)^n = F^{-\lfloor n/2\rfloor} C^\infty(\mathcal{M}_E)^n \supset \cdots \supset F^{0}C^\infty(\mathcal{M}_E)^n =  C^n(A_E),
\]
it follows that the filtration $\{F^{\bullet} C^\infty(\M_E)\}$ as defined above is bounded. Therefore, the corresponding spectral sequence converges to the cohomology of the minimal model $\mathcal{M}_E$, or
the standard cohomology of $E$. We analyze this spectral sequence  from its $0$-th sheet $(E_0,d_0)$:
\[
E^{p,q}_0:=F^{p}C^\infty(\mathcal{M}_E)^{p+q} /F^{p+1} C^\infty(\mathcal{M}_E)^{p+q}= C^{q+3p}(A_E ; S^{-p} (B[-2])),
\]
which lands in the second quadrant for the indices $(p,q)$ with $p \leqslant  0$ and $0\leqslant q+3p\leqslant  \mathrm{rank}(A_E)$, and the $0$-differential is given by
\[
d^{p,q}_0= \dstone \colon E^{p,q}_0 = C^{q+3p}(A_E; S^{-p} (B[-2])) \to C^{q+3p+1}(A_E; S^{-p} (B[-2])) = E^{p,q+1}_0\,.
\]
By taking the  cohomology of $d_0$, we see that the bi-graded vector space of the first sheet is given by
\[
E^{p,q}_1 = H^{q+3p}_{\CE}\bigl(A_E; S^{-p}(B[-2])\bigr).
\]
By Lemma~\ref{Lem:dst on B} and Proposition~\ref{Lem:Transgression}, the differential $d_1$ of the first sheet, which we call the \textbf{transgression map}, is specified by the characteristic class $[d_T] \in H^3_{\CE}(A_E; B^\ast[2])$,
\[
d^{p,q}_1= [d_T] \colon E^{p,q}_{1} =  H^{q+3p}_{\CE}(A_E; S^{-p}(B[-2])) \to  H^{q+3p+3}_{\CE}(A_E; S^{-p-1}(B[-2])) =  E^{p+1,q}_{1}.
\]
\begin{prop}
  If the cohomology class $[d_T] \in H^3_{\CE}(A_E; B^\ast[2])$ of $E$ vanishes, then the standard cohomology of $E$ is isomorphic to the Chevalley-Eilenberg cohomology of the $A_E$-module $S(B[-2])$
  \[
  H^n_{\operatorname{st}}(E) \cong \bigoplus_{k+2l = n}H^k_{\CE}(A_E; S^l(B[-2])).
  \]
\end{prop}
\begin{proof}
  By Theorem~\ref{Thm: geometric meaning of dT}, for each fixed dissection of $E$, one can choose a splitting $j$ of the short exact sequence~\eqref{Eq:SES of VB} and a metric $B$-connection $\nabla^B$ on $\G$ such that the linear map $d_T$ defined in~\eqref{Eq: Def of dT} vanishes. In this case, by Lemma~\ref{Lem:dst on B} the standard differential $d_E$ of $E$ becomes the Chevalley-Eilenberg differential $d_{\CE}$ of the $A_E$-module $S(B[-2])$. The rest is clear.
\end{proof}
Now we investigate the second sheet $E_2^{p,q}$ of this spectral sequence.
The bi-graded vector space of the second sheet is given by
\[
E^{p,q}_{2}\cong H^{p+q}\bigl(H_{\CE}^{q+3p}(A_E; S^{-p}(B[-2])), [d_T]\bigr).
\]
The differential on the second sheet $d_2 \colon E_2^{p,q} \to E_2^{p+2,q-1}$ can be described accordingly.
In fact, any element $\overline{\varphi} \in E_2^{p,q}$ is represented by a class $[\varphi] \in H_{\CE}^{q+3p}(A_E; S^{-p}(B[-2]))$ satisfying
\[
[d_T]([\varphi]) = 0 \in H^{q+3p+3}_{\CE}(A_E, S^{-p-1}(B[-2]),
\]
it follows that there exists an element $\psi \in C^{q+3p+2}(A_E; S^{-p-1}(B[-2]))$ such that $d_{\CE}(\psi) = d_T(\varphi)$. Then the element $d_T(\psi)$ is $d_{\CE}$-closed and its cohomology class $[d_T(\psi)]$ is $[d_T]$-closed. Thus,
$d_T(\psi)$ defines an element in $E_2^{p+2,q-1}$, which is exactly the image of $\overline{\varphi}$ under $d_2$, i.e.,
\[
 d_2(\overline{\varphi}) = \overline{d_T(\psi)} \in H^{p+q+1} \bigl(H^{q+3p+5}_{\CE}(A_E; S^{-p-2}(B[-2])), [d_T] \bigr) \cong E_2^{p+2,q-1}.
\]
We summaries these facts into a theorem.
\begin{thm}\label{Thm:regulartostandard}
Let $E$ be a regular Courant algebroid with characteristic distribution $F$,  normal bundle $ B=TM/F$ and ample Lie algebroid $A_E$.
There exists a spectral sequence $(E_r^{p,q},d_r)$ which converges to the standard cohomology of   $E$:
 \[
 	E^{p,q}_0 = C^{q+3p}(A_E ; S^{-p} (B[-2])) \Rightarrow H^{p+q}(\M_\E) \cong H^{p+q}_{\mathrm{st}}(E).
 \]
\end{thm}
Let us call $(E_r^{p,q},d_r)$ the \textbf{Chevalley-Eilenberg-to-standard spectral sequence}.
It is easy to see that $d_r\colon E^{p,q}_r\to E^{p+r,q-r+1}_{r}$ vanishes for every $r \geqslant 2$ by degree reasons when the indices $(p,q)$ are among $(0,0)$, $(0,1)$, $(0,2)$, $(0,3)$, $(-1,3)$ and $(-1,4)$, i.e. cases of $p+q\leqslant 3$. So we draw some  direct corollaries  about the standard cohomology in degrees $\leqslant 3$.
\begin{cor}\label{CorH0H1}\cite{SX2008}
  For any regular Courant algebroid $E$ with ample Lie algebroid $A_E$, we have
  \[
      H^0_{\operatorname{st}}(E) \cong H^0_{\CE}(A_E) \cong H^0_{\operatorname{naive}}(E), \quad \mbox{and} \quad
       H^1_{\operatorname{st}}(E) \cong H^1_{\CE}(A_E) \cong H^1_{\operatorname{naive}}(E).
  \]
\end{cor}
\begin{cor}\label{CorH2}
The standard cohomology in degree $2$, consisting of outer derivations of the algebra $C^\infty(\E)$, is given by
\[
  H^2_{\operatorname{st}}(E) \cong H^2_{\CE}(A_E) \bigoplus \ker\big([d_T] \colon H^0_{\CE}(A_E; B[-2]) \to H^3_{\CE}(A_E)\big).
\]
\end{cor}
Recall that the space of isomorphism classes of infinitesimal deformations of any regular Courant algebroid $E$ is isomorphic to $H^3_{\operatorname{st}}(E)$ (cf.~\cites{Roytenberg2001, KW}). Thus, we conclude the following fact.
\begin{cor}\label{CorH3}
  The space $H^3_{\operatorname{st}}(E)$ of isomorphism classes of infinitesimal deformations of any regular Courant algebroid $E$ with characteristic distribution $F$ and ample Lie algebroid $A_E$ decomposes as the direct sum
  \[
     H^3_{\operatorname{st}}(E) \cong \frac{H^3_{\CE}(A_E)}{\img  [\Transgression]} \bigoplus  \ker\big([\Transgression] \colon {H^1_{\CE}(A_E; B[-2])} \to H^4_{\CE}(A_E)\big).
\]
\end{cor}
More precisely,  any two elements
\[
[\alpha] \in H^3_{\CE}(A_E)/\img [\Transgression] \quad \mbox{and} \quad [\beta] \in \ker \big([\Transgression] \colon {H^1_{\CE}(A_E; B[-2])} \to H^4_{\CE}(A_E)\big)
\]
 give rise to a first order deformation of the generating Hamiltonian function $\Theta$
\[
  \Theta_t := \Theta + t(\alpha + \beta),
\]
since $\{\Theta_t, \Theta_t\} = 0 \; (\!\!\!\mod t^2)$. By Equations~\eqref{Eq:derived bracket 1} and~\eqref{Eq:derived bracket 2}, the associated infinitesimal deformed anchor and Dorfman bracket are represented by
\begin{align*}
  \rho_t(e_1)f &= \{\{\Theta_t, e_1[-1]\}, f\} = \rho(e_1)f + t\{\{\alpha+\beta, e_1[-1]\}, f\} = \rho(e_1)f + t j(\beta[e_1]) f,
  \end{align*}
  and
  \begin{align*}
  &( e_1 \circ_t e_2)[-1] = \{\{\Theta_t, e_1[-1]\}, e_2[-1]\} \\
  & = ( e_1 \circ e_2)[-1] + t\bigl(  \alpha([e_1], [e_2], -) +   \nabla_{j(\beta[e_1])}e_2[-1]  -   \nabla_{j(\beta[e_2])} e_1[-1] +  g(\nabla_{\beta(-)} e_1, e_2)\bigr),
\end{align*}
where    $e_1, e_2 \in \Gamma(E)$,    $[e_1], [e_2] \in A_E \cong E / (\ker\rho)^\perp$,  $f \in C^\infty(M)$.

Finally, $H^4_{\operatorname{st}}(E)$ contains obstructions to recursive constructions of formal deformations of $E$ because the obstruction class $(\!\!\!\mod t^3)$ of the infinitesimal deformation given by $[\alpha]+[\beta]$ is determined by
\begin{align*}
[O(\alpha,\beta)]&:= [\{\alpha+\beta, \alpha+\beta\}]\\ &\in  {H^4_{\CE}(A_E)}/{\img \left([\Transgression]\colon {H^{1}_{\CE}(A_E;   B[-2])}\to H^{4}_{\CE}(A_E) \right) } \\ & \qquad \oplus \ker\left([\Transgression]\colon {H^{2}_{\CE}(A_E;   B[-2])}\to H^{5}_{\CE}(A_E)\right).
\end{align*}
Besides the two direct summands as above, the space $H^4_{\operatorname{st}}(E)$  has an extra summand $\ker(d_2 \colon E_2^{-2,6} \to E_2^{0,5})$ where
\[
E_2^{-2,6}=\ker\left([\Transgression]\colon {H^{0}_{\CE}(A_E; S^2( B[-2]))}\to H^{3}_{\CE}(A_E;   B[-2]) \right),
\]
and
\[
E_2^{0,5}=H^{5}_{\CE}(A_E)/\img \left([\Transgression]\colon {H^{2}_{\CE}(A_E;   B[-2])}\to H^{5}_{\CE}(A_E) \right).
\]
At present, we cannot say at which sheet the spectral sequence $(E_r^{p,q},d_r)$ will degenerate.  However, if it happens at the second sheet, for example when the ample Lie algebroid $A_E$ is of rank $\leqslant 4$, then we get a nice formula describing the standard cohomology
\[
	H_{\operatorname{st}}^n(E) \cong \bigoplus_{k+2l=n} \frac{\ker [\Transgression] \colon H_{\CE}^{k}\big(A_E; S^{l}(B[-2])\big) \to H_{\CE}^{k+3}\big(A_E; S^{l-1}(B[-2])\big)}{\img  [\Transgression] \colon H_{\CE}^{k-3}\big(A_E; S^{l+1}(B[-2])\big) \to H_{\CE}^{k}\big(A_E; S^{l}(B[-2])\big)},
\]
for all $n\geqslant 0$.

\section{Applications}\label{Sec:Applications}
\subsection{Courant algebroids with split base}
Let $M= L \times N$ be a product of smooth manifolds. For each $n \in N$, there is an inclusion map $I_n \colon L \to M$ defined by $I_n(l) = (l,n) \in M$ for all $l \in L$. The differentials $(I_n)_{\ast l} \colon T_l L \to T_{(l,n)}M$ of all these inclusions determine a subbundle of $TM$, denoted by $ I_\ast(TL)$, with $\Gamma(I_\ast(TL)) \cong \Gamma(TL) \otimes_{C^\infty(L)} C^\infty(M)$.
Similarly, the differentials $(J_l)_{\ast n} \colon T_n N \to T_{l,n}M$ of the inclusions $J_l \colon N \to M$ defined by $J_l(n) = (l,n)$ determine a subbundle of $TM$, denoted by $J_\ast(TN)$, with $\Gamma(J_\ast (TN)) \cong \Gamma(TN) \otimes_{C^\infty(N)} C^\infty(M)$.
In other words, we have established two isomorphisms of vector bundles $I_\ast(TL) \cong \pr_L^\ast TL$ and $J_\ast(TN) \cong \pr_N^\ast TN$, where $\pr_L \colon M \to L$ and $\pr_N\colon M \to N$ are projections.
Meanwhile, the tangent bundle $TM$ of the product manifold $M$ is identified with the direct sum $I_\ast (TL) \oplus J_\ast(TN)$ of the two subbundles.

A regular Courant algebroid $(E,g,\rho,\circ)$ is said to be with \textit{split base}, if the base manifold $M$ is a product manifold $L \times N$, and the characteristic distribution of $E$ is the subbundle $F = I_\ast(TL)$ of $TM$.
In this case, Ginot and Gr\"{u}tzmann have introduced in \cite{GM2009} what they call the \textit{naive ideal spectral sequence}. Let us denote it by $(\widetilde{E}_s^{k,l},\tilde{d}_s)$ in order to distinguish from our  Chevalley-Eilenberg-to-standard spectral sequence $({E}_r^{p,q},d_r)$. Next we compare  the two spectral sequences.

Choose a dissection $\Psi$ of $E$ such that $E$ is identified with the standard Courant algebroid $ S(\nabla^\G, R, H)$.
We then fix a triple $(j,\nabla^F,\nabla^B)$ compatible with the split base, that is, $j$ is the splitting such that $j(B)$ is identified with the subbundle $J_\ast(TN) \subseteq TM$, and $\nabla^B$ is the pullback connection of a metric $TN$-connection on $(\G, g^\G)$ along $\pr_N$.
Note that the ample Lie algebroid $A_E$ is identified with the standard one $L(\nabla^\G, R)$ via the chosen dissection.
Then the graded manifold underlying the minimal model $\M_E$ is given by
\[
 A_E[1] \oplus B^\ast[2] \cong \G[1] \oplus F[1] \oplus B^\ast[2] \cong \G[1] \oplus \pr_L^\ast (T[1]L) \oplus \pr_N^\ast (T^\ast[2]N).
\]
The flat Lie algebroid $A_E=L(\nabla^\G, R)$-connection on $B\cong \pr_N^\ast(TN) \cong J_\ast(TN)$ is simply given by
\[
\widetilde{\nabla}_{r+x}(f \otimes v) :=\nabla^{\operatorname{Bott}}_x (f\otimes v) = x(f) \otimes v,
\]
for all $x \in \Gamma(F),  r \in \Gamma(\G), f\in C^\infty(M)$, and $v \in \Gamma(TN)$.
Therefore, the Chevalley-Eilenberg-to-standard spectral sequence starts from
\[
E_0^{p,q}=C^{q+3p}(A_E ; S^{-p} (B[-2]))=\Gamma(\wedge^{3p+q}A_E^*)\otimes_{C^\infty(N)} \Gamma(S^{-p}(T[-2]N)),
\quad d_0=\dCE\otimes 1
\]
and then
\[
E_1^{p,q} \cong H_{\CE}^{q+3p}\big(A_E; S^{-p}(B[-2])\big) \cong H_{\CE}^{q+3p}(A_E) \otimes_{C^\infty(N)} \Gamma(S^{-p} (T[-2]N) ).
\]
The differential $d_1 = [\Transgression] \in H^3_{\CE}(A_E; B^\ast[2]) \cong H^3_{\CE}(A_E; T^\ast[2]N)$ can be viewed as a $C^\infty(N)$-linear map
\[
[d_T] \colon \Gamma(T[-2]N) \to H^3_{\CE}(A_E).
\]
We are thus able to conclude that
\[
E_2^{p,q} \cong \frac{\ker [d_T] \colon H^{q+3p}_{\CE}(A_E) \otimes_{C^\infty(N)} \Gamma(S^{-p}(T[-2]N)) \to H^{q+3p+3}_{\CE}(A_E) \otimes_{C^\infty(N)} \Gamma(S^{-p-1}(T[-2]N))}{\img  [d_T] \colon H^{q+3p-3}_{\CE}(A_E) \otimes_{C^\infty(N)} \Gamma(S^{-p+1}(T[-2]N)) \to H^{q+3p}_{\CE}(A_E) \otimes_{C^\infty(N)} \Gamma(S^{-p}(T[-2]N))} .
\]
We now recall from \cite{GM2009}  the naive ideal spectral sequence $(\widetilde{E}_s^{k,l},\tilde{d}_s)$.  For $k,l\geqslant 0$ where $l$ is an even number, one has
\[
\widetilde{E}_{1}^{k,l}=\Gamma(\wedge^{k}\ker\rho)\otimes_{C^\infty(N)}  \Gamma(S^{\frac{l}{2}}(T[-2]N))= \Gamma(\wedge^{k}A_E^*)\otimes_{C^\infty(N)}  \Gamma(S^{\frac{l}{2}}(T[-2]N));
\]
When $l$ is odd,  $\widetilde{E}_1^{k,l}$ is simply zero. The differential $\tilde{d}_1\colon \widetilde{E}_{1}^{k,l}\to \widetilde{E}_{1}^{k+1,l}$ is identical to $\dCE\otimes 1$.

The spectral sequence continues with $(\widetilde{E}_{2}^{k,l},\tilde{d}_2=0)$ and
\[
  \widetilde{E}_{3}^{k,l}=H_{\CE}^{k}(A_E) \otimes_{C^\infty(N)} \Gamma(S^{{\frac{l}{2}}} (T[-2]N)),
\]
whose differential $\tilde{d}_3$ coincides with our $[d_T]$.

By Theorem \ref{Thm:regulartostandard}, we get the following conclusion.
\begin{prop}\label{Thm of GM2009}
  Let $E$ be a Courant algebroid with split base $M = L \times N$. If the Chevalley-Eilenberg-to-standard spectral sequence  $E^{p,q}_r$ degenerates at its second sheet, or equivalently the naive ideal spectral sequence $\widetilde{E}_{s}^{k,l}$ degenerates  at its fourth sheet, then
\[
H_{\mathrm{st}}^n(E) = \bigoplus_{k+2l = n}\frac{\ker [d_T] \colon H^k_{\CE}(A_E) \otimes_{C^\infty(N)} \Gamma(S^l(T[-2]N)) \to H^{k+3}_{\CE}(A_E) \otimes_{C^\infty(N)} \Gamma(S^{l-1}(T[-2]N))}{\img  [d_T] \colon H^{k-3}_{\CE}(A_E) \otimes_{C^\infty(N)} \Gamma(S^{l+1}(T[-2]N)) \to H^k_{\CE}(A_E) \otimes_{C^\infty(N)} \Gamma(S^l(T[-2]N))}.
\]
\end{prop}
\begin{Rem}
 It is claimed   by \cite{GM2009}*{Theorem 4.12} that the naive ideal spectral sequence always degenerates at its fourth page. However, the proof of this conclusion is imperfect. In fact, to compute $\widetilde{E}_4^{k,l}$, they used the following equation which is \textit{not} generally true:
 \begin{eqnarray*}
 	 && \frac{\ker [d_T] \colon H^{k}_{\CE}(A_E) \otimes_{C^\infty(N)} \Gamma(S^{\frac{l}{2}}(T[-2]N)) \to H^{k+3}_{\CE}(A_E) \otimes_{C^\infty(N)} \Gamma(S^{\frac{l}{2}-1}(T[-2]N))}{\img  [d_T] \colon H^{k-3}_{\CE}(A_E) \otimes_{C^\infty(N)} \Gamma(S^{\frac{l}{2}+1}(T[-2]N)) \to H^{k}_{\CE}(A_E) \otimes_{C^\infty(N)} \Gamma(S^{\frac{l}{2}}(T[-2]N))} \\
 	&=&  \frac{H^{k}_{\CE}(A_E)}{ (\img  [d_T])\wedge  H^{k-3}_{\CE}(A_E) }  \otimes_{C^\infty(N)} \Gamma(S^{\frac{l}{2}}(\ker [d_T])).
 \end{eqnarray*}
\end{Rem}

\subsection{Generalized exact Courant algebroids}
Let $E$ be a generalized exact Courant algebroid, i.e., a regular Courant algebroid whose bundle of quadratic Lie algebras $\G=\ker\rho/(\ker\rho)^\perp$ has rank zero. So by choosing a dissection of $E$, we have an isomorphism $E \cong F^\ast \oplus F$. Let us fix a splitting $j \colon B \to TM$ and a torsion-free $F$-connection $\nabla^F$ on $F$.

It is clear that the first sheet of the Chevalley-Eilenberg-to-standard spectral sequence for the minimal model $(\M_E = B^\ast[2] \oplus F[1], Q_E = d_{\CE} + d_T)$ of $E$ is simply the Chevalley-Eilenberg cohomology of the Bott representation of $F$ on $S(B[-2])$:
\[
E_1^{p,q} \cong H^{q+3p}_{\CE}(F, S^{-p}(B[-2])).
\]
In this case, the associated torsion part $C_\nabla$ is identified with the \u{S}evera $3$-form $C = \frac{1}{2}H \in \Gamma(\wedge^3 F^\ast)$ defined by
\[
  C(x,y,z) = \frac{1}{2}H(x,y,z) = \frac{1}{2} \langle \pr_{F^\ast}(x \circ y) \mid z \rangle = g(x \circ y, z)
\]
for all $x,y,z \in \Gamma(F)$; and the curvature part vanishes: $R^\nabla(\tilde{\rho[2]}, j(b)[-2]) = 0$ for all $b \in B = TM/F$.
Thus, the cohomology class $[d_T] \in H^3_{\CE}(A_E; B^\ast[2]) \cong H^3_{\CE}(F; B^\ast[2])$ is given by the covariant derivative $\nabla[C]$ of $[C]$ along the subbundle $j(B) \subseteq TM$, i.e.,
\begin{align*}
 &\quad \nabla[C](b[-2])(x,y,z) = j(b) C(x, y, z) - C(\nabla^B_b x, y, z) - C(x, \nabla^B_b y, z) -C(x, y, \nabla^B_b z)  \\
&= j(b) g(x \circ y, z) - g(\pr_F[j(b),x] \circ y, z) - g(x \circ \pr_F[j(b),y], z) - g(x \circ y, \pr_F[j(b),z]),
\end{align*}
for all $b  \in \Gamma(B )$ and all $x,y,z \in \Gamma(F)$.
Applying Theorem~\ref{Thm:regulartostandard}, we obtain the following fact.
\begin{thm}\label{Thm:application to GEC}
  Let $E$ be a generalized exact Courant algebroid with characteristic distribution $F \subset TM$ and \u{S}evera class $[C] \in H^3_{\CE}(F)$.
  If the corresponding Chevalley-Eilenberg-to-standard spectral sequence degenerates at its second sheet, then the standard cohomology $H_{\mathrm{st}}^\bullet(E)$ of $E$ is given by
\[
 H^n_{\mathrm{st}}(E) \cong  \bigoplus_{k+2l=n} \frac{\ker  \nabla [C] \colon H_{\CE}^{k}\big(F; S^{l}(B[-2])\big) \to H_{\CE}^{k+3}\big(F; S^{l-1}(B[-2])\big)}{\img   \nabla[C] \colon H_{\CE}^{k-3}\big(F; S^{l+1}(B[-2])\big) \to H_{\CE}^{k}\big(F; S^{l}(B[-2])\big)}.
\]
In particular, we have
\[
 H^3_{\operatorname{st}}(E) \cong \frac{H^3_{\CE}(F)}{\img  \nabla[C]} \bigoplus \ker \left(\nabla [C] \colon {H^1_{\CE}(F; B[-2])} \to H^4_{\CE}(F) \right).
\]
\end{thm}

\begin{example}\label{Example:irrational alpha}
  Let $\mathbb{T}^2=\R^2/(\Z\times \Z)$ be the smooth $2$-torus and $F \subset T\mathbb{T}^2$ any integrable distribution of rank $1$. Consider the generalized exact Courant algebroid $E = F^\ast \oplus F$ over $\mathbb{T}^2$. In this case, the \u{S}evera class $[C] \in H^3_{\CE}(F)$ vanishes since $H^3_{\CE}(F) = 0$.
  Hence, the standard cohomology of $E$ is given by
  \[
    H^n_{\operatorname{st}}(E) = \begin{cases}
                                         H^0_{\CE}\big(F, S^m((T\mathbb{T}^2/F)[-2])\big) \cong \ker d_F^{\operatorname{Bott}}\mid_{S^m((T\mathbb{T}^2/F)[-2])}, & n=2m; \\
                                         H^1_{\CE}\big(F, S^m((T\mathbb{T}^2/F)[-2])\big) \cong H^1_{\CE}(F) \otimes_{\mathbb{R}} H^0_{\CE}(F, S^m((T\mathbb{T}^2/F)[-2])),
                                         & n=2m+1.
                                         \end{cases}
  \]
  Assume that $F$ is generated by a vector $(1,\nu) \in \R^2$, which foliates into a submanifold $\mathcal{F} \subset \mathbb{T}$. Then one gets
  \[
    H^1_{\CE}(F) \cong H^1_{\operatorname{dR}}(\mathcal{F}, \R) = \begin{cases}
                                   \R, & \mbox{if $\nu$ is rational}, \\
                                   0, & \mbox{if $\nu$ is irrational}.

    \end{cases}
  \]
  In particular, we have $H^3_{\operatorname{st}}(E) = 0$ if $\nu$ is irrational, and thus the Courant algebroid structure on $E$ is rigid up to isomorphisms.

  When $\nu$ is rational, any class $[\beta] \in H_{\operatorname{st}}^3(E) \cong H_{\CE}^1(F, B[-2])$ determines a deformation of the anchor map by
  \[
    \rho_\beta(x+\xi) = x + \beta(x),
  \]
  for all $x \in \Gamma(F), \xi \in \Gamma(F^\ast)$.
  The Dorfman bracket is rigid in both rational and irrational cases, whose components are given by
  \begin{align*}
     x \circ \xi =  L_x \xi  & \qquad \mbox{ and } \qquad \xi \circ x  = -L_x \xi + d_F \langle \xi \mid x \rangle,
  \end{align*}
  for all $x \in \Gamma(F), \xi \in \Gamma(F^\ast)$.
  Hence, the infinitesimal formal deformation of the Courant algebroid $E = F \oplus F^\ast$ is equivalent to that of the rank $1$ distribution $F \subseteq TM$.
\end{example}

Complex Courant algebroid are defined analogously to (real) Courant algebroid in Definition~\ref{Def of Courant algebroid}, except that the pseudo-metric $g$ is $\C$-valued, the anchor is $TM \otimes \C$-valued, and $\C$-linearity replaces $\R$-linearity;  see~\cite{GS} for the precise definition.
Our main results and construction also work  for complex regular Courant algebroids. Below is an example of complex generalized exact Courant algebroid.
\begin{Ex}
	Given a complex manifold $X$ and a $\bar{\partial}$-closed $(0,3)$-form $H^{0,3}$, there is a complex Courant algebroid structure on $C_X^{0,1} = T_X^{0,1} \oplus (T_{X}^{0,1})^\ast$ twisted by $H^{0,3}$. In this case, the cohomology class $[d_T] \in H_{\operatorname{Dol}}^3(X, (T_X^{1,0})^\ast[2])$ is represented by $\partial H^{0,3}$ via an obvious degree shifting.
If the corresponding Chevalley-Eilenberg-to-standard spectral sequence degenerates at its second sheet,
then the standard cohomology of $C_X^{1,0}$ is given by
	\[
	H^n(C^{0,1}_X)  \cong  \bigoplus_{k+2l=n} \frac{\ker  \partial H^{0,3} \colon H_{\operatorname{Dol}}^{k}\big(X; S^{l}(T_X^{1,0}[-2])\big) \to H_{\operatorname{Dol}}^{k+3}\big(X; S^{l-1}(T_X^{1,0}[-2])\big)}{\img   \partial H^{0,3} \colon H_{\operatorname{Dol}}^{k-3}\big(X; S^{l+1}(T^{1,0}_X[-2])\big) \to H_{\operatorname{Dol}}^{k}\big(X; S^{l}(T_X^{1,0}[-2])\big)}.
	\]
	In particular, if $H^{0,3}$ is also $\partial$-closed, then we have
	\[
	H^n(C^{0,1}_X)  \cong \bigoplus_{k+2l=n} H_{\operatorname{Dol}}^{k}\big(X; S^{l}(T_X^{1,0}[-2])\big).
	\]
\end{Ex}

\subsection{Regular Courant algebroids arising from regular Lie algebroids}
Let $(A, \rho_A, [-,-]_A)$ be a regular Lie algebroid over $M$ with characteristic distribution $F=\rho_A(A)\subseteq TM$. Then there exists two short exact sequences of vector bundles over $M$
\begin{equation}\label{Eq:SES of A}
  0 \to K:= \ker(\rho_A) \hookrightarrow A \xrightarrow{\rho_A} F \to 0,
\end{equation}
and
\[
0 \rightarrow F \xrightarrow{i} TM \xrightarrow{\pr_B} B:= TM/F \rightarrow 0.
\]
Consider the vector bundle $E := A^\ast \oplus A$ equipped with a pseudo-metric $\langle -,- \rangle$ defined by
\[
\langle \xi+x,\eta+y\rangle  := \frac{1}{2}\langle \xi \mid y \rangle + \frac{1}{2} \langle \eta \mid x\rangle,
\]
for all $\xi,\eta \in \Gamma(A^\ast)$ and $x,y \in \Gamma(A)$. There is a regular Courant algebroid structure on $E$, whose anchor is given by
\[
 \rho(\xi+x) = \rho_A(x) \in \Gamma(F) \subseteq \Gamma(TM),
\]
and whose Dorfman bracket is given by
\[
 (\xi + x) \circ (\eta + y) = L_x \eta - \iota_y d_A(\xi) + [x,y]_A.
\]
Here $d_A$ is the Chevalley-Eilenberg differential of the Lie algebroid $A$, and $L$ is the associated Lie derivative.
Moreover, we have
\[
\G = \ker\rho/(\ker\rho)^\perp \cong (A^\ast/\rho_A^\ast(F^\ast)) \oplus K.
\]
The following lemma follows from direct verifications.
\begin{lem}\label{Lem:dissection for AAast}
  Each splitting $\tau \colon F \to A$  (or equivalently a projection $\pr_K \colon A \to K$) of the short exact sequence~\eqref{Eq:SES of A} determines an identification $\G \cong K^\ast \oplus K$, and thus a dissection $\Psi$ of the regular Courant algebroid $E$
  \[
   \Psi \colon (F^\ast \oplus \G \oplus F, \langle -, - \rangle) \to (A^\ast \oplus A, \langle -, - \rangle).
  \]
  Moreover,
  \begin{itemize}
  \item the associated map $H $ is trivial, and $\nabla^\G, R$ are given respectively by
  \begin{align*}
    \nabla^\G_{x_F} (z+\zeta) &= L_{\tau(x_F)}(z+\pr_K^\ast(\zeta)) = [\tau(x_F), z]_A + L_{\tau(x_F)}\pr_K^\ast(\zeta), \\
    R(x_F,y_F) &=  [\tau(x_F), \tau(y_F)]_A - \tau([x_F, y_F]),
  \end{align*}
  for all $x_F, y_F \in \Gamma(F), z \in \Gamma(K), \zeta \in \Gamma(K^\ast)$;
  \item the ample Lie algebroid $A_E$ is naturally identified with the abelian extension $A \ltimes K^\ast$ of the Lie algebroid $A$ along the $A$-module $K^\ast$.
  \end{itemize}
\end{lem}
Let us fix a splitting $\tau \colon F \to A$ (equivalently   $ \pr_K \colon A \to K$) of the short exact sequence~\eqref{Eq:SES of A} in the sequel and choose a triple $(j,\nabla^F,\nabla^B)$.
The Rothstein algebra of $E = A \oplus A^\ast$ reads
\[
C^\infty(\E) = C^\infty(T^\ast[2]M \oplus A[1] \oplus A^\ast[1]) \cong \Gamma(\widehat{S}(T[-2]M)  \otimes \wedge^\bullet A \otimes \wedge^\bullet A^\ast) .
\]
The underlying graded manifold $B^\ast[2] \oplus A_E[1]$ of the minimal model $\M_E$ is identified with $B^\ast[2] \oplus A[1] \oplus K^\ast[1]$.
Note that the $A_E$-module structure on $B$ is obtained from the trivial extension of the Bott $A$-module structure on $B$. As a consequence, the first sheet of the Chevalley-Eilenberg-to-standard spectral sequence is clear:
\[
    E_1^{p,q} \cong H_{\CE}^{q+3p}\big(A_E; S^{-p}(B[-2])\big) \cong \bigoplus_{m+n=q+3p} H^m_{\CE}\big(A; \wedge^n K \otimes S^{-p}(B[-2])\big).
\]

We now investigate the associated cohomology class $[d_T] \in H_{\CE}^3(A_E; B^\ast[2])$, viewed as a map
\[
[\Transgression] \colon \bigoplus_{n,m}H_{\CE}^n(A_E; S^m(B[-2])) \to \bigoplus_{n,m}H_{\CE}^{n+3}(A_E; S^{m-1}(B[-2])).
\]
Note that the associated metric connection $\nabla$ on $E \cong F^\ast \oplus \G \oplus F$ induces a linear connection $\nabla^A$ on $A$ defined by for all $x \in \Gamma(TM), a \in \Gamma(A)$,
\begin{equation}\label{Eq:Def of nablaA}
\nabla^A_x a = \nabla^B_{x_B} \pr_K(a) + [\tau(x_F), \pr_K(a)]_A + \tau(\nabla^F_{x_F} \rho_A(a) + \pr_F[j(x_B), \rho_A(a)]) + R(x_F, \rho_A(a)),
\end{equation}
where $x_F$ and $x_B$ are components of $x$ in $\Gamma(F)$ and $\Gamma(B)$, respectively.
Then we have the basic $A$-connection $\nabla^{\bas}$ on $TM$ and the basic curvature $R^{\bas}_{\nabla^A} \in C^2(A; \Hom(TM,A))$ induced by $\nabla^A$ (cf.~\cite{AC}):
   \[
     \nabla^{\bas}_a(u) := \rho_A(\nabla^A_u a) + [\rho_A(a), u],
   \]
  and
	 	\begin{align*}
 {R^{\bas}_{\nabla^A}}(u[-1])(a^\prime, a^{\prime\prime}) & := \nabla^A_{u}[a^\prime, a^{\prime\prime}]_A - [\nabla^A_u a^\prime, a^{\prime\prime}]_A - [a^\prime, \nabla^A_u  a^{\prime\prime}]_A - \nabla^A_{\nabla^{\bas}_{a^{\prime\prime}}u}a^\prime +\nabla^A_{\nabla^{\bas}_{a^\prime} u} a^{\prime\prime},
	 	\end{align*}
for all $a, a^\prime, a^{\prime\prime} \in \Gamma(A)$ and $u \in \Gamma(TM)$.

According to~\cite{GsM}, each regular Lie algebroid $A$ carries a canonical cohomology class
\[
[\omega] \in H^2_{\CE}(A;\Hom(B,K)),
\]
 which is represented by the restriction of the basic curvature $(-R^{\bas}_{\nabla^A})$. Explicitly, substituting the defining Equation~\eqref{Eq:Def of nablaA} of $\nabla^A$ into the above formula of $R^{\bas}_{\nabla^A}$,  we have
  \begin{align}\label{Eq:basic curvature}\nonumber
     &R^{\bas}_{\nabla^A}(j(b))(a,a^\prime) \\
     =& \nabla^B_{b}[\pr_K(a), \pr_K(a^\prime)] - [\nabla^B_b \pr_K(a), \pr_K(a^\prime)] - [\pr_K(a), \nabla^B_b \pr_K(a^\prime)] \notag \\
     &\quad  - R^{\nabla^A}(\rho_A(a), b)\pr_K(a^\prime) + R^{\nabla^A}(\rho_A(a^\prime), b)\pr_K(a) \notag \\
     &\quad\qquad + \nabla^B_b R(\rho_A(a), \rho_A(a^\prime)) - R(\pr_F[j(b), \rho_A(a)], \rho_A(a^\prime)) - R(\rho_A(a), \pr_F[j(b), \rho_A(a^\prime)]),
  \end{align}
  for all $b \in \Gamma(B)$ and $a, a^\prime \in \Gamma(A)$.
\begin{lem}\label{Lem:transgression of A}
  Under the isomorphism
  \[
      H^{3}_{\CE}(A_E; B^\ast[2]) \cong H^{3}_{\CE}(A \ltimes K^\ast; B^\ast[2]) \cong H^2_{\CE}(A; \Hom(B,K)),
  \]
  the cohomology class $[\Transgression]$ of the regular Courant algebroid $E =A \oplus A^\ast$ is identified with the cohomology class $[\omega]$ of the regular Lie algebroid $A$.
\end{lem}
\begin{proof}
Recall that the map $d_T \colon \Gamma(B[-2]) \to C^3(A_E)$ is defined by
\[
 d_T(b[-2]) = R^\nabla(\rhothree[2], j(b)[-2]) - \nabla_{j(b)}C_\nabla.
\]
Here $R^\nabla(\rhothree[2],j(b)[-2]) \in \Gamma(F^\ast \otimes K^\ast \otimes K) \subset C^3(A_E)$ is given by for all $x_F \in \Gamma(F), z_K \in \Gamma(K)$,
\begin{align*}
 R^\nabla(\rhothree[2],j(b)[-2])(\tau(x_F), z_K) &= R^\nabla(x_F, j(b))z_K  = R^{\nabla^A}(x_F, j(b))z_K \qquad \text{by Eq.~\eqref{Eq:basic curvature}} \\
 &= -R^{\bas}_{\nabla^A}(j(b))(\tau(x_F), z_K).
\end{align*}
Meanwhile, by Lemma~\ref{Lem:Cnabla} and Lemma~\ref{Lem:dissection for AAast}, the covariant derivative $\nabla_{j(b)}C_\nabla \in C^3(A_E)$ of the torsion has two components
\[
\gamma_1(b) \in \Gamma(\wedge^2 F^\ast \otimes K), \qquad \mbox{and} \qquad \gamma_2(b) \in \Gamma(\wedge^2 K^\ast \otimes K),
\]
which are respectively given by
\begin{align*}
 &\quad \gamma_1(b)(\tau(x_F), \tau(y_F)) = \nabla^B_b R(x_F,y_F) - R(\nabla^B_b x_F, y_F) - R(x_F, \nabla^B_b y_F) \\
 &= \nabla^B_b R(x_F, y_F) - R(\pr_F[j(b),x_F], y_F) - R(x_F, \pr_F[j(b), y_F]) \quad \text{by Eq.~\eqref{Eq:basic curvature}} \\
 &= R^{\bas}_{\nabla^A}(j(b))(\tau(x_F),\tau(y_F)),
\end{align*}
and
\begin{align*}
  \gamma_2(b)(a_K, a^\prime_K) &= \nabla^B_b [a_K, a_K^\prime]_A - [\nabla^B_b a_K, a_K^\prime]_A - [a_K, \nabla_b^B a_K^\prime]_A \quad \text{by Eq.~\eqref{Eq:basic curvature}}\\
  &= R^{\bas}_{\nabla^A}(j(b))(a_K, a_K^\prime).
\end{align*}
Hence, we have
\[
  d_T(b[-2]) = -R_{\nabla^A}^{\bas}(j(b)) \in C^2(A; K) \subset C^3(A_E),
\]
which concludes the proof.
\end{proof}

Applying Theorem~\ref{Thm:regulartostandard} to $E = A \oplus A^\ast$ and using Lemma~\ref{Lem:transgression of A}, we obtain
\begin{thm}\label{Thm:Regular Lie algebroid}
  Let $(A, \rho_A, [-,-]_A)$ be a regular Lie algebroid with characteristic distribution $F \subseteq TM$ and characteristic bundle $K := \ker \rho_A$ of Lie algebras.
  If the corresponding  Chevalley-Eilenberg-to-standard spectral sequence degenerates at its second sheet,
  then the standard cohomology $H_{\operatorname{st}}^n(E) $ of the regular Courant algebroid $E = A \oplus A^\ast$ is given by
  \[
    \bigoplus_{k+l+2m=n} \frac{\ker [\omega] \colon H^k_{\CE}\big(A; \wedge^l K \otimes S^m(B[-2])\big) \to H^{k+2}_{\CE}\big(A; \wedge^{l+1}K \otimes S^{m-1}(B[-2])\big)} {\img  [\omega] \colon H^{k-2}_{\CE}\big(A; \wedge^{l-1} K \otimes S^{m+1}(B[-2])\big) \to  H^k_{\CE}\big(A; \wedge^l K \otimes S^m(B[-2])\big)},
  \]
  where $[\omega] \in H^2_{\CE}(A; \Hom(B,K))$ is the cohomology class of $A$ introduced in~\cite{GsM}.
\end{thm}
\begin{Rem}\label{Rem:Aone}
We would like to point out how the Rothstein dg algebra $(C^\infty(\E), d_E)$ of $E = A \oplus A^\ast$ is related to the geometry of the regular Lie algebroid $A$. Indeed, the linear connection $\nabla^A$ as in~\eqref{Eq:Def of nablaA} induces an isomorphism of graded vector bundles over the graded manifold $A[1]$
\[
 T^\ast(A[1]) \cong \pi^\ast(A^\ast[-1] \oplus T^\ast M).
\]
Here $\pi^\ast(A^\ast[-1] \oplus T^\ast M)$ denotes the pullback bundle of $A^\ast [-1] \oplus T^\ast M$ via the projection $\pi \colon A[1] \to M$. Thus, we obtain an isomorphism of graded manifolds
\[
 T^\ast[2](A[1]) \cong A[1] \times_{M} (A^\ast[1] \oplus T^\ast [2]M),
\]
and an isomorphism of graded commutative algebras
\[
  C^\infty(T^\ast[2] (A[1])) \cong \Gamma(\wedge^\bullet A^\ast \otimes \wedge^\bullet A \otimes \widehat{S}(T[-2]M) ) \cong C^\infty(\E).
\]
Furthermore, it is easy to see that the above isomorphism sends the Lie derivative $L_{d_A}$ of the Chevalley-Eilenberg differential $d_A$ of the Lie algebroid $A$ to the Hamiltonian vector field $X_\Theta$ on $\E$. Therefore, one gets an isomorphism of commutative dg algebras
\[
  \big(C^\infty(T^\ast[2](A[1])), L_{d_A}\big) \cong (C^\infty(\E), d_E).
\]
Hence, we have recovered the cohomology of the shifted cotangent bundle $T^\ast[2](A[1])$ of the dg manifold $(A[1], d_A)$, which was first obtained by Abad and Crainic via representation up to homotopy of Lie algebroids in~\cite{AC}.
\end{Rem}

\appendix
\section{Proof of Lemma~\ref{Lem:how to construct a contraction}}\label{APP: Proof of contraction lemma}
 We first verify that $(B, d_B:= d_A\mid_B)$ is a subcomplex of $(A,d_A)$. In fact, for all $b \in B = \ker d_Ah \cap \ker hd_A $, we have
 \[
   (d_A h)(d_Bb) = d_A h d_A(b) = 0, \quad \mbox{and} \quad hd_A(d_B b) = h d_A^2(b) = 0,
 \]
 which implies that $d_B b \in B$, and thus $(B, d_B)$ is a subcomplex of $(A,d_A)$.
 We also need to check that $\phi \colon A \to B$ in~\eqref{Eq:Def of sigma} takes values in $B$: Since
 \begin{align*}
     (d_A h)(\phi(a)) &= d_A (h(a)) + (d_A h^2)(d_A(a)) + d_A((hd_A h)(a)) \quad \text{by Eq.~\eqref{Eq:condition on h}} \\
     &= d_A(h(a)) - d_A(h(a)) = 0,
 \end{align*}
 and
 \begin{align*}
   (hd_A)(\phi(a)) &= h(d_A(a)) + (hd_Ah)(d_A(a)) + (hd_A^2h)(a) \quad \text{by Eq.~\eqref{Eq:condition on h}} \\
   &= h(d_A(a)) - h(d_A(a)) = 0,
 \end{align*}
   we have $\phi(a) \in B$ for all $a \in A$, and thus $\phi$ is well-defined.

Now let us examine that the maps $\phi, h$ together with the inclusion $\psi \colon B \hookrightarrow A$ form a contraction of $(A,d_A)$. By the definition~\eqref{Eq:Def of sigma} of $\phi$, we have
\[
  \psi \phi - \id_A =  d_Ah+ hd_A.
\]
 Moreover, by Equation~\eqref{Eq:condition on h}, one has
 \begin{align*}
  &\phi(\phi(a)) = \phi(a + h(d_A(a)) + d_A(h(a))) \\
  &= a + h(d_A(a)) + d_A(h(a)) + (hd_A)(a + h(d_A(a)) + d_A(h(a))) + (d_Ah)(a + h(d_A(a)) + d_A(h(a))) \\
  &= a + h(d_A(a)) + d_A(h(a)) = \phi(a),
 \end{align*}
 and it follows that $\phi$ is indeed a projection. Thus, we have $\phi \circ \psi = \id_B$. Finally, we have for all $a \in A$,
 \[
  \phi(h(a)) = h(a) + (hd_Ah)a) + (d_Ah^2)a = h(a) - h(a) = 0,
 \]
 and
 \[
  (h\psi)(\phi(a)) = h(a) + (h^2d_A)(a) + (hd_Ah)(a) = h(a) - h(a) = 0.
 \]
 Hence, $(\phi, \psi, h)$ is a contraction of $(A,d_A)$ onto $(B,d_B)$.

\section{Proof of Proposition~\ref{Lem:Transgression}}\label{App: Proof of  trangression}
To see that $d_T \in C^3(A_E, B^\ast[2])$ is a cocycle, it suffices to show that
\[
 d_{\CE}(d_T) (b[-2]) = d_T (d_{\CE}(b[-2])) + d_{\CE}(d_T(b[-2])) = 0,
\]
for all $b \in \Gamma(B)$, which indeed follows from Lemma~\ref{Lem:dst on B} and the fact that $d_E^2(b[-2])=0$.

Now we fix a dissection $\Psi$ of $E$ and prove that $[\Transgression] \in H^3_{\CE}(A_E; B^\ast[2])$ is independent of the choice of a triple $(j,\nabla^F,\nabla^B)$.
Since the map $d_T$ defined in Equation~\eqref{Eq: Def of dT} does not depend on the choice of $F$-connection $\nabla^F$, so does $[\Transgression]$.
It remains to show that $[d_T]$ does not depend on the choice of $\nabla^B$ and $j$, which will be dealt with in the subsequent two lemmas, respectively.
\begin{lem}
Suppose that $\nabla^{B\prime}$ is another $B$-connection on $\G$.
Then the $B^\ast[2]$-valued $3$-cocycles $\Transgression$ and $\Transgression^\prime$ defined from $(j,\nabla^B)$ and $(j, \nabla^{B\prime})$, respectively, are cohomologous.
\end{lem}
\begin{proof}
Note that the difference of $\nabla^B$ and $\nabla^{B\prime}$ defines a bilinear bundle map
\[
\phi = \nabla^{B} -\nabla^{B^\prime} \colon \Gamma(B) \otimes \Gamma(\G) \to \Gamma(\G).
\]
Let $\nabla$ and $\nabla^\prime$ be the metric connections on $E \cong F^\ast \oplus F \oplus \G$ arising from $(j,\nabla^F,\nabla^B)$ and $(j,\nabla^F, \nabla^{B\prime})$, respectively.

Define a $B^\ast[2]$-valued $2$-cochain $\psi \in \Gamma(\wedge^2 \G^\ast \otimes B^\ast[2]) \subset C^2(A_E; B^\ast[2])$ of the ample Lie algebroid $A_E$ by
\[
  \psi(r,s)(b[-2]) := \psi_{b}(r,s) := {g^{\G}}(\phi(b,r),s) = -{g^{\G}}(r, \phi(b,s)),
\]
for all $r,s \in \Gamma(\G)$ and $b  \in \Gamma(B )$.
We claim that
\[
d_T - d_T^\prime = d_{\CE}(\psi),
\]
which is equivalent to
\begin{equation}\label{Eq:tobeproved1}
 \Transgression(b[-2]) - \Transgression^\prime(b[-2]) = d_{\CE}(\psi)(b[-2]) =   -\psi_{d_{\CE}(b)} + d_{\CE}(\psi_b),
\end{equation}
for all $b  \in \Gamma(B )$. Here $\psi_{d_{\CE}(b)} \in C^3(A_E)$ is given by for all $x \in \Gamma(F), r,s \in \Gamma(\G)$,
\[
  \psi_{d_{\CE}(b)}(x,r,s) = \psi_{\pr_B([x,j(b)])}(r,s) =  {g^{\G}}(\phi(\pr_B([x,j(b)]),r),s).
\]
Meanwhile, by Proposition~\ref{Prop:CSX 2.1}, the three summands of
\begin{align*}
  d_{\CE}(\psi_b) &\in \Gamma(\wedge^3 \G^\ast) ~\oplus~ \Gamma(\wedge^2 F^\ast \otimes \G^\ast) ~\oplus~ \Gamma(F^\ast \otimes \wedge^2 \G^\ast)
\end{align*}
are defined by the following relations, for all $r,s,t \in \Gamma(\G)$, $x, y \in \Gamma(F)$,
\begin{align*}
 &\quad  d_{\CE}(\psi_b)(r,s,t) = -\psi_b([r,s]^\G, t) + \psi_b([r,t]^\G, s) - \psi_b([s,t]^\G, r) \\
&= -{g^{\G}}(\phi(b,[r,s]^\G), t) + {g^{\G}}(\phi(b,[r,t]^\G), s) - {g^{\G}}(\phi(b, [s,t]^\G), r) \\
&= {g^{\G}}([r,s]^\G, \phi(b,t)) - {g^{\G}}([r,t]^\G, \phi(b,s)) + {g^{\G}}([s,t]^\G, \phi(b,r)) \quad \text{since ${g^{\G}}$ is ad-invariant}\\
&= {g^{\G}}([r,s]^\G, \phi(b,t)) + {g^{\G}}([r,\phi(b,s)]^\G, t) + {g^{\G}}([\phi(b,r),s]^\G, t) \\
&= -(\nabla_{j(b)} C_\nabla - \nabla^\prime_{j(b)} C_{\nabla^\prime}) (r,s,t),
\end{align*}
\begin{align*}
 d_{\CE}(\psi_b)(x,y,r) &=  -\psi_b(R(x,y),r) + \psi_b(\nabla^\G_x r, y) - \psi_b(\nabla^\G_y r, x) \\
&= {g^{\G}}(R(x,y), \phi(b,r)) =  -(\nabla_{j(b)} C_\nabla - \nabla^\prime_{j(b)} C_{\nabla^\prime}) (x,y,r),
\end{align*}
and
\begin{align*}
 d_{\CE}(\psi_b)(x,r,s) &= x \psi_b(r,s) -\psi_b(\nabla^\G_x r, s) + \psi_b(\nabla^\G_x s, r) - \psi_b([r,s]^\G, x) \\
&= x{g^{\G}}(\phi(b,r),s) - {g^{\G}}(\phi(b,\nabla^\G_x r), s) - {g^{\G}}(\nabla^\G_x s, \phi(b,r)) \\
&= x{g^{\G}}(\phi(b,r),s) -{g^{\G}}(\phi(b,\nabla^\G_x r), s) - x{g^{\G}}(s,\phi(b,r)) + {g^{\G}}(\nabla^\G_x \phi(b,r), s) \\
&= -{g^{\G}}((\nabla^B_{b}-\nabla^{B^\prime}_{b})\nabla^\G_x r - \nabla^\G_x (\nabla^B_{b}-\nabla^{B^\prime}_{b})r, s) \\
&= {g^{\G}}(R^\nabla(x,j(b))r + \nabla^B_{\pr_B[x,j(b)]}r - R^{\nabla^\prime}(x,j(b))r - \nabla^{B\prime}_{\pr_B[x,j(b)]}r, s) \\
&= (R^{\nabla}(\rhothree[2],j(b)[-2]) - R^{\nabla^\prime}(\rhothree[2],j(b)[-2]))(x,r,s) + g^\G(\phi(\pr_B[x,j(b)], r), s) \\
&=  (R^{\nabla}(\rhothree[2],j(b)[-2]) - R^{\nabla^\prime}(\rhothree[2],j(b)[-2]) + \psi_{d_{\CE}^B(b[-2])})(x,r,s).
\end{align*}
Combining the above three equations, we obtain the desired Equation~\eqref{Eq:tobeproved1}.
\end{proof}

\begin{lem}
  Suppose that $j^\prime \colon B \to TM$ is another splitting of the short exact sequence~\eqref{Eq:SES of VB}.
  Then the $B^\ast[2]$-valued $3$-cocycles $\Transgression$ and $\Transgression^\prime$ defined from $(j,\nabla^B)$ and $(j^\prime, \nabla^{B})$, respectively, are cohomologous.
\end{lem}
\begin{proof}
Note that the difference of the two splittings defines a bundle map
\[
J := j - j^\prime \colon B \to F.
\]
Consider a $B^\ast[2]$-valued 2-cochain $\alpha = \alpha_1 + \alpha_2 \in C^2(A_E; B^\ast[2])$, where $\alpha_1 \in \Gamma(\wedge^2 F^\ast \otimes B^\ast[2])$ is defined by for all $x, y \in \Gamma(F), b  \in \Gamma(B )$,
\[
 \alpha_1(b[-2])(x,y) := -H(J(b), x, y),
\]
and $\alpha_2 \in \Gamma(F^\ast \otimes \G^\ast \otimes B^\ast[2])$ is defined by for all $x \in \Gamma(F), b  \in \Gamma(B )$, and $r \in \Gamma(\G)$,
\[
 \alpha_2(b[-2])(x,r) := {g^{\G}}(R(x,J(b)),r).
\]
Denote by $\nabla$ and $\nabla^\prime$ the compatible metric connections arising from the triple $(j,\nabla^F,\nabla^B)$ and $(j^\prime, \nabla^F, \nabla^B)$, respectively.
We claim that
\[
  d_T - d_T^\prime = d_{\CE}(\alpha),
\]
which is equivalent to for all $b  \in \Gamma(B )$,
\begin{equation}\label{Eq:j-jprime}
 \Transgression(b[-2]) - \Transgression^\prime(b[-2]) = d_{\CE}(\alpha_1(b[-2])) + d_{\CE}(\alpha_2(b[-2])) - \alpha_1(d_{\CE}(b[-2])) - \alpha_2(d_{\CE}(b[-2])).
\end{equation}
Let us first compute the right-hand side of the above equation.
Note first that $\alpha_1(d_{\CE}(b[-2])) \in \Gamma(\wedge^3 F^\ast)$ and $\alpha_2(d_{\CE}(b[-2])) \in \Gamma(\wedge^2 F^\ast \otimes \G^\ast)$ are given by
\begin{align*}
 \alpha_1(d_{\CE}(b[-2]))(x,y,z) &= -H(J(\nabla^{\operatorname{Bott}}_x b), y, z) - H(x, J(\nabla^{\operatorname{Bott}}_y b), z) - H(x, y, J(\nabla^{\operatorname{Bott}}_z b)) ,
\end{align*}
and
\begin{align*}
  \alpha_2(d_{\CE}(b[-2]))(x,y,r) &= -{g^{\G}}(R(J(\nabla^{\operatorname{Bott}}_xb), y), r)  - {g^{\G}}(R(x,J(\nabla^{\operatorname{Bott}}_yb)), r),
\end{align*}
for all $x,y,z \in \Gamma(F)$ and $r \in \Gamma(\G)$.
According to Proposition~\ref{Prop:CSX 2.1}, the element $d_{\CE}(\alpha_1(b[-2])) \in \Gamma(\wedge^3 F^\ast)$ is given by
\begin{align*}
  d_{\CE}(\alpha_1(b[-2]))(x,y,z) &= -x H(J(b),y,z) + yH(J(b),x,z) - zH(J(b), x, y) \\
                                          &\quad - H([x,y],J(b),z) + H([x,z],J(b),y) - H([y,z],J(b),x),
\end{align*}
for all $x,y,z \in \Gamma(F)$.
Meanwhile, the element $d_{\CE}(\alpha_2(b[-2])) = \beta_1+ \beta_2+\beta_3$, where
\begin{compactenum}
\item $\beta_1 \in \Gamma(\wedge^3 F^\ast)$ is given by
\begin{align*}
  \beta_1(x,y,z) &= \alpha_2(z,R(x,y)) - \alpha_2(y, R(x,z)) + \alpha_2(x, R(y,z)) \\
   &= {g^{\G}}(R(z,J(b)), R(x,y)) - {g^{\G}}(R(y,J(b)), R(x,z)) + {g^{\G}}(R(x,J(b)), R(y,z));
\end{align*}
\item $\beta_2 \in \Gamma(\wedge^2F^\ast \otimes \G^\ast)$ is given by
\begin{align*}
  \beta_2(x,y,r) &= x {g^{\G}}(R(y,J(b)), r) - y {g^{\G}}(R(x,J(b)),r) \\
  &- {g^{\G}}(R([x,y],J(b)),r) - {g^{\G}}(R(y,J(b)), \nabla_x^\G r) + {g^{\G}}(R(x,J(b)), \nabla^\G_y r),
\end{align*}
\item and $\beta_3 \in \Gamma(F^\ast \otimes \wedge^2 \G^\ast)$ is given by
\begin{align*}
  \beta_3(x,r,s) &= \alpha_2(x, [r,s]^\G) = {g^{\G}}(R(x,J(b)), [r,s]^\G),
\end{align*}
for all $x,y \in \Gamma(F)$ and $r,s \in \Gamma(\G)$.
\end{compactenum}
We then compute the left-hand side of Equation~\eqref{Eq:j-jprime}:
Since
\[
 \nabla^\G_x\nabla^\G_{y}r - \nabla^\G_{y}\nabla^\G_x r - \nabla^\G_{[x,y]}r = [R(x,y), r]^\G,
\]
according to the Leibniz-Jacobi identity~\eqref{Eq:Jacobi} for $x,y,r$,    we have
\begin{align}\label{Eq:xrs}
  &\quad (R^\nabla(\rhothree[2], j(b)[-2]) - R^{\nabla^\prime}(\rhothree[2], j^\prime(b)[-2]))(x,r,s) \notag \\
  &= {g^{\G}}(\nabla^\G_x\nabla^\G_{J(b)}r - \nabla^\G_{J(b)}\nabla^\G_x r - \nabla^\G_{[x,J(b)]}r, s) \notag \\
  &= {g^{\G}}([R(x,J(b)), r]^\G, s) \qquad \text{since ${g^{\G}}$ is ad-invariant} \notag \\
  &= {g^{\G}}(R(x,J(b)), [r,s]^\G) = \beta_3(x,r,s).
\end{align}
Since $\pr_B[j(b),x] = \pr_B[j^\prime(b),x] = -\nabla^{\operatorname{Bott}}_x b$ for all $x \in \Gamma(F)$, it follows that
\begin{equation}\label{Eq:nablaforjjprime}
  \nabla_{j(b)}x - \nabla^\prime_{j^\prime(b)}x = \pr_F[j(b),x] - \pr^\prime_F[j^\prime(b),x]  = [J(b),x] + J(\nabla^{\operatorname{Bott}}_x b).
\end{equation}
Thus, we have
\begin{align*}
  &\quad (\nabla_{j(b)}C_\nabla - \nabla^\prime_{j^\prime(b)}C_{\nabla^\prime})(x,y,r) \\
  &= J(b){g^{\G}}(R(x,y),r) - {g^{\G}}(R([J(b),x],y),r) - {g^{\G}}(R(x, [J(b),y]),r) - {g^{\G}}(R(x,y), \nabla^\G_{J(b)}r) \\
  &\quad - g^{\G}(R(J(\nabla^{\operatorname{Bott}}_x b), y), r) - g^{\G}(R(x, J(\nabla^{\operatorname{Bott}}_y b)), r).
\end{align*}
Applying the  identity~\eqref{Eq:Jacobi} to $x,y,z \in \Gamma(F)$, we have
\begin{align*}
   \nabla_x^\G R(y,z) + \nabla_y^\G R(z,x) + \nabla_z^\G R(x,y) &= R([x,y],z)) + R([y,z], x) + R([z,x], y).
\end{align*}
Combining with the fact that $\nabla^\G$ is compatible with the metric ${g^{\G}}$, we obtain
\begin{align}\label{Eq:xyr}
   -(\nabla_{j(b)}C_\nabla - \nabla^\prime_{j^\prime(b)}C_{\nabla^\prime})(x,y,r) &= (\beta_2-  \alpha_2(d_{\CE}^{B}(b[-2])))(x,y,r).
\end{align}
Finally, using Equation~\eqref{Eq:nablaforjjprime}, we have
\begin{align*}
  &\quad (\nabla_{j(b)}C_\nabla - \nabla^\prime_{j^\prime(b)}C_{\nabla^\prime})(x,y,z) \\
  &= J(b)H(x,y,z) - H([J(b),x],y,z) - H(x,[J(b),y],z) - H(x,y,[J(b),z]) \\
  &\quad - H(J(\nabla^{\operatorname{Bott}}_x b), y, z) - H(x,J(\nabla^{\operatorname{Bott}}_y b),z) - H(x,y,J(\nabla^{\operatorname{Bott}}_z b)).
\end{align*}
Thus, we have
\begin{align*}
&\quad (\nabla_{j(b)}C_\nabla - \nabla^\prime_{j^\prime(b)}C_{\nabla^\prime})(x,y,z) + d_{\CE}(\alpha_1(b[-2]))(x,y,z) \\
&= (d_{\CE}H)(J(b),x,y,z) + \alpha_1(d_{\CE}^{B}(b[-2]))(x,y,z) \qquad\qquad\qquad\qquad  \text{by Eq.~\eqref{Eq:Jacobi}}\\
&= {g^{\G}}(R(J(b),x), R(y,z)) - {g^{\G}}(R(J(b),y), R(x,z)) \\
&\qquad + {g^{\G}}(R(J(b),z), R(x,y)) + \alpha_1(d_{\CE}^{B}(b[-2]))(x,y,z)  \\
&= -\beta_1(x,y,z) + \alpha_1(d_{\CE}^{B}(b[-2]))(x,y,z) ,
\end{align*}
which implies that
\begin{equation}\label{Eq:xyz}
  -(\nabla_{j(b)}C_\nabla - \nabla^\prime_{j^\prime(b)}C_{\nabla^\prime})(x,y,z) = (d_{\CE}(\alpha_1(b[-2])) + \beta_1 - \alpha_1(d_{\CE}^{B}(b[-2])))(x,y,z).
\end{equation}
Combining Equations~\eqref{Eq:xrs},\eqref{Eq:xyr}, and~\eqref{Eq:xyz}, we obtain Equation~\eqref{Eq:j-jprime} as desired.
\end{proof}

Finally, we prove that $[d_T] \in H^3_{\CE}(A_E; B^\ast[2])$ is independent of the choice of a dissection of $E$. Suppose that $\hat{\Psi}$ is another dissection of $E$ under assumptions in Proposition~\ref{Prop: different dissection}. By Proposition~\ref{Prop:CSX 2.1}, the ample Lie algebroid $A_E$ is identified with $L(\nabla^\G, R)$ and $L(\hat{\nabla}^\G, \hat{R})$ via the two dissections $\Psi$ and $\hat{\Psi}$, respectively.
By Proposition~\ref{Prop: different dissection}, the isomorphism between these two Lie algebroids is given by
\[
 \delta_L \colon L(\nabla^\G, R) \to  L(\hat{\nabla}^\G, \hat{R}), \quad \delta_L(r+x) = (\tau(r)+\tau(\varphi(x))) + x,
\]
for all $r \in \Gamma(\G)$ and $x \in \Gamma(F)$ for some automorphism $\tau$ of the bundle $\G$ of quadratic Lie algebras and some bundle map $\varphi \colon F \to \G$.

Let us fix a triple $(j,\nabla^F, \nabla^B)$ for the dissection $\Psi$. For the new dissection $\widehat{\Psi}$, we choose the triple $(j,\nabla^F, \widehat{\nabla}^B)$ with the same splitting $j$ and torsion-free $F$-connection $\nabla^F$, while the metric $B$-connection $\widehat{\nabla}^B$ on $\G$ is the twisted one from $\nabla^B$ by $\tau$ that is defined by
\begin{equation}\label{Eq: nablahatB}
 \widehat{\nabla}^B_b r := \tau(\nabla^B_b \tau^{-1}(r)),
\end{equation}
for all $b \in \Gamma(B)$ and $r \in \Gamma(\G)$. The metric $B$-connection $\widehat{\nabla}^B$ is well-defined since $\tau$ is an automorphism of bundle of quadratic Lie algebras.
Let $\nabla$ (resp. $\widehat{\nabla}$) be the $\Psi$ (resp. $\widehat{\Psi}$)-metric connection on $F^\ast \oplus \G \oplus F$ as in Lemma~\ref{Lem:F-connection on E}.
And denote the corresponding Chevalley-Eilenberg cochains in $C^3(A_E; B^\ast[2])$ by $d_T$ and $\widehat{d_T}$, respectively.
We conclude the proof by the following fact.
\begin{lem}
   We have
   \[
    \delta_L^\ast(\widehat{d_T}) - d_T = d_{\CE}(\gamma),
   \]
   where $\gamma \in C^2(A_E; B^\ast[2]) \cong \Hom_{C^\infty(M)}(\Gamma(B[-2] \otimes S^2(\G[1] \oplus F[1])), C^\infty(M))$ is defined by
   \begin{align*}
     &\quad \gamma(b[-2], r[1]+x[1], s[1]+y[1]) \\
     &:= g^\G\left(r+\frac{1}{2}\varphi(x), (\nabla^B_b\varphi)(y) \right) - g^\G\left(r+\frac{1}{2}\varphi(y), (\nabla^B_b\varphi)(x) \right) + (\nabla^B_b\beta)(x,y)\\
     &= g^\G\left(r + \frac{1}{2}\varphi(x), \nabla_b^B \varphi(y) - \varphi(\nabla_b^B y)\right) - g^\G\left(s + \frac{1}{2}\varphi(y), \nabla_b^B \varphi(x) - \varphi(\nabla^B_bx)\right) \\
     &\quad + j(b)\beta(x,y) - \beta(\nabla^B_bx, y) - \beta(x, \nabla^B_by),
   \end{align*}
   for all $b \in \Gamma(B), r,s \in \Gamma(\G)$ and all $x,y \in \Gamma(F)$. Here $\beta \in \Gamma(\wedge^2 F^\ast)$ is the element induced from the bundle map $\beta \colon F \to F^\ast$ in the assumption of Proposition~\ref{Prop: different dissection}.
\end{lem}
\begin{proof}
  We need the following relations in~\cite{CSX}*{Proposition 2.7} (see also~\cite{MR3661534}*{Theorem 2.23}).
  \begin{align}
    \widehat{\nabla}^\G_x r &= \tau \nabla^\G_x \tau^{-1}(r) + [\tau(\varphi(x)), r]^\G, \label{Eq: hatG}\\
    \widehat{R}(x,y)  &= \tau(R(x,y)) - \tau((d_{\CE}\varphi)(x,y)) + \tau([\varphi(x), \varphi(y)]^\G), \label{Eq:hatR}\\
    \frac{1}{2}\widehat{H}(x,y,z) &= \frac{1}{2}H(x,y,z) - (d\beta)(x,y,z) - g^\G(R(x,y),\varphi(z))
    + \frac{1}{2}g^\G(\varphi(x), (d_{\CE}\varphi)(y,z)) \notag \\
    &\quad - g^\G([\varphi(x),\varphi(y)]^\G, \varphi(z)) + c.p., \label{Eq:hatH}
  \end{align}
  for all $x,y,z \in \Gamma(F)$ and $r \in \Gamma(\G)$.
The proof consists of several case by case verifications:

  $(1)$ For all $b \in \Gamma(B)$ and  $r,s,t \in \Gamma(\G)$, we have
  \begin{align*}
    &\quad (\delta_L^\ast \widehat{d_T}(b[-2]))(r[1],s[1],t[1]) = \widehat{d_T}(b[-2])(\tau(r)[1],\tau(s)[1],\tau(t)[1]) = -(\widehat{\nabla}_bC_{\widehat{\nabla}})(\tau(r),\tau(s),\tau(t)) \\
    &= -j(b)C^\G(\tau(r),\tau(s),\tau(t)) + C^\G(\widehat{\nabla}^B_b \tau(r), \tau(s), \tau(t)) +  C^\G(\tau(r), \widehat{\nabla}^B_b\tau(s), \tau(t)) \\
    &\quad +  C^\G(\tau(r),\tau(s), \widehat{\nabla}^B_b\tau(t)) \\
    &= -j(b)C^\G(\tau(r),\tau(s),\tau(t)) + C^\G(\tau(\nabla^B_b r), \tau(s), \tau(t)) +  C^\G(\tau(r), \tau(\nabla^B_bs), \tau(t)) \\
    &\quad  + C^\G(\tau(r),\tau(s), \tau(\nabla^B_bt)) \\
     &= -j(b)C^\G(r,s,t) + C^\G(\nabla^B_b r, s, t) +  C^\G(r,\nabla^B_bs, t)+  C^\G(r,s,\nabla^B_bt) \\
     &= -(\nabla_bC^\G)(r,s,t) = d_T(b[-2])(r[1],s[1],t[1]).
  \end{align*}
  Here we have used the fact that $\tau^\ast C^\G = C^\G$ since $\tau$ is an automorphism of bundle of quadratic Lie algebras.

 $(2)$ For all $b \in \Gamma(B), x \in \Gamma(F)$ and $r,s \in \Gamma(\G)$, we have
 \begin{align*}
   (\delta_L^\ast \widehat{d_T}(b[-2]))(x[1],r[1],s[1]) &= \widehat{d_T}(b[-2])(x[1] + \tau(\varphi(x))[1],\tau(r)[1],\tau(s)[1]) \\
   &= g^\G(R^{\widehat{\nabla}}(x,b)\tau(r), \tau(s)) - (\widehat{\nabla}_b C^\G)(\tau(\varphi(x)),\tau(r),\tau(s)).
 \end{align*}
 Keeping in mind that $\tau^\ast g^\G = g^\G$ and that $\tau$ preserves the bracket $[-,-]^\G$, we have
 \begin{align}
     &\quad g^\G(R^{\widehat{\nabla}}(x,b)\tau(r),\tau(s)) \notag \\
     &= g^\G(\widehat{\nabla}^\G_x\widehat{\nabla}^B_b\tau(r) - \widehat{\nabla}^B_b \widehat{\nabla}^\G_x \tau(r) - \widehat{\nabla}_{[x,b]} \tau(r), \tau(s)) \qquad  \text{by Eqs.~\eqref{Eq: nablahatB} and~\eqref{Eq: hatG}} \notag \\
    &= g^\G(R^\nabla(x,b)r, s) - g^\G([\varphi(x),\nabla^B_b r]^\G -\nabla^B_b[\varphi(x), r]^\G + [\varphi(\nabla^B_b x), r]^\G , s), \label{Eq: Rxbr}
 \end{align}
 and
 \begin{align*}
   &\quad (\widehat{\nabla}_b C^\G)(\tau(\varphi(x)),\tau(r),\tau(s)) = (\nabla_b C^\G)(\varphi(x), r,s) \\
   &= j(b)g^\G([\varphi(x),r]^\G, s) - g^\G([\nabla_b^B\varphi(x), r]^\G-  [\varphi(x), \nabla^B_b r], s) - g^\G([\varphi(x),r]^\G, \nabla^B_b s) \\
 &= g^\G(\nabla^B_b[\varphi(x),r]^\G - [\nabla_b^B \varphi(x), r]^\G - [\varphi(x),\nabla^B_b r]^\G, s).
 \end{align*}
 Thus, we obtain
 \begin{align*}
   (\delta_L^\ast \widehat{d_T}(b[-2]))(x[1],r[1],s[1]) &= g^\G(R^\nabla(x,b)r, s) + g^\G([\nabla_b^B \varphi(x) - \varphi(\nabla^B_b x), r]^\G, s) \\
   &= d_T(b[-2])(x[1],r[1], s[1]) + (d_{\CE}\gamma)(b[-2], x[1], r[1], s[1]).
 \end{align*}

 $(3)$ For all $b \in \Gamma(B), x,y \in \Gamma(F)$ and $r \in \Gamma(\G)$, we have
 \begin{align}\label{Eq: 3case}
   &\quad (\delta_L^\ast \widehat{d_T}(b[-2]))(x[1], y[1], r[1]) = \widehat{d_T}(b[-2])(x[1] + \tau(\varphi(x))[1], y[1] + \tau(\varphi(y))[1], r[1]) \notag \\
   &= g^\G(R^{\widehat{\nabla}}(x,b)\tau(\varphi(y)) - R^{\widehat{\nabla}}(y,b)\tau(\varphi(x)), \tau(r)) - (\widehat{\nabla}_bC_{\widehat{\nabla}})(x+\tau(\varphi(x)), y+\tau(\varphi(y)), \tau(r)).
 \end{align}
 Using Equation~\eqref{Eq: Rxbr}, we have
 \begin{align}\label{Eq: gRanti}
   &\quad g^\G(R^{\widehat{\nabla}}(x,b)\tau(\varphi(y)) - R^{\widehat{\nabla}}(y,b)\tau(\varphi(x)), \tau(r)) = g^\G(R^\nabla(x,b)\varphi(y) - R^{\nabla}(y,b)\varphi(x), r) \notag \\
   &\qquad + g^\G(2\nabla^B_b[\varphi(x),\varphi(y)]^\G - [\varphi(x), \nabla^B_b\varphi(y)+\varphi(\nabla^B_b y)]^\G +  [\varphi(y), \nabla^B_b\varphi(x)+\varphi(\nabla^B_b x)]^\G, r).
 \end{align}
 The computation in the step $(1)$ shows that
  \begin{align}\label{Eq: nablahatxyr}
    &\quad (\widehat{\nabla}_bC_{\widehat{\nabla}})(\tau(\varphi(x)), \tau(\varphi(y)), \tau(r)) = (\nabla_bC^\G)(\varphi(x),\varphi(y),r) \notag \\
    &= g^\G(\nabla^B_b[\varphi(x),\varphi(y)]^\G, r) - g^\G([\nabla^B_b\varphi(x),\varphi(y)]^\G, r) - g^\G([\varphi(x),\nabla^B_b\varphi(y)]^\G, r).
 \end{align}
 Meanwhile, using Equation~\eqref{Eq:hatR} and the fact that $g^\G$ is $\tau$-invariant, we have
 \[
  g^\G(\widehat{R}(x,y), \tau(r)) = g^\G(R(x,y) - (d_{\CE}\varphi)(x,y) + [\varphi(x),\varphi(y)]^\G, r).
 \]
 Thus,
 \begin{align}
   &\quad (\widehat{\nabla}_bC_{\widehat{\nabla}})(x, y,\tau(r)) \notag\\
   &= j(b) g^\G(\widehat{R}(x,y),\tau(r)) - g^\G(\widehat{R}(x,y), \tau(\nabla^B_br))  - g^\G(\widehat{R}(\nabla_b^B x, y)+ \widehat{R}(x, \nabla^B_b y), \tau(r)) \notag \\
   &= (\nabla_bC_\nabla)(x,y,r) - g^\G(\nabla^B_b(d_{\CE}\varphi)(x,y), r) \notag \\
    &\quad + g^\G(\nabla_b^B[\varphi(x),\varphi(y)]^\G - [\varphi(\nabla^B_bx), \varphi(y)]^\G - [\varphi(x), \varphi(\nabla^B_b y)]^\G, r).\label{Eq: nablaCnabla xyr}
 \end{align}
 By a direct computation, we also have
 \begin{equation}\label{Eq: dnablaphi}
   R^\nabla(x,b)\varphi(y) - R^\nabla(y,b)\varphi(x) = d_{\CE}(\nabla^B_b \varphi)(x,y) - \nabla_b^B(d_{\CE}\varphi)(x,y).
 \end{equation}
Applying Equations~\eqref{Eq: gRanti},\eqref{Eq: nablahatxyr},\eqref{Eq: nablaCnabla xyr} and~\eqref{Eq: dnablaphi} to Equation~\eqref{Eq: 3case}, we obtain
 \begin{align*}
 (\delta_L^\ast \widehat{d_T}(b[-2]))(x[1], y[1], r[1]) &= d_T(b[-2])(x[1],y[1],r[1]) + g^\G(r, d_{\CE}( \nabla^B_b \varphi)(x,y)) \\
    &= d_T(b[-2])(x[1],y[1],r[1]) + d_{\CE}(\gamma)(b[-2], x[1], y[1], r[1]).
 \end{align*}

$(4)$ For all $b \in \Gamma(B), x,y,z \in \Gamma(F)$, we have
 \begin{align}\label{Eq: 4case}
  &\quad  (\delta_L^\ast \widehat{d_T}(b[-2]))(x[1],y[1],z[1]) = \widehat{d_T}(b[-2])(x[1] + \tau(\varphi(x))[1], y[1] + \tau(\varphi(y))[1], z[1] + \tau(\varphi(z))[1]) \notag \\
  &= -(\widehat{\nabla}_b C^\G)(\tau(\varphi(x)),\tau(\varphi(y)),\tau(\varphi(z))) -  \frac{1}{2}(\widehat{\nabla}_b \widehat{H})(x,y,z) \notag \\
  &\quad +\frac{1}{2}g^\G(R^{\widehat{\nabla}}(x,b)\tau(\varphi(y)) - R^{\widehat{\nabla}}(y,b)\tau(\varphi(x)), \tau(\varphi(z))) -  (\widehat{\nabla}_b C_{\widehat{\nabla}})(x,y,\tau(\varphi(z))) + c.p..
 \end{align}
 The computation in step $(1)$ shows that
 \begin{align}\label{Eq: phixyz}
   &\quad (\widehat{\nabla}_b C^\G)(\tau(\varphi(x)),\tau(\varphi(y)),\tau(\varphi(z))) \notag \\
   &= g^\G(\nabla^B_b[\varphi(x),\varphi(y)]^\G - [\nabla^B_b\varphi(x),\varphi(y)]^\G -  [\varphi(x),\nabla^B_b\varphi(y)]^\G, \varphi(z)).
 \end{align}
 Using Equation~\eqref{Eq:hatH}, we have
 \begin{align}\label{Eq:nablahatH}
   &\quad \frac{1}{2}(\widehat{\nabla}_b \widehat{H})(x,y,z) \notag \\
   &= \frac{1}{2}(\nabla_b H)(x,y,z) - \nabla^B_b(d_{\CE}\beta)(x,y,z) - (\nabla_bC_\nabla)(x,y,\varphi(z)) - g^\G(R(x,y), (\nabla^B_b \varphi)(z)) - c.p.\notag\\
   &\; - g^\G(\nabla^B_b[\varphi(x),\varphi(y)]^\G - [\varphi(\nabla^B_bx), \varphi(y)]^\G - [\varphi(x), \varphi(\nabla^B_b y)]^\G, \varphi(z)) - g^\G([\varphi(x),\varphi(y)]^\G, (\nabla^B_b \varphi)(z)) \notag \\
   &\; +\frac{1}{2}\left(j(b)g^\G((d_{\CE}\varphi)(x,y), \varphi(z)) - g^\G((d_{\CE}\varphi)(x,y), \varphi(\nabla^B_bz))\right) \notag \\
   &\;-\frac{1}{2}\left(g^\G((d_{\CE}\varphi)(\nabla_b^B x, y), \varphi(z)) + g^\G((d_{\CE}\varphi)(x, \nabla^B_b y), \varphi(z)) \right) + c.p..
 \end{align}
 By Equations~\eqref{Eq: gRanti} and~\eqref{Eq: dnablaphi}, we have
 \begin{align}
   &\quad g^\G(R^{\widehat{\nabla}}(x,b)\tau(\varphi(y)) - R^{\widehat{\nabla}}(y,b)\tau(\varphi(x)), \tau(\varphi(z))) \notag \\
   &= g^\G(2\nabla^B_b[\varphi(x),\varphi(y)]^\G - [\varphi(x), \nabla^B_b\varphi(y)+\varphi(\nabla^B_b y)]^\G + [\varphi(y), \nabla^B_b\varphi(x)+\varphi(\nabla^B_b x)]^\G, \varphi(z)) \notag \\ \label{Eq: Rxyz}
   &\quad + g^\G(d_{\CE}(\nabla^B_b \varphi)(x,y) - \nabla_b^B(d_{\CE}\varphi)(x,y), \varphi(z)).
 \end{align}
 By Equation~\eqref{Eq: nablaCnabla xyr}, we also have
 \begin{align}
   &\quad  (\widehat{\nabla}_b C_{\widehat{\nabla}})(x,y,\tau(\varphi(z))) \notag \\
   &= (\nabla_bC_\nabla)(x,y, \varphi(z)) - g^\G(\nabla^B_b(d_{\CE}\varphi)(x,y), \varphi(z)) \notag \\
    &\quad + g^\G(\nabla_b^B[\varphi(x),\varphi(y)]^\G - [\varphi(\nabla^B_bx), \varphi(y)]^\G - [\varphi(x), \varphi(\nabla^B_b y)]^\G, \varphi(z)). \label{Eq: nablaCnabla xyphiz}
 \end{align}
 Substituting Equations~\eqref{Eq: phixyz},\eqref{Eq:nablahatH}, \eqref{Eq:    Rxyz} and \eqref{Eq: nablaCnabla xyphiz} into Equation~\eqref{Eq: 4case}, we obtain
 \begin{align*}
     (\delta_L^\ast &\widehat{d_T}(b[-2]))(x[1],y[1],z[1]) = -\frac{1}{2}(\nabla_bH)(x,y,z) + \nabla^B_b(d_{\CE}\beta)(x,y,z) \\
     &\; + g^\G\left(R(x,y)-\frac{1}{2}(d_{\CE}\varphi)(x,y), (\nabla^B_b\varphi)z\right) + g^\G\left( d_{\CE}(\nabla^B_b\varphi)(x,y), \frac{1}{2}\varphi(z)\right) + c.p. \\
   &= d_T(b[-2])(x[1],y[1],z[1]) + (d_{\CE}\gamma) (b[-2],x[1],y[1],z[1]).
 \end{align*}
 Here we have used the fact that $\nabla^B_b(d_{\CE}\beta)(x,y,z) = d_{\CE}(\nabla^B_b\beta)(x,y,z)$, which follows from a straightforward computation.
\end{proof}

\end{document}